\documentclass[11pt, reqno]{amsart}
\usepackage[utf8]{inputenc}
\usepackage{amssymb}
\usepackage{bbm}
\usepackage{amsthm}
\usepackage{nccmath}
\usepackage[bookmarks=false]{hyperref}
\usepackage{etoolbox}
\usepackage{array}
\usepackage{xparse}
\usepackage{enumitem}
\usepackage{tikz}
\usepackage{mathtools}
\usepackage[font=scriptsize]{caption}
\calclayout

\newtheorem{thm}{Theorem}[section]
\newtheorem{theorem}[thm]{Theorem}
\newtheorem{lem}[thm]{Lemma}
\newtheorem{lemma}[thm]{Lemma}
\newtheorem{prop}[thm]{Proposition}
\newtheorem{cor}[thm]{Corollary}

\theoremstyle{definition}
\newtheorem{definition}[thm]{Definition}

\theoremstyle{remark}
\newtheorem{rem}[thm]{Remark}
\newtheorem{rmk}[thm]{Remark}

\newcommand{\thmref}[1]{Theorem~\ref{#1}}
\newcommand{\corref}[1]{Corollary~\ref{#1}}
\newcommand{\defref}[1]{Definition~\ref{#1}}
\newcommand{\secref}[1]{\S\ref{#1}}
\newcommand{\lemref}[1]{Lemma~\ref{#1}} 
\newcommand{\propref}[1]{Proposition~\ref{#1}}
\newcommand{\remref}[1]{Remark~\ref{#1}}

\newcommand*{\dis}{\displaystyle}
\newcommand*{\q}{\quad}
\newcommand*{\qq}{\qquad}
\newcommand*{\lpar}{ }

\newcommand*{\tx}[1]{\text{#1}}

\newcommand*{\del}{\delta}

\newcommand*{\ep}{\epsilon}

\newcommand*{\suchthat}{\, \middle| \,}

\newcommand*{\Ra}{\Rightarrow}
\newcommand*{\M}{M}
\newcommand*{\Mcal}{\mathcal{M}}

\newcommand*{\Pminus}{P_{-}}
\newcommand*{\Pminusbar}{\myoverline{-3}{0}{P}_{-}}

\newcommand*{\myoverline}[3]{\mkern -#1mu\overline{\mkern#1mu#3\mkern#2mu}\mkern -#2mu}	

\newcommand*{\Zbar}{\myoverline{-3}{0}{\Z}}
\newcommand*{\zbar}{\myoverline{-2}{0}{\z}}

\newcommand*{\Dapbar}{\myoverline{-3}{-1}{D}_\ap}
\newcommand*{\wbar}{\myoverline{0}{-0.5}{\omega}}
\newcommand*{\Fbar}{\myoverline{-3}{0}{F} }

\newcommand*{\ubar}{\myoverline{0}{0}{u}}

\newcommand*{\half}{\frac{1}{2}}

\newcommand*{\Rsp}{\mathbb{R}}
\newcommand*{\Csp}{\mathbb{C}}
\newcommand*{\Nsp}{\mathbb{N}}
\newcommand*{\Zsp}{\mathbb{Z}}

\newcommand*{\Scalsp}{\mathcal{S}}

\newcommand*{\Ltwo}{L^2}

\newcommand*{\Linfty}{L^{\infty}}
\newcommand*{\Hhalf}{\dot{H}^\half}

\newcommand*{\al}{\alpha}
\newcommand*{\ap}{{\alpha'}}

\newcommand*{\be}{\beta}
\newcommand*{\bp}{{\beta'}}

\newcommand*{\xp}{{x'}}
\newcommand*{\yp}{{y'}}
\newcommand*{\zp}{{z'}}

\newcommand*{\eqdef}{\mathrel{\mathop:}=}
\newcommand*{\diff}{\mathop{}\! d}

\newcommand*{\compose}[1]{\circ{#1}}

\newcommand*{\conv}{*}
\newcommand*{\Hil}{\mathbb{H}}
\newcommand*{\Hcal}{\mathcal{H}}
\newcommand*{\Hcaltil}{\widetilde{\mathcal{H}}}
\newcommand*{\Htil}{\Hcaltil}

\newcommand*{\Id}{\mathbb{I}}
\newcommand*{\Imag}{\tx{Im}}
\newcommand*{\Real}{\tx{Re}}
\newcommand*{\Jdel}{J_\delta}

\newcommand*{\grad}{\nabla}
\newcommand*{\Dt}{D_t}

\newcommand*{\Dtep}{\Dt^{\ep}}

\newcommand*{\pt}{\partial_t}

\newcommand*{\px}{\partial_x}
\newcommand*{\py}{\partial_y}

\newcommand*{\pxp}{\partial_\xp}
\newcommand*{\pyp}{\partial_\yp}

\newcommand*{\pap}{\partial_\ap}
\newcommand*{\papabs}{\abs{\pap}}
\newcommand*{\pbp}{\partial_\bp}
\newcommand*{\pal}{\partial_\al}

\newcommand*{\Dap}{D_{\ap}}

\newcommand*{\Dapabs}{\abs{D_{\ap}}}
\newcommand*{\Dapfrac}{\frac{1}{\Zap}\pap}
\newcommand*{\Dapbarfrac}{\frac{1}{\Zapbar}\pap}

\newcommand*{\Dapabsfrac}{\frac{1}{\Zapabs}\pap}

\newcommand*{\Ecal}{\mathcal{E}}

\newcommand*{\EDelta}{E_{\Delta}}

\newcommand*{\Fcal}{\mathcal{F}}

\newcommand*{\Pfrak}{\mathfrak{P}}

\newcommand*{\A}{\mathcal{A}}
\newcommand*{\Aone}{A_1}
\newcommand*{\Bone}{B_1}

\newcommand*{\bvar}{b}

\newcommand*{\bap}{\bvar_\ap}
\newcommand*{\bvarap}{\bap}
\newcommand*{\bstar}{b^*}

\newcommand*{\h}{h}
\newcommand*{\hep}{\h^\ep}

\newcommand*{\hal}{\h_\al}
\newcommand*{\halep}{\hal^\ep}
\newcommand*{\htal}{h_{t\al}}
\newcommand*{\hinv}{\h^{-1}}
\newcommand*{\htil}{\widetilde{\h}}
\newcommand*{\htilap}{\htil_\ap}
\newcommand*{\htilbp}{\htil_\bp}

\newcommand*{\Util}{\widetilde{U}}

\newcommand*{\g}{g}

\newcommand*{\cvar}{c}

\newcommand*{\f}{f}

\newcommand*{\w}{\omega}

\newcommand*{\Psiep}{\Psi^\ep}

\newcommand*{\Psizp}{\Psi_{\zp}}
\newcommand*{\Psizpep}{\Psi_{\zp}^\ep}
\newcommand*{\onePsizp}{\frac{1}{\Psizp}}
\newcommand*{\onePsizpep}{\frac{1}{\Psizpep}}

\newcommand*{\Psit}{\Psi_{t}}
\newcommand*{\Psitep}{\Psit^\ep}

\newcommand*{\z}{z}
\newcommand*{\zep}{\z^\ep}
\newcommand*{\zal}{\z_\al}
\newcommand*{\zalep}{\zal^\ep}
\newcommand*{\zalbar}{\zbar_\al}

\newcommand*{\zt}{\z_t}
\newcommand*{\ztep}{\zt^\ep}

\newcommand*{\ztal}{\z_{t\al}}

\newcommand*{\ztalbar}{\zbar_{t\al}}

\newcommand*{\ztt}{\z_{tt}}

\newcommand*{\zttbar}{\zbar_{tt}}

\newcommand*{\Z}{Z}
\newcommand*{\Zep}{\Z^\ep}
\newcommand*{\Zap}{\Z_{,\ap}}

\newcommand*{\Zapep}{\Zap^\ep}
\newcommand*{\oneZapep}{\frac{1}{\Zapep}}

\newcommand*{\Zapbar}{\Zbar_{,\ap}}
\newcommand*{\Zapepbar}{\Zapbar^\ep}

\newcommand*{\Zapabs}{\abs{\Zap}}

\newcommand*{\Zt}{\Z_t}
\newcommand*{\Ztep}{\Zt^\ep}
\newcommand*{\Ztbar}{\Zbar_t}

\newcommand*{\Ztap}{\Z_{t,\ap}}
\newcommand*{\Ztapep}{\Ztap^\ep}
\newcommand*{\Ztapbar}{\Zbar_{t,\ap}}

\newcommand*{\Ztt}{\Z_{tt}}

\newcommand*{\Zttbar}{\Zbar_{tt}}

\newcommand*{\nobrac}[1]{ #1 }

\DeclarePairedDelimiter{\oldbrac}{\lparen}{\rparen}			
\NewDocumentCommand{\brac}{ s o m }{						
	\IfBooleanT{#1}{
  		\IfValueT{#2}{\oldbrac[#2]{#3}}
		\IfValueF{#2}{\oldbrac{#3}} 
	}
	\IfBooleanF{#1}{
  		\IfValueT{#2}{\PackageError{mypackage}{Incorrect use of brac. Insert star}{}}
		\IfValueF{#2}{\oldbrac*{#3}} 
	}		
}

\DeclarePairedDelimiter\oldcbrac{\lbrace}{\rbrace}				
\NewDocumentCommand{\cbrac}{ s o m }{					
	\IfBooleanT{#1}{
  		\IfValueT{#2}{\oldcbrac[#2]{#3}}
		\IfValueF{#2}{\oldcbrac{#3}} 
	}
	\IfBooleanF{#1}{
  		\IfValueT{#2}{\PackageError{mypackage}{Incorrect use of cbrac. Insert star}{}}
		\IfValueF{#2}{\oldcbrac*{#3}} 
	}		
}

\DeclarePairedDelimiter\oldsqbrac{\lbrack}{\rbrack}				
\NewDocumentCommand{\sqbrac}{ s o m }{					
	\IfBooleanT{#1}{
  		\IfValueT{#2}{\oldsqbrac[#2]{#3}}
		\IfValueF{#2}{\oldsqbrac{#3}} 
	}
	\IfBooleanF{#1}{
  		\IfValueT{#2}{\PackageError{mypackage}{Incorrect use of sqbrac. Insert star}{}}
		\IfValueF{#2}{\oldsqbrac*{#3}} 
	}		
}

\DeclarePairedDelimiter{\oldabs}{\lvert}{\rvert}
\NewDocumentCommand{\abs}{ s o m }{						
	\IfBooleanT{#1}{
  		\IfValueT{#2}{\oldabs[#2]{#3}}
		\IfValueF{#2}{\oldabs{#3}} 
	}
	\IfBooleanF{#1}{
  		\IfValueT{#2}{\PackageError{mypackage}{Incorrect use of abs. Insert star}{}}
		\IfValueF{#2}{\oldabs*{#3}} 
	}		
}

\DeclarePairedDelimiterX{\oldnorm}[1]{\lVert}{\rVert}{#1}
\NewDocumentCommand{\norm}{ s o o m }{					
	\IfValueT{#2} {
		\IfBooleanT{#1}{
  			\IfValueT{#3}{\oldnorm[#2]{#4}_{#3}}
			\IfValueF{#3}{\oldnorm{#4}_{#2}} 
		}
		\IfBooleanF{#1}{
  			\IfValueT{#3}{\PackageError{mypackage}{Incorrect use of norm. Insert star}{}}
			\IfValueF{#3}{\oldnorm*{#4}_{#2}} 
		}
	}
	\IfValueF{#2} {
		\IfBooleanT{#1}{\oldnorm{#4}}	
		\IfBooleanF{#1}{\oldnorm*{#4}}		
	}	
}

\makeatletter
\def\black@#1{%
    \noalign{%
        \ifdim#1>\displaywidth
            \dimen@\prevdepth
            \nointerlineskip
            \vskip-\ht\strutbox@
            \vskip-\dp\strutbox@
            \vbox{\noindent\hbox to \displaywidth{\hbox to#1{\strut@\hfill}}}%
            \prevdepth\dimen@
        \fi
    }%
}
\makeatother

\makeatletter
\renewcommand{\tocsection}[3]{%
  \indentlabel{\@ifnotempty{#2}{\bfseries\ignorespaces#1 #2\quad}}\bfseries#3}
\renewcommand{\tocsubsection}[3]{%
  \indentlabel{\@ifnotempty{#2}{\ignorespaces#1 #2\quad}}#3}

\newcommand\@dotsep{4.5}
\def\@tocline#1#2#3#4#5#6#7{\relax
  \ifnum #1>\c@tocdepth 
  \else
    \par \addpenalty\@secpenalty\addvspace{#2}%
    \begingroup \hyphenpenalty\@M
    \@ifempty{#4}{%
      \@tempdima\csname r@tocindent\number#1\endcsname\relax
    }{%
      \@tempdima#4\relax
    }%
    \parindent\z@ \leftskip#3\relax \advance\leftskip\@tempdima\relax
    \rightskip\@pnumwidth plus1em \parfillskip-\@pnumwidth
    #5\leavevmode\hskip-\@tempdima{#6}\nobreak
    \leaders\hbox{$\m@th\mkern \@dotsep mu\hbox{.}\mkern \@dotsep mu$}\hfill
    \nobreak
    \hbox to\@pnumwidth{\@tocpagenum{\ifnum#1=1\bfseries\fi#7}}\par
    \nobreak
    \endgroup
  \fi}
\AtBeginDocument{%
\expandafter\renewcommand\csname r@tocindent0\endcsname{0pt}
}
\def\l@subsection{\@tocline{2}{0pt}{2.5pc}{5pc}{}}
\makeatother

\makeatletter
 \def\@testdef #1#2#3{%
   \def\reserved@a{#3}\expandafter \ifx \csname #1@#2\endcsname
  \reserved@a  \else
 \typeout{^^Jlabel #2 changed:^^J%
 \meaning\reserved@a^^J%
 \expandafter\meaning\csname #1@#2\endcsname^^J}%
 \@tempswatrue \fi}
\makeatother

\makeatletter
\newcommand*{\rom}[1]{\expandafter\@slowromancap\romannumeral #1@}
\makeatother

\makeatletter
\patchcmd{\@sect}{\@addpunct.}{}{}{}
\patchcmd{\subsection}{-.5em}{1em}{}{}
\makeatother

\title{Rigidity of Acute Angled Corners for One Phase Muskat Interfaces}
\author[Siddhant Agrawal]{Siddhant Agrawal$^1$}
\address{$^1$ 
Instituto de Ciencias Matem\'aticas, ICMAT, Madrid, Spain}
\email{siddhant.govardhan@icmat.es
}

\author[Neel Patel]{Neel Patel$^2$}
\address{$^2$ 
Instituto de Ciencias Matem\'aticas, ICMAT, Madrid, Spain}
\email{neel.patel@icmat.es}

\author[Sijue Wu]{Sijue Wu $^3$}
\address{$^3$ 
Department of Mathematics, University of Michigan, Ann Arbor, MI, USA}
\email{sijue@umich.edu}

\begin{document}

\begin{abstract}
We consider the one-phase Muskat problem modeling the dynamics of the free boundary of a single fluid in porous media. We prove local well-posedness for fluid interfaces that are general curves and can have singularities. In particular, the free boundary can have acute angle corners or cusps. Moreover, we show that isolated corners/cusps on the interface must be rigid, meaning the angle of the corner is preserved for a finite time, there is no rotation at the tip, the particle at the tip  remains at the tip and the velocity of that particle at the tip points vertically downward.
\end{abstract}

\maketitle

\tableofcontents

\section{Introduction}\label{sec:introduction}

In this paper, we are concerned with the dynamics of the gravity driven interface of a single fluid in porous media. Fluid flow in porous media is governed by Darcy's law, established in \cite{Da56}:
\begin{equation}\label{eq:Darcy}
    \mu(x,t) u(x,t)=-\nabla P(x,t)-(0,\rho(x,t)),  x\in \Omega(t) \subset \mathbb{R}^{2}, t\geq 0
\end{equation}
where $\mu(x,t)$ is the viscosity, $u(x,t)$ is the fluid velocity, $P(x,t)$ is the pressure and $\rho(x,t)$ is the fluid density in the time-dependent fluid domain $\Omega(t)$. Moreover, the fluid flow is incompressible:
\begin{align*}
    \nabla \cdot u(x,t) = 0, \q x\in \Omega(t) \subset \mathbb{R}^{2}, t\geq 0.
\end{align*}
The dynamics of the sharp fluid interface for a single fluid is known as the \textit{one phase} Muskat problem. In the case that there are two viscous fluids occupying the entire plane, it is the \textit{two-phase} problem \cite{Mu34}. Local and global well-posedness, including possibly in critical regularity and allowing for large slopes, as well as dynamics of the interface are well-studied for interfaces that do not allow for corners: see e.g. \cite{AlLa20}, \cite{AlNg20}, \cite{AlNg21a}, \cite{AlNg21b}, \cite{Am04}, \cite{ChGr16}, \cite{CoCo16}, \cite{CoGa17}, \cite{CoCoGa11}, \cite{CoCoGa13}, \cite{CoLa20}, \cite{EsMa11}, \cite{GaGaPa19}, \cite{Ma19}, \cite{NgPa20} and references therein. Breakdown scenarios starting from smooth initial data has also been shown \cite{CaCoFe13}, \cite{CaCo12}.

Recently, some results in the two-phase regime have considered well-posed of interfaces of regularity that allow for obtuse corners, but no results are known for acute corners or cusps. For example, in \cite{Ca19}, corners on the initial graph interface can be considered with the condition that the product of the modulus of the minimum and maximum slopes is bounded by 1. Hence, the angle of any corner must be between $\pi/2$ and $\pi$. In \cite{KeNgXu21}, modulo a smooth function, the initial interface should have slope less than or equal to 1, and hence, the angle of any corner must again be between $\pi/2$ and $\pi$. These obtuse angle corners are shown to smooth out instantly. A more recent self-similar argument \cite{GaGoNg21} in the two-phase regime shows this smoothing explicitly for very small slopes, or in other words, corners of the interface of large angle very close to $\pi$.

On the other hand, interfaces that allow for both obtuse and acute corners and cusps have been extensively studied for Hele-Shaw flows in a finite domain without gravity, specifically when a single fluid is injected into a horizontal Hele-Shaw cell. Here, instead of gravity, the injection force drives the interface dynamics. The phenomena of a corner on the interface remaining rigid for finite time is often called \textit{waiting time} phenomenon in the Hele-Shaw problem. \cite{KiLa95} constructed self-similar solutions for which corners less than $\pi/2$ persist during the waiting time while corners greater than $\pi/2$ smoothen out. For Lipschitz domains with only obtuse angle corners, the boundary is shown to smoothen \cite{ChJeKi07}. The aforementioned waiting time and smoothing result for corners is extended to more general non-tangentially accessible domains \cite{ChKi06}. Then, \cite{Sa10} uses the theory of quadrature domains to full classify the dynamics of corners of any angle and also cusps. Interestingly, without injection, but with two fluids occupying annulus shaped domains in an annulus shaped rigid container, \cite{BaVa14} can also show the waiting time phenomena for acute corners.

The one-phase gravity driven Muskat problem considered in this paper is equivalent to the Hele Shaw problem without gravity in an infinite domain with injection at infinity. More precisely, the one phase Muskat problem above with $\mu = \rho = 1$ is equivalent to the model considered in section 9 of \cite{ChJeKi09} and this equivalence can be seen by considering the one phase Muskat problem in a reference frame moving at constant speed. In \cite{ChJeKi09}, considering initial interfaces with Lipchitz constant less than $1$, they prove  existence of unique global solutions which smoothen out instantaneously. In the one-phase gravity driven Muskat formulation, for initial data of arbitrarily large Lipschitz norm, \cite{DoGaNg21} shows that graph solutions exist in the $L^{\infty}_{t}L^{2}_{x}$ sense with no growth in the slope. While their regularity allows for sharp corners, the results therein do not give any information about the dynamics of the interface at those corners and moreover, cusps can not be considered. To the best of our knowledge, in the infinite domain one-phase regime with either gravity or an injection force at infinity, results on explicit dynamics of acute angle corners or cusps do not exist in the literature.  

In this paper, we consider the dynamics of a fluid interface in the gravity driven one-phase Muskat problem. We use a Riemann mapping formulation of the system to prove a local well-posedness theorem. The novelty of the result is that it allows us to deal with both smooth interfaces and interfaces that can have singularities such as acute angle corners and cusps. Moreover, we prove that acute corners and cusps remain rigid for a finite time. Namely, the angle of the corner or cusp does not change, there is no rotation at the tip of the corner or cusp, the particle at the tip stays at the tip and the velocity of the particle at the tip points vertically downward. The methods in this paper are inspired from the work on interfaces with corner and cusp singularities done in 2D gravity water waves \cite{Ag20}, \cite{KiWu18}, \cite{Wu19}. We would also like to point out a recent result \cite{CoEnGr21} that studies changing angles for the 2D incompressible Euler equations. While it may be possible that the methods in the previously discussed results \cite{ChKi06}, \cite{Sa10} in the finite fluid domain with injection regime could apply in the infinite domain case considered in this paper; however we are using different methods which in particular do not use the maximum principle techniques utilized in the papers on the finite domain problem. Hence our method will have an advantage in settings where maximum principles are not available and it also brings a different perspective to the problem. In addition, we also prove an interesting identity which is a new conservation law for the one phase Muskat equation. This identity is proven via another equivalent formulation that we derive in the paper as part of the existence proof and may be useful in future work.

\smallskip

\noindent\emph{Outline}: The outline of the paper is as follows. In the remainder of Section \ref{sec:introduction}, we derive the Riemann mapping formulation of the interface equation for the one-phase Muskat problem. Then in Section \ref{sec:mainresults}, we explain the main results of this paper. Sections \ref{sec:apriori} and \ref{sec:uniqueness} give the apriori energy estimates. In Section \ref{sec:interestingidentity}, we derive another equivalent formulation of the interface equation and use this new equivalent formulation to prove an interesting exact identity. Finally, in Section \ref{sec:existence}, we first prove existence in Sobolev spaces using the equivalent formulation and then prove the existence of rough solutions.

\medskip

\noindent \textbf{Acknowledgment}: We are very grateful to Inwon Kim for explaining to us the equivalence of the one phase Muskat equation and the Hele Shaw problem with injection at infinity in an infinite domain. Siddhant Agrawal was partially supported by the National Science Foundation under Grant No. DMS-1928930 while participating in a program hosted by MSRI during the Spring 2021 semester. Neel Patel was partially supported by an AMS-Simons Travel Grant, which are administered by the American Mathematical Society with support from the Simons Foundation. Sijue Wu is supported in part by NSF grant DMS-1764112. She was also supported by NSF grant DMS-1928930 through the program on Mathematical problems in fluid dynamics at MSRI during the Spring 2021 semester.

\subsection{Formulation}

In this section, we will derive the Riemann mapping formulation for the fluid interface. We will try to be as consistent as possible with the notation used in \cite{KiWu18}.  The Fourier transform is defined as
\[
\hat{f}(\xi) = \frac{1}{\sqrt{2\pi}}\int e^{-ix\xi}f(x) \diff x
\] 
We will let $\Scalsp(\Rsp)$ will denote the Schwartz space of rapidly decreasing functions and a Fourier multiplier with symbol $a(\xi)$ is the operator $T_a$ defined formally by the relation $\dis \widehat{T_a{f}} = a(\xi)\hat{f}(\xi)$. The operators $\papabs^s $ for $s\in\Rsp$ are defined as the Fourier multipliers with symbol $\abs{\xi}^s$. 
The Sobolev space $H^s(\Rsp)$ for $s\geq 0$  is the space of functions with  $\norm[H^s]{f} = \norm*[\Ltwo(\diff \xi)]{(1+\abs{\xi}^2)^{\frac{s}{2}}\hat{f}(\xi)} < \infty$. The homogenous Sobolev space $\Hhalf(\Rsp)$ is the space of functions modulo constants with  $\norm[\Hhalf]{f} = \norm*[\Ltwo(\diff \xi)]{\abs{\xi}^\half \hat{f}(\xi)} < \infty$. We will use the notation $\norm[\Linfty\cap\Hhalf]{f} = \norm[\Linfty]{f} + \norm[\Hhalf]{f}$. The Poisson kernel is given by
\begin{align}\label{eq:Poissonkernel}
K_\ep(x) = \frac{\ep}{\pi(\ep^2 + x^2)} \qquad \tx{ for } \ep>0
\end{align}

From now on compositions of functions will always be in the spatial variables. We write $f = f(\cdot,t), g = g(\cdot,t), f \compose g(\cdot,t) :=  f(g(\cdot,t),t)$. Define the operator $U_g$ as given by $U_g f = f\compose g$. Observe that $U_f U_g = U_{g\compose f}$. Let $[A,B] := AB - BA$ be the commutator of the operators $A$ and $B$. We will denote the spacial coordinates in $\Omega(t) $ with $z = x+iy$, whereas $\zp = \xp + i\yp$ will denote the coordinates in the lower half plane $\Pminus = \cbrac{(x,y) \in \Rsp^2 \suchthat y<0}$. We will also use the notation $a \lesssim b$ to denote that there exists a universal constant $C>0$ such that $a \leq Cb$. 

We consider the one-phase setting in which a fluid of constant density $\rho$ and constant viscosity $\mu$ occupies a simply connected domain of infinite depth $\Omega(t)$ and the region $\Rsp^{2} \setminus \Omega(t)$ is a vacuum. Expressing the domain in complex coordinates and assuming $\rho = \mu =1$, Darcy's law becomes
\begin{align}\label{eq:Darcycomplex}
u(x,t)=-\grad P(x,t)-i. 
\end{align} 
Let $\Gamma(t)$ be the boundary of $\Omega(t)$. By the continuity of $P(x,t)$, we know that $P(x,t) = 0$ for any $x\in\Gamma(t)$. Hence, if we parameterize the boundary by Lagrangian coordinates
\begin{align*}
    \Gamma(t) = \cbrac{\z(\alpha,t) \  \suchthat \ \alpha\in\Rsp }
\end{align*}
then the velocity on the boundary is given by $u(\z(\alpha,t),t) = \zt(\alpha,t)$. Also note that the gradient of the pressure is orthogonal to $\Gamma(t)$ and can be expressed as
\begin{align*}
    \grad P(\z(\al,t)) = - i \textbf{a} \zal(\al,t)
\end{align*}
where we define
\begin{align}\label{eq:a}
\textbf{a} \eqdef -\frac{1}{\abs{\zal}}\frac{\partial P}{\partial n}.
\end{align}
On the boundary, we obtain the interface equation
\begin{align}\label{eq:DarcyOnePhase}
 \zt(\al,t) = i \textbf{a}  \zal(\al,t) - i.
\end{align}

Let $\Psi(\cdot,t): \Pminus \to  \Omega(t)$ be conformal maps satisfying $\lim_{\z\to \infty} \Psi_\z(\z,t) =1$  and also $\lim_{\z\to \infty} \Psi_t(\z,t) =0$.  With this, the only ambiguity left in the definition of $\Psi$ is that of the choice of translation of the conformal map at $t=0$, which does not play any role in the analysis. Let $\Phi(\cdot,t):\Omega(t) \to \Pminus $ be the inverse of the map $\Psi(\cdot,t)$ and define $\h(\cdot,t):\Rsp \to \Rsp$ as
\begin{align}\label{eq:h}
    \h(\alpha,t) \eqdef \Phi(\z(\alpha,t),t).
\end{align}
For $\ap \in \Rsp$, we define the map $\hinv(\cdot,t): \Rsp \to \Rsp$ as the inverse of $h(\cdot,t)$ and hence
\begin{align*}
    \h(\hinv(\ap,t),t) = \ap.
\end{align*}
Precomposing $\z(\al,t)$ with $\hinv$, we define
\begin{align*}
    Z(\ap,t) & \eqdef z\compose \hinv (\ap,t), \\
    \Zt(\ap,t) &\eqdef \zt\compose \hinv (\ap,t), \\
    \Ztt(\ap,t) &\eqdef \ztt\compose \hinv (\ap,t)
\end{align*}
as the re-parametrizations of the position, velocity and accelerations of the interface under the Riemann mapping. Denote the derivatives in $\ap$ by
\begin{align*}
    \Zap(\ap,t) = \pap \Z(\al,t) \quad \text{  and  } \quad  \Ztap(\ap,t) = \pap \Zt(\al,t).
\end{align*}
Under the above conventions, we have that $\Psi\compose \h(\al,t) = \z(\al,t)$, and hence, it follows that
\begin{align*}
\Psi(\al;t) &= \Psi(\h(\hinv(\al,t);t);t)\\
&= \z(\hinv(\al,t);t)\\
&= \Z(\al,t).
\end{align*}
Thus, by the chain rule,
\begin{align}
\Zap(\ap,t) \eqdef \pap \Z(\ap,t) & = \pap (\z\compose \hinv(\al,t)) \nonumber\\
&= \zal\compose \hinv (\al,t) \cdot \frac{1}{\hal\compose \hinv(\al,t)}. \label{eq:Zaphapzap}
\end{align}

Now precomposing \eqref{eq:DarcyOnePhase} with $\hinv$, we obtain
\begin{align}\label{eq:DarcyHalfPlane}
 \Zt + i = i\A \Zap
\end{align}
where
\begin{align}\label{eq:A}
\textbf{a} \hal = \A\compose \h.
\end{align}
Because the fluid is irrotational, $\grad \times u = 0$ and incompressible, $\grad \cdot u = 0$, we know that the complex conjugate of the velocity $\ubar (x+iy,t)$ satisfies the Cauchy-Riemann equations and is holomorphic in $\Omega(t)$.
Hence $\ubar\compose \Psi$ is also a holomorphic function in $\Pminus$ with boundary value $\Ztbar$. Additionally, $\abs{\ubar\compose \Psi(\zp,t)}\to 0$ as $\z\to \infty$ and hence we obtain that
\begin{align}\label{eq:hilbertZ}
\Ztbar = \Hil \Ztbar
\end{align}
where the Hilbert transform is given by
\begin{align}\label{eq:HilbertP}
\Hil f(\alpha) \eqdef \frac{1}{i \pi} \text{p.v.} \int_{\Rsp} \frac{f(\beta)}{\alpha-\beta} \diff\beta
\end{align}
for a function $f:\Rsp\to\Rsp$. Observe that with this definition of the Hilbert transform we have $\sqrt{-\Delta} = \papabs = i \Hil \pap$ and the Fourier multiplier of $\Hil$ is $- sgn(\xi)$. We define the operators
\begin{align*}
 \Dap = \Dapfrac \qquad \Dapbar = \Dapbarfrac \qquad \Dapabs = \Dapabsfrac \qquad \Dt = \pt + \bvar\pap
\end{align*}
where $\bvar = \h_t\compose \hinv$. One important property of $\Dap$ is that if a function $f$ satisfies $\Hil f = f$ and if it decays at infinity, then we also have that $\Hil(\Dap f) = \Dap f$. Similarly if $\Hil f = - f$, then $\Hil(\Dapbar f) = - \Dapbar f$. Also $\Dt$ is the material derivative in the Riemann mapping coordinate system as we see that $\Dt (f(\cdot,t)\compose \hinv) = (\pt f(\cdot,t)) \compose \hinv$ or equivalently $\pt (F(\cdot,t)\compose \h) = (\Dt F(\cdot,t)) \compose \h$. Hence $\Dt = U_{h}^{-1}\pt U_{h}$ and we have $\Dt \Z = \Zt$ and $\Dt \Zt = \Ztt$. It is important to remember that in our notation $\pt \Z \neq \Zt$ and $\pt \Zt \neq \Ztt$. We also define the variable
\begin{align*}
\w = \frac{\Zap}{\Zapabs}
\end{align*}

Next, we wish to find an expression for the term $\A$ from \eqref{eq:DarcyHalfPlane}. First, define
\begin{align}\label{eq:Aone}
\Aone = \A\abs{\Zap}^{2}
\end{align}
to obtain
\begin{align*}
    \Zap( \Ztbar -i) = -i\Aone.
\end{align*}
Applying $\Id-\Hil$ to both sides, we obtain
\begin{align*}
    (\Id-\Hil)( \Zap\Ztbar)-(\Id-\Hil)(i\Zap) = -i(\Id-\Hil)\Aone.
\end{align*}
Notice that $\Zap$ is the boundary value of the holomorphic function $\Psi_{\zp}$. As $\Psi_{\zp}\to 1$ as $\z\to \infty$, we see that $(\Id-\Hil)\Zap = 1$. Along with \eqref{eq:hilbertZ}, this yields that
\begin{align*}
    -i(\Id-\Hil)\Aone= -i.
\end{align*}
Considering the imaginary components of both sides and using the fact that $\Aone$ is real valued, we obtain
\begin{align*}
    \Aone = 1
\end{align*}
Thus, Darcy's law becomes 
\begin{align}\label{eq:DarcyLawA1}
 \Ztbar - i  = -i \frac{1}{\Zap}
\end{align}

To obtain the quasilinear equation, we differentiate \eqref{eq:DarcyOnePhase} in time:
\begin{align}
     \zttbar(\alpha,t) = -i \textbf{a}  \ztalbar(\alpha,t) - i \textbf{a}_{t} \zalbar(\alpha,t).
\end{align}
Precomposing with $\hinv$, we obtain
\begin{align*}
     \Zttbar = -i \A \Ztapbar - i (\textbf{a}_{t} \hal)\compose \hinv \Zapbar \eqdef  -i \A \Ztapbar - i\mathcal{B}\Zapbar.
\end{align*}
Hence,
\begin{align*}
     \Zap(\Zttbar + i \A \Ztapbar)= -i\mathcal{B}\abs{\Zapbar}^{2}.
\end{align*}
As earlier, let $\Bone = \mathcal{B} \abs{\Zap}^{2}$. We can alternatively express $\Ztt$ by differentiating in time the formula $\zt (\al, t) = u(\z(\al,t),t)$ to obtain
\begin{align*}
    \zttbar = \ubar_{t} \compose \z + (\ubar_{z}\compose \z)z_{t}
\end{align*}
and therefore,
\begin{align*}
    \Zttbar = \ubar_{t}\compose Z + \frac{\Ztapbar\Zt}{\Zap}.
\end{align*}
Substituting into the above equation for $\Ztt$, we obtain
\begin{align*}
    \Zap(\Zttbar + i \A \Ztapbar)= \Zap\ubar_{t}\compose \Z +  \Ztapbar\Zt + i\A \Zap \Ztapbar.
\end{align*}
Thus, an equation for $B_{1}$ is obtained:
\begin{align*}
    -i\Bone = \Zap\ubar_{t}\compose Z +  \Ztapbar\Zt + i\A \Zap \Ztapbar.
\end{align*}
From \eqref{eq:DarcyLawA1}, we have that
\begin{align*}
i\A \Zap  = \frac{i }{\Zapbar} =  \Zt + i
\end{align*}
Hence,
\begin{align*}
    \Bone = i\Zap\ubar_{t}\compose Z -   \Ztapbar + 2 i \Zt\Ztapbar  .
\end{align*}
Using \eqref{eq:DarcyLawA1} we get
\begin{align}\label{eq:Bonehol}
 \Bone = i\Zap\ubar_{t}\compose Z  -i\pap\frac{1}{\Zap} + \frac{2 i}{\Zapbar}\pap\frac{1}{\Zap}
\end{align}
Applying $\Real(\Id-\Hil)$ to the above equation and  and observing that $\Bone$ is real valued gives
\begin{align}\label{eq:Bone}
    \Bone = - 2\Imag\brac{\sqbrac{\frac{1}{\Zapbar}, \Hil}\pap\frac{1}{\Zap}}
\end{align}
Observe that 
\begin{align*}
    \Bone & = \frac{2}{\pi}\Real\cbrac{\int \frac{\frac{1}{\Zapbar}(\ap) - \frac{1}{\Zapbar}(\bp)}{\ap - \bp} \brac{\pbp \frac{1}{\Zap}}(\bp) \diff \bp} \\
    & = -\frac{1}{\pi}\int \frac{\pbp \abs{\frac{1}{\Zap}(\ap) - \frac{1}{\Zap}(\bp)}^2}{\ap - \bp} \diff\bp \\
    & = \frac{1}{\pi}\int \abs{\frac{\frac{1}{\Zap}(\ap) - \frac{1}{\Zap}(\bp)}{\ap - \bp}}^2 \diff \bp
\end{align*}
Hence
\begin{align}\label{eq:Zttbar}
     \Zttbar + i \frac{1}{\Zapabs^2} \Ztapbar= -i \frac{\Bone}{\Zap}.
\end{align}

Let us now derive a formula for $\bvar$. Recall that $\h(\al,t) = \Phi(\z(\al,t),t)$ and so by taking derivatives we get
\begin{align*}
    \h_t = \Phi_t \compose \z + (\Phi_\z \compose \z)\zt  \quad \qquad \hal = (\Phi_\z \compose \z)\zal 
\end{align*}
Hence
\begin{align*}
 \h_t = \Phi_t \compose \z + \frac{\hal}{\zal}\zt 
\end{align*}
Precomposing with $ \hinv$ we obtain 
\begin{align*}
\h_t \compose \hinv = \Phi_t \compose \Z + \frac{\Zt}{\Zap}
\end{align*}
Apply $(\Id - \Hil)$ and take real part, to get
\begin{align*}
\bvar = \Real(\Id-\Hil)\brac{\frac{\Zt}{\Zap}} 
\end{align*}
Using \eqref{eq:DarcyLawA1}, we obtain
\begin{align*}
\bvar = \Real(\Id-\Hil)\brac{\frac{-i}{\Zap} + \frac{i}{|\Zap|^{2}}} = -\Real(\Id-\Hil)\brac{\frac{i}{\Zap}} - \Hil\brac{\frac{i}{|\Zap|^{2}}}.
\end{align*}
Now $\Hil \brac{\frac{1}{\Zap} - 1} = \brac{\frac{1}{\Zap} - 1}$ and $\Hil 1 = 0$. Therefore we have
\begin{align*}
    \Real(\Id-\Hil)\brac{\frac{i}{\Zap}} = \Real \brac{i} = 0
\end{align*}
Hence
\begin{align}\label{eq:b}
    \bvar = - i\Hil\brac{\frac{1}{|\Zap|^{2}}}
\end{align}
Now observe that as $\bvar$ is real valued, we have
\begin{align*}
     \bvarap  = 2\Real\cbrac{-i\Hil\brac{\frac{1}{\Zapbar}\pap\frac{1}{\Zap}}} 
\end{align*}     
Hence we have
\begin{align}\label{eq:bvarap}
\begin{split}
   \bvarap  & = 2\Real\cbrac{i\sqbrac{\frac{1}{\Zapbar},\Hil}\pap\frac{1}{\Zap} -i\Dapbar\frac{1}{\Zap} } \\
   & = \Bone + 2\Real\cbrac{i\Dap\frac{1}{\Zapbar}}
\end{split}
\end{align}

\subsection{The system}

In summary we see that the system is given in the variable $\Zap(\cdot,t)$ by the equations:
\begin{align}\label{eq:main}
\begin{split}
\bvar & = - i\Hil\brac{\frac{1}{|\Zap|^{2}}} \\
\Bone & = - 2\Imag\brac{\sqbrac{\frac{1}{\Zapbar}, \Hil}\pap\frac{1}{\Zap}} \\
(\pt + \bvar\pap)\frac{1}{\Zap} & =  - i\frac{1}{\Zapabs^2}\pap\frac{1}{\Zap} + \frac{\Bone}{\Zap}
\end{split}
\end{align}
along with the condition that its harmonic extension, namely $\Psizp(\cdot + iy) = K_{-y}\conv \Zap$ for all $y<0$, \footnote{here $K_{-y}$ is the Poisson kernel \eqref{eq:Poissonkernel}} is a holomorphic function on $\Pminus$ and satisfies
\begin{align*}
\lim_{z \to \infty} \pt \Psi(z,t) = 0, \quad 
\lim_{c \to \infty} \sup_{\abs{\zp}\geq c}\nobrac{\abs{\Psizp(\zp) - 1}}  = 0 \tx{ and } \quad \Psizp(\zp) \neq 0 \quad \tx{ for all } \zp \in \Pminus 
\end{align*}

We now make a few remarks about the above system. Observe that for such a $\Psizp$  we can uniquely define $\log(\Psizp) : \Pminus \to \Csp$ such that $\log(\Psizp)$ is a continuous function with $\Psizp = \exp\cbrac{\log(\Psizp)}$ and $(\log(\Psizp))(\zp) \to 0$ as $\zp \to \infty$. Also a straightforward computation using the formulae \eqref{eq:bvarap} and \eqref{eq:commutator} shows that the last equation of \eqref{eq:main} implies
\begin{align*}
\pap(\pt + \bvar\pap)\Zbar = -i\pap\frac{1}{\Zap}
\end{align*}
Now using the fact that $\lim_{z \to \infty} \Psi_t(z,t) = 0$ and  
$\lim_{c \to \infty} \sup_{\abs{\zp}\geq c}\nobrac{\abs{\Psizp(\zp) - 1}}  = 0$ implies that 
\begin{align}\label{eq:Darcynew}
(\pt + \bvar\pap)\Zbar = i - i\frac{1}{\Zap}
\end{align}
which is equation \eqref{eq:DarcyLawA1}. Hence we see that one can obtain $\Z(\cdot,t)$ from the system \eqref{eq:main} by the formula 
\begin{align*}
\Z(\ap,t) & = \Z(\ap,0) +  \int_0^t \cbrac{(\pt + \bvar\pap)\Z - \bvar\Zap}(\ap,s) \diff s \\
& = \Z(\ap,0) +  \int_0^t \cbrac{-i + \frac{i}{\Zapbar} - \bvar\Zap}(\ap,s) \diff s
\end{align*}
In particular instead of the variables $\Zap(\cdot,t)$, one can view the system being in variable $\Z(\cdot,t)$. 

One can rewrite the function $\h(\al,t)$ defined in \eqref{eq:h} as the solution to the ODE
\begin{align*}
\frac{\diff h}{\diff t} &= \bvar(\h, t) \\
h(\al,0) &= \al
\end{align*}
where $\bvar$ is given by \eqref{eq:main}. From this we easily see that as long as $\sup_{[0,T]} \norm[\infty]{\bvarap}(t) <\infty$ we can solve this ODE and for any $t\in [0,T]$ we have that $h(\cdot,t)$ is a homeomorphism. Hence it makes sense to talk about the functions $z = \Z\compose \h, \zt = \Zt \compose \h$ which are Lagrangian parametrizations of the interface and the velocity on the boundary. 

We can also derive the original Darcy's law \eqref{eq:Darcycomplex} from the above system. Let $F(\cdot,t):\Pminus \to \Csp$, $\Psi(\cdot, t): \Pminus \to \Csp$ be holomorphic functions continuous on $\Pminus$ such that
\begin{align*}
F(\ap,t) = \Ztbar(\ap,t) \qq \Psi(\ap,t) = \Z(\ap,t)
\end{align*}
with $F(\zp)(\zp,t)\to 0 $ and $\Psi_\zp(\zp,t) \to 1$ as $\zp \to \infty$. As we have already derived \eqref{eq:Darcynew} from \eqref{eq:main}, we see that
\begin{align*}
    (F-i)\Psi_\zp = -i \qq \tx{ on } \partial \Pminus
\end{align*}
As the left hand side is holomorphic and because of the boundary behaviour of $F$ and $\Psizp$, this implies that
\begin{align*}
    F = -\frac{i}{\Psizp} + i\qq \tx{ on }   \Pminus
\end{align*}
Let $\Pfrak(\cdot,t):\Pminus \to \Rsp$ be defined as $\Pfrak(\zp,t) = -\Imag(\zp)$, which implies that $\Pfrak = 0 $ on $\partial\Pminus$, $\grad \Pfrak \to -i$ as $\zp \to \infty$ and that $\Pfrak(\cdot,t)$ is harmonic on $\Pminus$. Letting $u = \Fbar\compose \Psi^{-1}$ and $P = \Pfrak \compose \Psi^{-1}$ we see that
\begin{align*}
    \ubar\compose \Psi = -\frac{1}{\Psizp}(\pxp - i \pyp)\Pfrak + i = ((\px - i\py) P)\compose \Psi + i \qq \tx{ on } \Pminus 
\end{align*}
which is the same as \eqref{eq:Darcycomplex} written on $\Pminus$ and with a conjugate.

We also observe that as we working in conformal coordinates, there is no restriction on the function $\Z(\cdot,t)$ being injective and hence the above system allows for self intersections, similar to the works of \cite{CaCo16}, \cite{CaCoFe12} where splash singularities for the Hele Shaw or water wave equations were studied. However we will not focus on this aspect in this paper. 
\smallskip

\section{Main Results}\label{sec:mainresults}
In this section, we will outline the main results of this paper. Given conformal maps $\Psi(\cdot,t):\Pminus \to \Omega(t)$, we define the energy 
\begin{align}\label{eq:Mcal1}
\Mcal_{1}(t) =   \sup_{y' < 0} \norm[2]{\partial_{\zp}\frac{1}{\Psizp}(\cdot + iy',t)}^{2} + \sup_{y'<0} \norm[2]{\frac{1}{\Psizp}\partial_{\zp}\brac{\frac{1}{\Psizp}\partial_{\zp}\frac{1}{\Psizp}}(\cdot + iy',t)}^{2} 
\end{align}
We will work in this energy class. For smooth enough interfaces, the energy $\Mcal_1(t)$ is equivalent to an energy on the boundary
\begin{align*}
\Mcal_{1}(t) =  \norm[2]{\pap\frac{1}{\Zap}}^2(t)  + \norm[2]{\Dap^2\frac{1}{\Zap}}^2(t)
\end{align*}
Note that this energy is essentially a sum of weighted Sobolev norms with the weights being powers of $\frac{1}{\Zap}$. As mentioned in the introduction, the question we are concerned with is the dynamics of sharp angles (less than $\pi/2$) or cusps on the curve. Heuristically, a corner of angle $\nu \pi$ at $\ap = 0$ is approximately given nearby as
\begin{equation}
Z(\ap)\approx (\ap)^{\nu}, \frac{1}{\Zap} \approx (\ap)^{1-\nu}, \pap  \frac{1}{\Zap} \approx (\ap)^{-\nu}, \text{ etc. }
\end{equation}
Hence, the condition $\Mcal_1(t) < \infty$ means that our interface can have $0<v < \frac{1}{2}$, i.e. acute angles. Similarly, it allows for cusps (corresponding to $\nu = 0$), see \cite{Ag20} for more details. First, we need to define the notion of initial data and solutions to the system \eqref{eq:main}.

\medskip

\noindent\textbf{The initial data:} We assume that the initial conformal map $\Psi(\cdot,0):\Pminus \to \Csp$ satisfies the following:
\begin{enumerate}
\item The conformal map $\Psi(\cdot,0): \Pminus \rightarrow \mathbb{C}$ extends continuously to the boundary with boundary value being $\Psi(\ap,0) = Z(\ap,0)$,
\item We have $\Psizp(\zp,0) \neq 0$ for all $\zp \in \Pminus$ and  $\Psizp(\zp,0)\rightarrow 1$ as $\abs{\zp}\rightarrow \infty$
\item We assume that $\Mcal_{1}(0) < \infty$
\item We also assume that
\begin{align}\label{eq:c0}
c_0 \eqdef \sup_{y'<0} \norm[2]{\frac{1}{\Psizp}(\cdot + iy',0)- 1}  < \infty.
\end{align}
\end{enumerate}

The following lemma is a straightforward consequence of basic complex analysis.
\begin{lemma}
Suppose that $\Psi: \Pminus \times \sqbrac{0,T} \to  \Csp$ be continuous and suppose that for each $t \in [0,T]$, the map $\Psi(\cdot,t)$ is conformal with $\Psi(\zp,t) \neq 0$ for all $(\zp,t) \in \Pminus\times [0,T]$. Also assume that
\begin{align*}
\sup_{\substack{y'<0\\t\in\sqbrac{0,T}}} \norm[H^1]{\frac{1}{\Psizp}(\cdot + iy',t)- 1} < \infty.
\end{align*}
Then $\frac{1}{\Psizp}$ extends to a continuous function from $\Pminus\times[0,T]$ to $\Pminusbar\times[0,T]$. 
\end{lemma}

\begin{rmk}\label{rem:oneoverZap}
In the above lemma, as $\onePsizp$ extends continuously to $\Pminusbar\times[0,T]$, we call its boundary value as $\frac{1}{\Zap}$ by abuse of notation (as we will not try to define $\Zap$ on the boundary.) We will think of $\frac{1}{\Zap}$ in this case as a single notation and not as one divided by $\Zap$. Hence from this terminology, the function $\frac{1}{\Zap}$ is a continuous function on $\Rsp \times [0,T]$. 
\end{rmk}

Next, we define the class of solutions that we will consider.
\begin{definition}\label{def:solution}
Let $\Psi(\cdot,t): \Pminus \times \sqbrac{0,T} \to  \Omega(t)$. We say $\Psi$ solves the system \eqref{eq:main} in the class $\mathcal{SA}$ (smoothly approximable solutions) if $\Psi$ satisfies
\begin{enumerate}
\item $\Psi(\cdot,t)$ is a conformal map for each $t \in [0,T]$, $\Psi_{\zp}(\zp,t) \neq 0$ for $(\zp,t) \in \Pminus\times [0,T]$ and  $\Psizp(\zp,t) \rightarrow 1$ as $\abs{\zp}\rightarrow \infty$ for any $t\in\sqbrac{0,T}$,
\item $\Psi$ extends continuously to $\overline{\Pminus}\times \sqbrac{0,T}$ and $\Psi \in C^1(\Pminus\times (0,T))$.
\item $\Psi$ satisfies
\begin{align*}
\sup_{\substack{y'<0\\t\in\sqbrac{0,T}}} \norm[H^1]{\frac{1}{\Psizp}(\cdot + iy',t)- 1} < \infty.
\end{align*}
and hence $\frac{1}{\Psizp}$ extends continuously from $\Pminus\times[0,T]$ to $\Pminusbar\times[0,T]$,
\item and $\frac{1}{\Zap}(\ap,t) = \frac{1}{\Psi}(\ap,t)$ solves \eqref{eq:main} in the sense of distributions.
\end{enumerate}
Moreover there exists a sequence of smooth $Z^{(n)}: \mathbb{R}\times \sqbrac{0,T} \rightarrow \mathbb{C}$ satisfying:
\begin{enumerate}[label=(\alph*)]
    \item for each $n\in \mathbb{N}$,
    \begin{align*}
        \sup_{t\in \sqbrac{0,T}} \cbrac*[\Bigg]{\norm[H^2]{\Zap^{(n)}\brac{\cdot,t}-1} + \norm*[\bigg][H^2]{\frac{1}{\Zap^{(n)}}\brac{\cdot,t}-1}} < \infty,
    \end{align*}
    \item for each $n\in \mathbb{N}$, $\Zap^{(n)}$ solves the system, \eqref{eq:main}
    \item the extension $\Psi^{(n)} = K_{y'}\ast Z^{(n)}$ satisfies that $\Psi^{(n)} \rightarrow \Psi$ uniformly on compact subsets of $\overline{\Pminus} \times \sqbrac{0,T}$,
    \item $\frac{1}{\Zap^{(n)}}(\cdot,0)$ converges to $\frac{1}{\Zap}(\cdot,0)$ in $\dot{H}^{\frac{1}{2}}$, and
    \item 
    \begin{align*}
       \sup_{\substack{n\in\mathbb{N}\\ t\in\sqbrac{0,T}}} \Mcal_{1}^{(n)}(t) +  \sup_{\substack{n\in\mathbb{N}\\ t\in\sqbrac{0,T}}}\norm[2]{\frac{1}{\Zap^{(n)}}\brac{\cdot,t}-1} \leq M < \infty
    \end{align*}
    for some constant $M \geq 0$.
\end{enumerate}
\end{definition}

Note that class $\mathcal{SA}$ defined above only contains smoothly approximable solutions and does not contain all functions for which the energy $\Mcal_1(t)$ is finite.

We can now state the first main result of our paper namely that in this class of functions, we have the well-posedness of the system:
\begin{theorem}\label{thm:mainwellposed}
Let $\Psi(\cdot,0): \Pminus \rightarrow \mathbb{C}$ be an initial data satisfying the conditions mentioned above. Then there exists a $T>0$ depending only on $\Mcal_{1}(0)$ such that there is a unique $\Psi:\Pminus\times \sqbrac{0,T}\rightarrow \mathbb{C}$ solving the system \eqref{eq:main} in the class $\mathcal{SA}$ and moreover,
\begin{align*}
    \sup_{t\in\sqbrac{0,T}} \Mcal_{1}(t) \leq C\brac{\Mcal_{1}(0)}
\end{align*}
for $ C\brac{\Mcal_{1}(0)}$ a constant only depending on $\Mcal_{1}(0)$.
\end{theorem}

In particular, the above theorem proves a local well-posedness result for both interfaces that are smooth or can have singularities such as acute angle corners and cusps.

We now define the singular set of the interface at time $t$ as
\begin{align*}
    S(t) = \cbrac{\ap\in\mathbb{R} \ \Big| \  \frac{1}{\Psizp}(\ap,t) = 0}.
\end{align*}
The non-singular set is then defined as $NS(t) = \mathbb{R}\backslash S(t)$. Observe that corners and cusps are indeed included in the singular set. We can now present our result on the rigidity of singularities:

\begin{thm}\label{thm:mainangle}
Let $\Psi$ be a solution in $[0,T]$ as given by Theorem \ref{thm:mainwellposed}. Then
\begin{enumerate}[leftmargin =*, align=left]
\item $S(t) =  \cbrac{h(\al,t) \in \Rsp \suchthat \al \in S(0) }$ for all $t\in[0,T]$.
\item For every fixed $t\in [0,T]$, the  functions  $\brac{\frac{1}{\Psizp}\partial_{\zp}\frac{1}{\Psizp}}(\cdot,t)$ and $\brac{\frac{1}{\Psizp}\overline{\partial_{\zp}\frac{1}{\Psizp}}}(\cdot,t)$ extend to continuous functions on $\Pminusbar$ with 
\begin{align*}
\brac{\frac{1}{\Psizp}\partial_{\zp}\frac{1}{\Psizp}}(\ap,t) = \brac{\frac{1}{\Psizp}\overline{\partial_{\zp}\frac{1}{\Psizp}}} = 0 \quad \forall \  \ap \in S(t)
\end{align*}

\item If $t\in [0,T]$, then we have the following formula
\begin{align}\label{eq:rigididentity}
\frac{\Zap}{\Zapabs}(h(\al,t),t) = \frac{\Zap}{\Zapabs}(\al,0)\exp\cbrac{i\Imag \brac{\int_0^t \brac{\frac{1}{\Psizp}\overline{\partial_{\zp}\frac{1}{\Psizp}}}(h(\al,s),s)\diff s}} \quad \forall \ \al \in NS(0)
\end{align}
Moreover if $\al_{0} \in S(0)$, and if $\{\al_n\}$ is any sequence such that $\al_n \in NS(0)$ for all $n$ with $\al_n \to \al$, then 
\begin{align}\label{eq:rigidlimit}
\lim_{\al_n \to \al_{0}} \frac{\frac{\Zap}{\Zapabs}(h(\al_n,t),t)) }{\frac{\Zap}{\Zapabs}(\al_n,0)} = 1 
\end{align}
\item For all $\ap \in S(t)$, we have $\Zt(\ap,t) = -i$. In particular the velocity of the particle at the tip of a corner or cusp singularity is $-i$. 
\end{enumerate}
\end{thm}

The first part of the theorem implies the particle at the tip of an acute angle corner or cusp singularity remains at the tip. The second part of the theorem says that the gradient of the velocity is zero at the tip. From the third part of the above theorem we see that, if there exists an isolated singularity $\al_{0}\in S(0)$ and unit vectors $\beta,\gamma$, such that 
\begin{align*}
\lim_{\al\to \al_0^{-}} \frac{\Zap}{\Zapabs}(\al,0) = \beta \q \text{ and } \lim_{\al\to \al_0^{+}} \frac{\Zap}{\Zapabs}(\al,0) = \gamma 
\end{align*}
then for all $t\in [0,T]$, there is an isolated singularity of the interface at $h(\al_0,t) \in S(t)$ in conformal coordinates and
\begin{align*}
\lim_{\al\to \al_0^{-}} \frac{\Zap}{\Zapabs}(h(\al,t),t) = \beta \q \text{ and } \lim_{\al\to \al_0^{+}} \frac{\Zap}{\Zapabs}(h(\al,t),t) = \gamma.
\end{align*}
Hence, the acute angle of a corner singularity in $S(0)$ is rigid in $\sqbrac{0,T}$. In particular, the tangent directions of the interface up to the corner singularity are also rigid, and therefore, there is no rotation at the tip. The remainder of this paper will be dedicated to the proof of the above stated main results.

\section{A priori estimates}\label{sec:apriori}

In this section we prove our main a priori energy estimate on the energy:

\begin{align}\label{eq:Mt}
\Mcal(t) = \norm[2]{\pap\frac{1}{\Zap}}^2(t)  + \norm[2]{\Dap^2\frac{1}{\Zap}}^2(t) + 2\int_0^t\norm[\Hhalf]{\frac{1}{\Zap}\Dap^2\frac{1}{\Zap}}^2(s) \diff s.
\end{align}
For smooth solutions, we have the following bound on the evolution of $\Mcal(t)$:
\begin{thm}\label{thm:apriori}
Let $T>0$ and let  $\Z(t)$ be a smooth solution to \eqref{eq:main} in $[0,T]$ such that $(\Zap - 1, \frac{1}{\Zap} - 1) \in \Linfty([0,T], H^{s}(\Rsp) \times H^{s}(\Rsp))$ for all $s \geq 1$. Then there exists a universal constant $C>0$ such that
\begin{align*}
\frac{\diff \Mcal}{\diff t} \leq C(\Mcal + \Mcal^3)
\end{align*}
\end{thm}
The rest of this section is devoted to the proof of this theorem.

\subsection{Quantities controlled by the energy $\Mcal(t)$}

In this subsection, we record some of the terms which are controlled by the energy $\Mcal(t)$. First, we record some frequently used commutator identities. They are easily seen by differentiating
\begin{align}\label{eq:commutator}
 [\pap, \Dt ] &= \bap \pap   \qquad &[\Dapabs, \Dt] &= \Real{(\Dap \Zt)} \Dapabs = \Real{(\Dapbar \Ztbar) \Dapabs} \\   \relax 
 [\Dap, \Dt]  &=  \brac{\Dap \Zt} \Dap   \qquad &  [\Dapbar, \Dt] &= \brac{\Dapbar \Ztbar } \Dapbar 
\end{align}
Using the commutator relation ${[\pap, \Dt ]} = \bap \pap$ we obtain the following formulae
\begin{align}
\Dt \Zapabs & = \Dt e^{\Real\log \Zap} =  \Zapabs \cbrac{\Real(\Dap\Zt) - \bvarap} \label{form:DtZapabs} \\ 
 \Dt \frac{1}{\Zap} & = \frac{-1}{\Zap}(\Dap\Zt - \bvarap) = \frac{1}{\Zap}\cbrac{(\bvarap - \Dap\Zt - \Dapbar \Ztbar) + \Dapbar \Ztbar} \label{form:DtoneoverZap}
\end{align}

We now estimate some of the terms controlled by the energy:
\begin{enumerate}[leftmargin =*, align = left]
\item We have the estimates 
\begin{align*}
\norm[2]{\pap\frac{1}{\Zapabs}} + \norm[2]{\frac{1}{\Zapabs}\pap\w} \lesssim \norm[2]{\pap\frac{1}{\Zap}}
\end{align*}
Proof: We have
\begin{align*}
\frac{\Zap}{\Zapabs}\pap\frac{1}{\Zap} = \w\pap\brac{\frac{\wbar}{\Zapabs}} = \pap \frac{1}{\Zapabs} + \w\Dapabs \wbar
\end{align*}
Observe that $\dis \pap \frac{1}{\Zapabs}$ is real valued and $\dis \w\Dapabs \wbar$ is purely imaginary. From this we obtain the relations
\begin{align} \label{form:RealImagTh}
\Real\brac{\frac{\Zap}{\Zapabs}\pap\frac{1}{\Zap} } = \pap \frac{1}{\Zapabs} \quad \qquad \Imag\brac{\frac{\Zap}{\Zapabs}\pap\frac{1}{\Zap}} = i\brac{ \wbar\Dapabs \w} 
\end{align}
The estimates follow easily from these relations. 
\medskip

\item We have
\begin{align}\label{eq:DapZInfinity}
\norm[\Linfty\cap\Hhalf]{\Dap\frac{1}{\Zap}}^2 \lesssim \norm[2]{\pap\frac{1}{\Zap}}\norm[2]{\Dap^2\frac{1}{\Zap}} \lesssim \Mcal
\end{align}
Proof: Observe that
\begin{align*}
 \pap\brac{\Dap\frac{1}{\Zap}}^2 = 2\brac{\pap\frac{1}{\Zap}}\Dap^2\frac{1}{\Zap}
\end{align*}
As $\Dap\frac{1}{\Zap} \to 0$ at infinity, integrating the above identity gives the $\Linfty$ estimate. For the $\Hhalf$ estimate we observe that $\papabs \Dap\frac{1}{\Zap} = i\pap \Dap\frac{1}{\Zap}$ and hence
\begin{align*}
\norm[\Hhalf]{\Dap\frac{1}{\Zap}}^2 = i\int \brac{\pap\Dap\frac{1}{\Zap}}\Dapbar\frac{1}{\Zapbar} \diff \ap \lesssim  \norm[2]{\pap\frac{1}{\Zap}}\norm[2]{\Dap^2\frac{1}{\Zap}}
\end{align*}
\medskip

\item We have
\begin{align}\label{eq: DalphaHhalf}
\norm[\Hhalf]{\Dapabs\frac{1}{\Zap}} + \norm[\Hhalf]{\Dapbar\frac{1}{\Zap}} \lesssim \norm[\Hhalf]{\Dap\frac{1}{\Zap}} + \norm[2]{\pap\frac{1}{\Zap}}^2 \lesssim \Mcal
\end{align}
Proof: We will only prove the first estimate and the second one is similar. Observe that 
\begin{align*}
\Dapabs\frac{1}{\Zap} = \frac{\w}{\Zap}\pap\frac{1}{\Zap}
\end{align*}
Hence using \propref{prop:Hhalfweight} with $f = \pap\frac{1}{\Zap}$, $w = \frac{1}{\Zap}$ and $h = \w$ we see that
\begin{align*}
\norm[\Hhalf]{\Dapabs\frac{1}{\Zap}} \lesssim \norm[\Hhalf]{\Dap\frac{1}{\Zap}} + \norm[2]{\pap\frac{1}{\Zap}}\brac{\norm*[\bigg][2]{\pap\frac{1}{\Zapabs}} + \norm[2]{\pap\frac{1}{\Zap}}}
\end{align*}
from which the estimate follows. 
\medskip

\item We have 
\begin{align}\label{eq:B1inftyHhalf}
\norm[\Linfty\cap\Hhalf]{\Bone} \lesssim \norm[2]{\pap\frac{1}{\Zap}}^2 \lesssim \Mcal.
\end{align}
and also
 \begin{align}\label{eq:Bap1infty}
 \norm[\infty]{\bvarap} \lesssim \norm[2]{\pap\frac{1}{\Zap}}^2 + \norm[\infty]{\Dap\frac{1}{\Zap}}  \lesssim \Mcal
\end{align}
Proof: The first estimate follows from applying \propref{prop:commutator} to the identity \eqref{eq:Bone}. The second estimate now follows directly from \eqref{eq:bvarap}. 
\medskip

    \item We have
    \begin{align}\label{eq:DapB1}
\norm[2]{\Dapabs\Bone} \lesssim \norm[\infty]{\Dap\frac{1}{\Zap}}\norm[2]{\pap\frac{1}{\Zap}} + \norm[2]{\pap\frac{1}{\Zap}}^{3}.
\end{align}
    Proof: Using Proposition \ref{prop:tripleidentity} on \eqref{eq:Bone} we get
\begin{align*}
\frac{1}{\abs{\Zap}}\pap\Bone &= -2\Imag\brac{\frac{1}{\abs{\Zap}}\pap\sqbrac{\frac{1}{\Zapbar}, \Hil}\pap\frac{1}{\Zap}}\\
&= -2\Imag\brac{    \sqbrac{\frac{1}{\abs{\Zap}}\pap\frac{1}{\Zapbar}, \Hil}\pap\frac{1}{\Zap}    +    \sqbrac{\frac{1}{\Zapbar}, \Hil}\pap\brac{\frac{1}{\abs{\Zap}}\pap\frac{1}{\Zap}}}\\
&\qquad + 2\Imag\brac{\sqbrac{\frac{1}{\abs{\Zap}}, \frac{1}{\Zapbar}; \pap\frac{1}{\Zap}}} \eqdef I_{31} + I_{32} + I_{33}.
\end{align*}
By Proposition \ref{prop:commutator} we see that
\begin{align*}
\norm[2]{I_{31}} + \norm[2]{I_{32}} \lesssim \norm[\infty]{|\Dap|\frac{1}{\Zap}}\norm[2]{\pap\frac{1}{\Zap}} \lesssim \norm[\infty]{\Dap\frac{1}{\Zap}}\norm[2]{\pap\frac{1}{\Zap}} .
\end{align*}
For the last term we use Proposition \ref{prop:triple} to get
\begin{align*}
\norm[2]{I_{33}} \lesssim \norm[2]{\pap\frac{1}{\Zap}}^{3}.
\end{align*}
thereby proving the estimate. 
\medskip

\item We have
\begin{align}\label{eq:1Zap2papBone}
\begin{split}
    \norm[\infty]{\frac{1}{\Zapabs^2}\pap \Bone} & \lesssim \norm[\infty]{\Dap\frac{1}{\Zap}}\norm[2]{\pap\frac{1}{\Zap}}^2 + \norm[2]{\pap\frac{1}{\Zap}}\norm[2]{\Dap^2\frac{1}{\Zap}}  + \norm[2]{\pap\frac{1}{\Zap}}^4 
\end{split}
\end{align}
\medskip
Proof: Observe that $\dis  \frac{1}{\Zapabs^2}\pap\Bone  = \Real \cbrac{\frac{\w^2\wbar^2}{\Zapabs}(\Id - \Hil)\Dapabs\Bone}$ and hence
\begin{align*}
    \norm[\infty]{\frac{1}{\Zapabs^2}\pap \Bone} \leq \norm[\infty]{\frac{\wbar^2}{\Zapabs}(\Id - \Hil)\Dapabs\Bone}
\end{align*}
Now
\[
\frac{\wbar^2}{\Zapabs}(\Id - \Hil)\Dapabs\Bone = (\Id - \Hil)\brac*[\bigg]{\frac{1}{\Zap^2}\pap\Bone} - \sqbrac{\frac{\wbar^2}{\Zapabs},\Hil}\Dapabs\Bone
\]
Using the formula of $\Bone$ from \eqref{eq:Bonehol} we see that  
\begin{align*}
\lpar (\Id - \Hil)\brac*[\bigg]{\frac{1}{\Zap^2}\pap\Bone} & = 2i(\Id - \Hil)\cbrac*[\bigg]{\brac*[\bigg]{\frac{1}{\Zap^2}\pap\frac{1}{\Zapbar}}\pap\frac{1}{\Zap}}\\
&\hspace{20mm}+ 2i(\Id - \Hil)\cbrac*[\bigg]{\frac{1}{\Zapbar}\brac*[\bigg]{\frac{1}{\Zap^2}\pap^2\frac{1}{\Zap}}} \\
& = 2i\sqbrac*[\bigg]{ \frac{1}{\Zap^2}\pap\frac{1}{\Zapbar} ,\Hil}\pap\frac{1}{\Zap} + 2i\sqbrac*[\bigg]{ \frac{1}{\Zapbar},\Hil} \brac*[\bigg]{\frac{1}{\Zap^2}\pap^2\frac{1}{\Zap}}
\end{align*}
Hence from \propref{prop:commutator} we have 
\begin{align*}
\begin{split}
\lpar \norm[\infty]{\frac{\wbar^2}{\Zapabs}(\Id - \Hil)\Dapabs\Bone} & \lesssim \norm[\infty]{\Dap\frac{1}{\Zap}}\norm[2]{\pap\frac{1}{\Zap}}^2 + \norm[2]{\pap\frac{1}{\Zap}}\norm[2]{\frac{1}{\Zap^2}\pap^2\frac{1}{\Zap}} \\
& \quad  + \norm*[\Big][2]{\pap\frac{1}{\Zap}}\norm[2]{\Dapabs\Bone}
\end{split}
\end{align*}
from which the estimate follows.
\medskip

\item  We have
\begin{align}\label{eq:DappapBone}
\begin{split}
 & \norm*[\Bigg][2]{\Dapabs\brac*[\bigg]{\frac{1}{\Zapabs^2} \pap\Bone}} + \norm*[\Bigg][2]{\Dapabs\brac*[\bigg]{\frac{1}{\Zap^2} \pap\Bone}} \\
 &  \lesssim \norm[\infty]{\Dap\frac{1}{\Zap}}\norm[2]{\pap\frac{1}{\Zap}}^3 + \norm[2]{\pap\frac{1}{\Zap}}^2\norm[2]{\Dap^2\frac{1}{\Zap}} + \norm[2]{\pap\frac{1}{\Zap}}^5\\
& \quad   + \norm[\infty]{\Dap\frac{1}{\Zap}}^2\norm[2]{\pap\frac{1}{\Zap}} + \norm[\infty]{\Dap\frac{1}{\Zap}}\norm[2]{\Dap^2\frac{1}{\Zap}}
\end{split}
\end{align}
\medskip
Proof: Observe that $\dis \Dapabs\brac*[\bigg]{\frac{1}{\Zapabs^2}\pap\Bone}  = \Real \cbrac*[\bigg]{\frac{\w^3\wbar^3}{\Zapabs}(\Id - \Hil)\pap\brac*[\bigg]{\frac{1}{\Zapabs^2}\pap\Bone}}$ and hence 
\begin{align*}
\norm[2]{\Dapabs\brac*[\bigg]{\frac{1}{\Zapabs^2}\pap\Bone}} \lesssim \norm[2]{\frac{\wbar^3}{\Zapabs}(\Id - \Hil)\pap\brac*[\bigg]{\frac{1}{\Zapabs^2}\pap\Bone}}
\end{align*}
Now
\begin{align*}
 & \frac{\wbar^3}{\Zapabs}(\Id - \Hil)\pap\brac*[\bigg]{\frac{1}{\Zapabs^2}\pap\Bone} \\
 & = (\Id - \Hil) \cbrac*[\bigg]{\wbar^2\Dap \brac*[\bigg]{\frac{1}{\Zapabs^2}\pap\Bone}} - \sqbrac{ \frac{\wbar^3}{\Zapabs} ,\Hil}\pap \brac*[\bigg]{\frac{1}{\Zapabs^2}\pap\Bone}  \\
 & = (\Id - \Hil) \cbrac*[\bigg]{\Dap \brac*[\bigg]{\frac{1}{\Zap^2}\pap\Bone} - 2\brac*[\bigg]{\frac{\wbar}{\Zapabs^2}\pap\Bone}\brac{\Dap\wbar}}\\
 &\hspace{50mm}- \sqbrac{ \frac{\wbar^3}{\Zapabs} ,\Hil}\pap \brac*[\bigg]{\frac{1}{\Zapabs^2}\pap\Bone}
\end{align*}
Using the formula of $\Bone$ from \eqref{eq:Bonehol} we see that  
\begin{align*}
& (\Id - \Hil) \cbrac*[\bigg]{\Dap \brac*[\bigg]{\frac{1}{\Zap^2}\pap\Bone}} \\
 & = 2i(\Id - \Hil)\cbrac{2\brac{\Dap\frac{1}{\Zap}}\brac{\Dap\frac{1}{\Zapbar}}\pap\frac{1}{\Zap} + \brac{\Dap\frac{1}{\Zap}}\brac{\frac{1}{\Zap^2}\pap^2\frac{1}{\Zapbar}}} \\
 & \quad  + 4i(\Id - \Hil)\cbrac{\brac{\Dap\frac{1}{\Zapbar}}\brac{\frac{1}{\Zap^2}\pap^2\frac{1}{\Zap}}  } + 2i\sqbrac{\frac{1}{\Zapabs^2}, \Hil}\pap\brac{\frac{1}{\Zap^2}\pap^2\frac{1}{\Zap}}
\end{align*}
Hence from the boundedness of Hilbert transform on $\Ltwo$,  \propref{prop:commutator} and \eqref{form:RealImagTh} we have
\begin{align*}
 & \norm*[\Bigg][2]{\Dapabs\brac*[\bigg]{\frac{1}{\Zapabs^2} \pap\Bone}} \\
 &  \lesssim \norm[\infty]{\Dap\frac{1}{\Zap}}\norm[2]{\pap\frac{1}{\Zap}}^3 + \norm[2]{\pap\frac{1}{\Zap}}^2\norm[2]{\Dap^2\frac{1}{\Zap}} + \norm[2]{\pap\frac{1}{\Zap}}^5\\
& \quad   + \norm[\infty]{\Dap\frac{1}{\Zap}}^2\norm[2]{\pap\frac{1}{\Zap}} + \norm[\infty]{\Dap\frac{1}{\Zap}}\norm[2]{\frac{1}{\Zap^2}\pap^2\frac{1}{\Zap}}
\end{align*}
from which the first estimate follows. For the second one we see that
\begin{align*}
\Dapabs\brac*[\bigg]{\frac{1}{\Zap^2} \pap\Bone} & = \Dapabs\brac*[\bigg]{\frac{\wbar^2}{\Zapabs^2} \pap\Bone} \\
& = \wbar^2\Dapabs\brac*[\bigg]{\frac{1}{\Zapabs^2} \pap\Bone} + 2\wbar\brac{\Dapabs\wbar}\frac{1}{\Zapabs^2}\pap\Bone
\end{align*}
Hence we have
\begin{align*}
 \norm*[\Bigg][2]{\Dapabs\brac*[\bigg]{\frac{1}{\Zap^2} \pap\Bone}} \lesssim  \norm*[\Bigg][2]{\Dapabs\brac*[\bigg]{\frac{1}{\Zapabs^2} \pap\Bone}} + \norm[2]{\pap\frac{1}{\Zap}}\norm[\infty]{\frac{1}{\Zapabs^2}\pap\Bone}
\end{align*}
from which the second estimate follows. 
\end{enumerate}

\subsection{Closing the energy estimate}

We now take the time derivative of the energy $\Mcal(t)$ and prove \thmref{thm:apriori}. We first require the following lemma which is direct consequence of Lemma 5.4 in \cite{Ag21}. 

\begin{lem} \label{lem:timederiv}
Let $T>0$ and let $f,\bvar \in C^2([0,T), H^2(\Rsp))$ with $\bvar$ being real valued. Let $\Dt = \pt + \bvar\pap$. Then 
\begin{enumerate}[leftmargin =*, align=left]
\item $\dis \frac{\diff }{\diff t} \int f \diff \ap = \int \Dt f \diff\ap + \int \bvarap f \diff\ap$
\item $\dis \abs*[\Big]{ \frac{d}{dt}\int \abs{f}^2 \diff \ap - 2\Real \int \bar{f} (\Dt f) \diff \ap} \lesssim \min\cbrac{\norm[2]{f}^2 \norm[\infty]{\bvarap},\norm[2]{f}\norm[\infty]{f}\norm[2]{\bvarap}}$
\item $\dis \abs{ \frac{d}{dt}\int (\papabs\bar{f})f \diff\ap- 2\Real \cbrac{\int (\papabs\bar{f})\Dt f \diff \ap}} \lesssim   \norm[\Hhalf]{f}^2 \norm[\infty]{\bvarap}$
\end{enumerate}
\end{lem}
This lemma helps us move the time derivative inside the integral as a material derivative. We will now control the time derivative of the energy.
\smallskip

Using the commutator identity \eqref{eq:commutator} and the equation \eqref{eq:main},
\begin{align*}
& \Dt\pap\frac{1}{\Zap} \\
& = -\bvarap\pap\frac{1}{\Zap} + \pap\Dt\frac{1}{\Zap} \\
& = -\bvarap\pap\frac{1}{\Zap} + \pap\cbrac{- \frac{i}{\Zapabs^2}\pap\frac{1}{\Zap} + \frac{\Bone}{\Zap}} \\
& = -\bvarap\pap\frac{1}{\Zap} -i\brac{\pap\frac{1}{\Zapbar}}\Dap\frac{1}{\Zap} -i\Dapbar\Dap\frac{1}{\Zap} + \Bone\pap\frac{1}{\Zap} + \Dap\Bone
\end{align*}
Hence from \lemref{lem:timederiv} and previously proved estimates we see that
\begingroup
\allowdisplaybreaks
\begin{align*}
& \frac{d}{dt} \norm[2]{\pap\frac{1}{\Zap}}^2 \\
& \lesssim  \norm[2]{\pap\frac{1}{\Zap}}^2\norm[\infty]{\bvarap} + \norm[2]{\pap\frac{1}{\Zap}}\norm[2]{\Dt\pap\frac{1}{\Zap}} \\
& \lesssim  \norm[2]{\pap\frac{1}{\Zap}}^2\norm[\infty]{\bvarap} + \norm[2]{\pap\frac{1}{\Zap}}^2\norm[\infty]{\Dap\frac{1}{\Zap}} + \norm[2]{\pap\frac{1}{\Zap}}\norm[2]{\Dap^2\frac{1}{\Zap}} \\*
& \quad + \norm[2]{\pap\frac{1}{\Zap}}^2\norm[\infty]{\Bone} + \norm[2]{\pap\frac{1}{\Zap}}\norm[2]{\Dap\Bone} \\
& \lesssim \norm[2]{\pap\frac{1}{\Zap}}^4 + \norm[2]{\pap\frac{1}{\Zap}}\norm[2]{\Dap^2\frac{1}{\Zap}} \\
& \lesssim \Mcal + \Mcal^2
\end{align*}
\endgroup

Now from \eqref{eq:DarcyLawA1} and \eqref{eq:main} we have
\begin{align*}
& \Dt\Dap \frac{1}{\Zap} \\
& = \sqbrac{\Dt,\Dap}\frac{1}{\Zap} + \Dap\Dt\frac{1}{\Zap} \\
& = - (\Dap\Zt)\brac{\Dap\frac{1}{\Zap}} + \Dap\cbrac{- \frac{i}{\Zapbar}\Dap\frac{1}{\Zap} + \frac{\Bone}{\Zap}} \\
& =  - 2i \brac{\Dap\frac{1}{\Zapbar}}\brac{\Dap\frac{1}{\Zap}}  - \frac{i}{\Zapbar}\Dap^{2}\frac{1}{\Zap} +\frac{1}{\Zap^2}\pap\Bone  +  \Bone \Dap\frac{1}{\Zap}
\end{align*}
From this we get
\begin{align}\label{eq:DtDapsquaredZap}
\begin{split}
 & \Dt\Dap^{2} \frac{1}{\Zap}\\
& = \sqbrac{\Dt,\Dap}\Dap\frac{1}{\Zap} + \Dap\Dt\Dap\frac{1}{\Zap} \\
& = - (\Dap\Zt)\brac{\Dap^{2}\frac{1}{\Zap}} + \Dap\cbrac{ - 2i \brac{\Dap\frac{1}{\Zapbar}}\brac{\Dap\frac{1}{\Zap}}} \\
& \quad + \Dap\cbrac{ - \frac{i}{\Zapbar}\Dap^{2}\frac{1}{\Zap} +\frac{1}{\Zap^2}\pap\Bone  +  \Bone \Dap\frac{1}{\Zap}} \\
& = -4i\brac*[\bigg]{\Dap\frac{1}{\Zapbar}}\Dap^{2}\frac{1}{\Zap} - 2i\brac*[\bigg]{\Dap^{2}\frac{1}{\Zapbar}}\Dap\frac{1}{\Zap} - i\frac{1}{\Zapbar}\Dap^{3}\frac{1}{\Zap} \\
& \quad + \Dap\brac{\frac{1}{\Zap^{2}}\pap\Bone} + (\Dap\Bone)\Dap\frac{1}{\Zap} + \Bone\Dap^{2}\frac{1}{\Zap} \\
& = -i\Dapbar\brac{\frac{1}{\Zap}\Dap^2\frac{1}{\Zap}} +i\brac*[\bigg]{\Dapbar\frac{1}{\Zap}}\Dap^{2}\frac{1}{\Zap} -4i\brac*[\bigg]{\Dap\frac{1}{\Zapbar}}\Dap^{2}\frac{1}{\Zap}  \\
& \quad  - 2i\brac*[\bigg]{\Dap^{2}\frac{1}{\Zapbar}}\Dap\frac{1}{\Zap} + \Dap\brac{\frac{1}{\Zap^{2}}\pap\Bone} + (\Dap\Bone)\Dap\frac{1}{\Zap} + \Bone\Dap^{2}\frac{1}{\Zap} \\
& =  -i\Dapbar\brac{\frac{1}{\Zap}\Dap^2\frac{1}{\Zap}} + K_1
\end{split}
\end{align}

For the higher order estimate, we have by \lemref{lem:timederiv},
\begin{align}\label{eq:higherorderstart}
\begin{split}
    \abs{\frac{\diff}{\diff t} \norm[2]{\Dap^2\frac{1}{\Zap}}^2 - 2\Real\cbrac{\int\brac{\Dapbar^2\frac{1}{\Zapbar}}\brac{\Dt\Dap^2\frac{1}{\Zap}}} \diff \ap } \lesssim \norm[2]{\Dap^2\frac{1}{\Zap}}^2\norm[\infty]{\bvarap}
\end{split}
\end{align}
Now since
\begin{align*}
\papabs\brac{ \frac{1}{\Zap}\Dap\frac{1}{\Zap}} = i\Hil\pap \brac{\frac{1}{\Zap}\Dap^{2}\frac{1}{\Zap}} = i\pap \brac{ \frac{1}{\Zap}\Dap^{2}\frac{1}{\Zap}} \end{align*}

We see that
\begin{align*}
& 2\Real\cbrac{\int\brac{\Dapbar^2\frac{1}{\Zapbar}}\brac{\Dt\Dap^2\frac{1}{\Zap}} \diff \ap} \\
& = 2\Real\cbrac{\int\brac{\Dapbar^2\frac{1}{\Zapbar}}\brac{ -i\Dapbar\brac{\frac{1}{\Zap}\Dap^2\frac{1}{\Zap}} + K_1 } \diff \ap} \\
& = -2\norm[\Hhalf]{\frac{1}{\Zap}\Dap^{2}\frac{1}{\Zap}}^{2} + 2\Real\cbrac{\int\brac{\Dapbar^2\frac{1}{\Zapbar}} K_1  \diff \ap}
\end{align*}
Hence from \eqref{eq:higherorderstart} and \eqref{eq:DtDapsquaredZap} we see that
\begin{align*}
& \frac{\diff}{\diff t} \norm[2]{\Dap^2\frac{1}{\Zap}}^2 + 2\norm[\Hhalf]{\frac{1}{\Zap}\Dap^{2}\frac{1}{\Zap}}^{2} \\
&  \lesssim \norm[2]{\Dap^2\frac{1}{\Zap}}^2\norm[\infty]{\bvarap} + \abs{\int\brac{\Dapbar^2\frac{1}{\Zapbar}} K_1  \diff \ap} \\
& \lesssim  \norm[2]{\Dap^2\frac{1}{\Zap}}^2\norm[\infty]{\bvarap} + \norm[2]{\Dap^2\frac{1}{\Zap}}\norm[2]{K_1} \\
& \lesssim \norm[2]{\Dap^2\frac{1}{\Zap}}^2\norm[\infty]{\bvarap} + \norm[2]{\Dap^2\frac{1}{\Zap}}^2\norm[\infty]{\Dap\frac{1}{\Zap}} + \norm[2]{\Dap^2\frac{1}{\Zap}}\norm[2]{\Dap\brac{\frac{1}{\Zap^2}\pap\Bone}} \\
& \quad + \norm[2]{\Dap^2\frac{1}{\Zap}}\norm[2]{\Dap\Bone}\norm[\infty]{\Dap\frac{1}{\Zap}} + \norm[2]{\Dap^2\frac{1}{\Zap}}^2\norm[\infty]{\Bone} \\
& \lesssim \Mcal + \Mcal^3
\end{align*}
thereby proving the a priori estimate. 

\section{Uniqueness}\label{sec:uniqueness}

In this section we prove the a priori estimate for the uniqueness of solutions that can be smooth or have singular points. Note that in this paper, we only consider such solutions which are limits of smooth solutions (with suitable conditions) and hence we prove uniqueness only in this class. Therefore to prove uniqueness, we only need to prove an estimate for the difference of two smooth solutions which both approximate the two singular solutions respectively. As a consequence of this a priori estimate, we will also see that the choice of the particular sequence of smooth solutions used to approximate the singular solution is immaterial. 

Let $\Z_a$ and $\Z_b$ be two smooth solutions of \eqref{eq:main}. We denote the two solutions as $A$ and $B$ respectively for simplicity. We will denote the terms and operators for each solution by their subscript $a$ or $b$. We  have the operators
\begin{align*}
(\Dapabs)_a = \frac{1}{\Zapabs_a}\pap \quad (\Dapabs)_b = \frac{1}{\Zapabs_b}\pap \quad \tx{ etc. }
\end{align*} 
Let $h_a, h_b$ be the homeomorphisms from \eqref{eq:h} for the respective solutions and let the material derivatives by given by $(\Dt)_a = U_{h_a}^{-1}\pt U_{h_a}$ and $(\Dt)_b = U_{h_b}^{-1}\pt U_{h_b}$. We define
\begin{align}\label{def:Util}
\htil = h_b \compose h_a^{-1} \quad \tx{ and } \quad \Util = U_{\htil} = U_{h_a}^{-1}U_{h_b}
\end{align}
While taking the difference of the two solutions, we will subtract in Lagrangian coordinates and then bring it to the Riemmanian coordinate system of $A$. The reason we want to subtract in the Lagrangian coordinate system is that in our proof of the energy estimate we mainly use the material derivative, and in the Lagrangian coordinate system the material derivative for both the solutions is given by the same operator $\pt$ and subtracting in Lagrangian coordinates helps us avoid a loss of derivatives. The operator $\Util$ takes a function in the Riemmanian coordinate system of $B$ to the Riemmanian coordinate system of $A$. We define 
\begin{align}\label{def:Delta}
\Delta (f) = f_a - \Util(f_b)
\end{align}
For example $\Delta ({\Zap}) = (\Zap)_a - \Util(\Zap)_b$, where we have written $\Util(f)_b$ instead of $\Util(f_b)$ for easier readability for the term  $\Util(\Zap)_b$. This notation allows us subtract the corresponding quantities of the two solutions in the correct manner, while still using conformal coordinates. Define the operators
\begin{align}\label{eq:Hcal}
\begin{split}
(\Hcal f)(\ap) &= \frac{1}{i\pi} p.v. \int \frac{\htilbp(\bp)}{\htil(\ap) - \htil(\bp)}f(\bp) \diff \bp \\
 (\Hcaltil f)(\ap) &=\frac{1}{i\pi} p.v. \int \frac{1}{\htil(\ap) - \htil(\bp)}f(\bp) \diff \bp
 \end{split}
\end{align}
In this section in addition to these above operators, we will also use the notation $\sqbrac{f_1, f_2;  f_3}$  defined in \eqref{eq:foneftwofthree}. The following lemmas will be useful throughout this section; see \cite{Ag21} for their proofs.

\begin{lem}\label{lem:basicUtil}
Let $\Util$ be defined by \eqref{def:Util} and let $\Hcal, \Hcaltil$ be defined by \eqref{eq:Hcal}. Then
\begin{enumerate}[leftmargin =*, align = left]
\item $(\Dt)_a \Util = \Util (\Dt)_b$
\item $\pap\Util = \htilap\Util\pap \quad$ $\dis \pap\Util^{-1} = \frac{1}{\htilap\compose\htil^{-1}}\Util^{-1}\pap$ and $\dis \htilap = U_{h_a}^{-1}\brac{\frac{(\hal)_b}{(\hal)_a}}$
\item $(\Dt)_a \htil_\ap = -\htil_\ap \Delta(\bvarap)$
\item $\Hcal \Util = \Util \Hil $
\item $\Util\sqbrac{f, \Hil}\pap g = \sqbrac*{(\Util f), \Htil}\pap(\Util g) $
\item $\Util [f_1, f_2;\pap f_3 ] = [(\Util f_1), (\Util f_2) ; \pap(\Util f_3)]_{\htil}$
\end{enumerate}
\end{lem}

\begin{lem}\label{lem:Delta}
Let $\Delta$ be defined by \eqref{def:Delta}. Then
\begin{enumerate}[leftmargin =*, align=left]
\item $\dis \Delta (f_1f_2\cdots f_n) = \sum_{i=1}^{n} \cbrac*[\big]{\Util(f_1)_b\cdots \Util(f_{i-1})_b} \Delta (f_{i}) \cbrac*[\big]{(f_{i+1})_a \cdots (f_n)_a} $
\item $\dis \Delta \sqbrac{f,\Hil}\pap g = \sqbrac{\Delta f, \Hil}\pap(g_a) + \sqbrac*[\big]{\Util(f)_b, \Hil - \Htil}\pap(g_a) + \Util\cbrac*[\big]{\sqbrac{f_b,\Hil}\pap\brac*[\big]{\Util^{-1}(\Delta g)} }$
\end{enumerate}
\end{lem}

\begin{lem}\label{lem:quantM}
Let $T>0$ and let $\Z_a(t)$, $\Z_b(t)$ be two smooth solutions in $[0,T]$ to  \eqref{eq:main}, such that for all $s\geq 2$ we have $(\Zap-1,\frac{1}{\Zap} - 1)_i \in \Linfty([0,T], H^{s}(\Rsp)\times H^{s}(\Rsp))$ for both $i=a,b$.  Let $L>0$ be such that 
\begin{align}\label{eq:Mdef}
\sup_{t \in [0,T]}\Mcal_a(t), \sup_{t \in [0,T]}\Mcal_b(t)  \leq L
\end{align}
We write $a\lesssim_L b$ to mean that $a\leq C(L)b$ for some constant $C(L)$ depending only on $L$. Let $f\in \Scalsp(\Rsp)$. With this notation we have the following estimates for all  $t \in [0,T)$
\begin{enumerate}[leftmargin =*, align=left]
\item $\dis \norm*[\Linfty]{\htilap}(t), \norm[\Linfty]{\frac{1}{\htilap}}(t) \lesssim_L 1$ 
\item $\dis \abs*[\bigg]{\frac{\htil(\ap,t) - \htil(\bp,t)}{\ap-\bp}}, \abs*[\bigg]{\frac{\ap-\bp}{\htil(\ap,t) - \htil(\bp,t)}} \lesssim_L 1$ for all $\ap \neq \bp$
\item $\norm*[2]{\Util f} \lesssim_L \norm[2]{f}$ and $\norm*[\Hhalf]{\Util f} \lesssim_L \norm[\Hhalf]{f}$. These estimates are also true for the operator $\Util^{-1}$ instead of $\Util$. 
\item $\norm[2]{\Hcal(f)} \lesssim_L  \norm[2]{f}$, $\norm[\Hhalf]{\Hcal(f)} \lesssim_L \norm[\Hhalf]{f}$ and $\norm*[2]{\Htil(f)} \lesssim_L  \norm[2]{f}$ 
\item $\dis \norm*[\Hhalf]{\htilap}(t), \norm[\Hhalf]{\frac{1}{\htilap}}(t) \lesssim_L 1$ 
\end{enumerate}
\end{lem}

To prove uniqueness of solutions, we consider the energy
\begin{align}\label{eq:uniquenessenergy}
\mathcal{E}(t) = \norm{\Delta \brac{\frac{1}{\Zap}}}_{\Hhalf}^{2}(t) + \norm[\Linfty\cap\Hhalf]{\Delta\brac{\frac{1}{\Zap}}(\cdot, 0)}^{2} + \int_{0}^{t}\norm{\Delta\brac{\Dap \frac{1}{\Zap}}}_{2}^{2}(s)\diff s
\end{align}
With this energy we can state our main a priori estimate for the difference of solutions.
\begin{thm}\label{thm:mainuniqueness}
Let $T>0$ and let $\Z_a(t)$, $\Z_b(t)$ be two smooth solutions in $[0,T]$ to  \eqref{eq:main}, satisfying the conditions of \lemref{lem:quantM} and let $M>0$ be such that
\begin{align*}
     T, \sup_{t \in [0,T]}\Mcal_a(t), \sup_{t \in [0,T]}\Mcal_b(t)   \leq M
\end{align*}
Then there exists a constant $C(M)>0$ depending only on $M$, so that if   $\mathcal{F}(t) = \sup_{s\in [0,t]} \Ecal(s)$, then for all $t \in [0,T]$ we have
\begin{align*}
    \Fcal(t) \leq C(M) \Fcal(0)
\end{align*}
\end{thm}
The rest of this section is devoted to the proof of this theorem. In this section, we will write $a\lesssim_M b$ to mean that $a\leq C(M)b$ for some constant $C(M)$ depending only on $M$, where $M$ is as given in the statement of the above theorem.

\subsection{Some quantities controlled by $\Ecal(t)$}

To control the evolution of $\mathcal{E}(t)$, we first state several important computations. First, we decompose $\Delta\brac{\frac{1}{\Zap}}$ by the following calculation.

\begin{lemma}
We have
\begin{align}\label{eq:DeltaZdecompb}
\begin{split}
& \Delta\brac{\frac{1}{\Zap}}(\ap,t) \\
& = \Delta\brac{\frac{1}{\Zap}}(\ap,0)\exp\cbrac{\int_{0}^{t}\brac{\bap-\Dap\Zt}_{a}(\h_a(\hinv_a(\ap,t),s),s)\diff s}  \\
& \quad + \Util\brac{\frac{1}{\Zap}}_{b}(\ap,t)\brac{\exp\cbrac{\int_{0}^{t}\Delta\brac{\bap-\Dap\Zt}(\h_a(\hinv_a(\ap,t),s),s)\diff s} - 1}. 
\end{split}
\end{align}
Similarly we also have
\begin{align} \label{eq:DeltaZdecompb2}
\begin{split}
& \Delta\brac{\frac{1}{\Zap}}(\ap,t) \\
& = \Delta\brac{\frac{1}{\Zap}}(\ap,0)\exp\cbrac{\int_{0}^{t}\Util\brac{\bap-\Dap\Zt}_{b}(\h_a(\hinv_a(\ap,t),s),s)\diff s} \\
& \quad + \brac{\frac{1}{\Zap}}_{a}(\ap,t)\brac{-\exp\cbrac{-\int_{0}^{t}\Delta\brac{\bap-\Dap\Zt}(\h_a(\hinv_a(\ap,t),s),s)\diff s} + 1}
\end{split}
\end{align}
\end{lemma}
\begin{proof}
The second identity is very similar to the first identity so we will only prove the first one. Since
\begin{align*}
    \frac{d}{\diff t}\frac{\hal}{\zal} = \frac{\htal}{\hal}\frac{\hal}{\zal} - \frac{\hal}{\zal}\frac{\ztal}{\zal} = \frac{\hal}{\zal}\brac{\frac{\htal}{\hal}-\frac{\ztal}{\zal}},
\end{align*}
we have
\begin{align}\label{eq:hapzapratio}
    \frac{\hal}{\zal}(\al,t) &= \frac{\hal}{\zal}(\al,0)\exp\cbrac{\int_{0}^{t}\brac{\frac{\htal}{\hal}-\frac{\ztal}{\zal}}(\al,s)\diff s}.
\end{align}
So
\begingroup
\allowdisplaybreaks
\begin{align*}
& \frac{\hal^{a}}{\zal^{a}}(\al,t) - \frac{\hal^{b}}{\zal^{b}}(\al,t) \\
& = \brac{\frac{\hal^{a}}{\zal^{a}}(\al,0) - \frac{\hal^{b}}{\zal^{b}}(\al,0)} \exp\cbrac{\int_{0}^{t}\brac{\frac{\htal}{\hal}-\frac{\ztal}{\zal}}_{a}(\al,s) \diff s}\\*
& \quad +\frac{\hal^{b}}{\zal^{b}}(\al,0)\brac{ \exp\cbrac{\int_{0}^{t}\brac{\frac{\htal}{\hal}-\frac{\ztal}{\zal}}_{a}(\al,s) ds}-\exp\cbrac{\int_{0}^{t}\brac{\frac{\htal}{\hal}-\frac{\ztal}{\zal}}_{b}(\al,s) \diff s}}\\
& = \brac{\frac{\hal^{a}}{\zal^{a}}(\al,0) - \frac{\hal^{b}}{\zal^{b}}(\al,0)} \exp\cbrac{\int_{0}^{t}\brac{\frac{\htal}{\hal}-\frac{\ztal}{\zal}}_{a}(\al,s) \diff s}\\*
& \quad +\frac{\hal^{b}}{\zal^{b}}(\al,0)\exp\cbrac{\int_{0}^{t}\brac{\frac{\htal}{\hal}-\frac{\ztal}{\zal}}_{b}(\al,s) \diff s} \\* 
& \qquad \cdot \brac{\exp\cbrac{{\int_{0}^{t}\brac{\frac{\htal}{\hal}-\frac{\ztal}{\zal}}_{a}(\al,s) \diff s -\int_0^t\brac{\frac{\htal}{\hal}-\frac{\ztal}{\zal}}_{b}(\al,s) \diff s}}-1}\\
& = \brac{\frac{\hal^{a}}{\zal^{a}}(\al,0) - \frac{\hal^{b}}{\zal^{b}}(\al,0)} \exp\cbrac{\int_{0}^{t}\brac{\frac{\htal}{\hal}-\frac{\ztal}{\zal}}_{a}(\al,s) \diff s} \\*
& \quad +\frac{\hal^{b}}{\zal^{b}}(\al,t)\brac{\exp\cbrac{{\int_{0}^{t}\brac{\frac{\htal}{\hal}-\frac{\ztal}{\zal}}_{a}(\al,s) \diff s - \int_0^t\brac{\frac{\htal}{\hal}-\frac{\ztal}{\zal}}_{b}(\al,s) \diff s} }-1}
\end{align*}
\endgroup
Writing $\al = \hinv_a(\h_a(\al,t),t) = \hinv_a(\h_a(\al,s),s)$, using the definition of $b$ and \eqref{eq:Zaphapzap}, we obtain
\begin{align*}
&\Delta\brac{\frac{1}{\Zap}}(\h_a(\al,t),t)\\ &= \Delta\brac{\frac{1}{\Zap}}(\al,0)\exp\cbrac{\int_{0}^{t}\brac{\bap-\Dap\Zt}_{a}(\h_a(\al,s),s)\diff s} \\
&\hspace{10mm} + \Util\brac{\frac{1}{\Zap}}_{b}(\h_a(\al,t),t)\brac{\exp\cbrac{\int_{0}^{t}\Delta\brac{\bap-\Dap\Zt}(\h_a(\al,s),s)\diff s} - 1}
\end{align*}
Now plugging in $\al = \hinv_a(\ap,t)$ we get the first identity. 
\end{proof}
The terms showing up in \eqref{eq:DeltaZdecompb} and \eqref{eq:DeltaZdecompb2} satisfy nice estimates. By \eqref{eq:Bap1infty} and \eqref{eq:DapZInfinity}, we obtain that
\begin{align}\label{eq:DeltaZLinfinity}
\begin{split}
& \norm[\Linfty(\Rsp, \diff \ap)]{\Delta\brac{\frac{1}{\Zap}}(\ap,0)\exp\cbrac{\int_{0}^{t}\brac{\bap-\Dap\Zt}_{a}(\h_a(\hinv_a(\ap,t),s),s)\diff s}}  \\
& + \norm[\Linfty(\Rsp, \diff \ap)]{\Delta\brac{\frac{1}{\Zap}}(\ap,0)\exp\cbrac{\int_{0}^{t}\Util\brac{\bap-\Dap\Zt}_{b}(\h_a(\hinv_a(\ap,t),s),s)\diff s}} \\
& \lesssim_{M} \norm[\Linfty(\Rsp, \diff\ap)]{\Delta\brac{\frac{1}{\Zap}}(\ap,0)}
\end{split}
\end{align}
Now since $\abs{e^x - 1} \lesssim_{C} \abs{x}$ for  $\abs{x} \leq C$ for a fixed $C > 0$, we have that
\begin{align}
& \norm[\Ltwo(\Rsp,\diff \ap)]{\exp\cbrac{\int_{0}^{t}\Delta\brac{\bap-\Dap\Zt}(\h_a(\hinv_a(\ap,t),s),s)\diff s} - 1} \nonumber \\
& + \norm[\Ltwo(\Rsp,\diff \ap)]{-\exp\cbrac{-\int_{0}^{t}\Delta\brac{\bap-\Dap\Zt}(\h_a(\hinv_a(\ap,t),s),s)\diff s} + 1} \nonumber \\
&\lesssim_{M} \int_{0}^{t}\norm[2]{\Delta\brac{\bap-\Dap\Zt}}(s)\diff s \nonumber\\
&\lesssim_{M} \int_{0}^{t} \cbrac{\norm[2]{\Delta\brac{\bap}}(s) + \norm[2]{\Delta\brac{\Dap\frac{1}{\Zapbar}}}(s)} \diff s \label{eq:DeltaZL2}\\
&\lesssim_{M} \sqrt{\Ecal(t)} + \int_{0}^{t}\sqrt{\Ecal(s)}\diff s \label{eq:DeltaZL2final}
\end{align}
where last line uses the result of Proposition \ref{prop:deltabalpha} which is proven later.

The next two propositions will provide additional useful bounds for the uniqueness proof. The first proposition is

\begin{prop}\label{prop:halphatildebound}
We have the bound
\begin{align*}
\norm[2]{\htilap-1}(t)  \lesssim_{M} \int_{0}^{t} \norm[2]{\Delta\bap}(s)\diff s
\end{align*}
\end{prop}
\begin{proof}
We have from \lemref{lem:basicUtil}
\begin{align*}
(\Dt)_a\brac{\htilap-1} = -\brac{\htilap-1}\Delta\brac{\bap} - \Delta\brac{\bap}
\end{align*}
and hence, by Lemma \ref{lem:timederiv}
\begin{align*}
\frac{d}{dt}\norm[2]{\brac{\htilap-1}}^{2} \lesssim \norm[2]{\htilap-1}\norm[2]{\Delta\brac{\bap}}\brac{\norm[\infty]{\htilap-1}+1} + \norm[\infty]{\bap}\norm[2]{\htilap-1}^{2}.
\end{align*}
As a result, by Lemma \ref{lem:quantM} and \eqref{eq:Bap1infty} we obtain that
\begin{align}\label{eq:htilapforpropbap}
\frac{d}{dt}\norm[2]{\brac{\htilap-1}} \lesssim_M  \norm[2]{\Delta\brac{\bap}} + \norm[2]{\htilap-1}
\end{align}
and hence 
\begin{align*}
\norm[2]{\brac{\htilap-1}}(t) \lesssim_{M}  \int_{0}^{t}\norm[2]{\Delta\brac{\bap}}(s) \diff s + \int_{0}^{t}\norm[2]{\htilap-1}(s)\diff s
\end{align*}
The proposition now follows from Gronwall's inequality Lemma \ref{lem:Gronwall}.
\end{proof}

Next, we also have the following estimate:

\begin{prop}\label{prop:deltabalpha}
We have the bound
\begin{align*}
\int_{0}^{t}\cbrac{\norm{\Delta\brac{\Dapbar \frac{1}{\Zap}}}_{2}(s) + \norm{\Delta \bap}_{2}(s)} \diff s \lesssim_{M} \sqrt{\Ecal(t)} + \int_{0}^{t}\sqrt{\Ecal(s)}\diff s.
\end{align*}
Using Proposition \ref{prop:halphatildebound}, we additionally obtain that
\begin{align*}
\norm[2]{\htilap-1}(t)  \lesssim_{M} \sqrt{\Ecal(t)} + \int_{0}^{t} \sqrt{\Ecal(s)}\diff s
\end{align*}
\end{prop}

\begin{proof}
By \eqref{eq:bvarap},
\begin{align}\label{eq:deltabetadecomp}
\Delta\bap = \Delta \Bone + 2\Real\cbrac{\Delta \brac{ i\Dap\frac{1}{\Zapbar}}}.
\end{align}
Using \eqref{eq:Bone}, the first term in \eqref{eq:deltabetadecomp} is decomposed using Lemma \ref{lem:Delta}:
\begin{align*}
\Delta \Bone &= -2 \Imag \cbrac{\sqbrac{\Delta\brac{\frac{1}{\Zapbar}}, \Hil}\pap\brac{\frac{1}{\Zap}}_{a} + \sqbrac{\Util\brac{\frac{1}{\Zapbar}}_{b}, \Hil-\Htil}\pap\brac{\frac{1}{\Zap}}_{a}}\\
&\hspace{50mm} -2 \Imag \cbrac{\Util\cbrac{\sqbrac{\brac{\frac{1}{\Zapbar}}_{b},\Hil}\pap\brac{\Util^{-1}\brac{\Delta\frac{1}{\Zap}}}}}.
\end{align*}
We now bound each term. First, using Proposition \ref{prop:commutator} and Sobolev embedding,
\begin{align*}
\norm[2]{\sqbrac{\Delta\brac{\frac{1}{\Zapbar}}, \Hil}\pap\brac{\frac{1}{\Zap}}_{a}} &\lesssim_{M} \norm[\Hhalf]{\Delta\brac{\frac{1}{\Zap}}}
\end{align*}
The second term is bounded using Proposition \ref{prop:HilHtilcaldiff} and using $\pap\Util = \htilap\Util\pap$ from \lemref{lem:basicUtil}
\begin{align*}
\norm[2]{\sqbrac{\Util\brac{\frac{1}{\Zapbar}}_{b}, \Hil-\Htil}\pap\brac{\frac{1}{\Zap}}_{a}}&\lesssim_{M} \norm[2]{\htilap-1}.
\end{align*}
Finally, the third term bounded using Lemma \ref{lem:quantM}, Proposition \ref{prop:commutator} and Sobolev embedding,
\begin{align*}
\norm[2]{\Util\cbrac{\sqbrac{\brac{\frac{1}{\Zapbar}}_{b},\Hil}\pap\brac{\Util^{-1}\brac{\Delta\frac{1}{\Zap}}}}}&\lesssim_{M} \norm[\Hhalf]{\Delta\brac{\frac{1}{\Zap}}}.
\end{align*}
Hence,
\begin{align}\label{eq:DeltaBonebound}
\norm[2]{\Delta\Bone} \lesssim_{M} \norm[\Hhalf]{\Delta\brac{\frac{1}{\Zap}}} + \norm[2]{\htilap-1}.
\end{align}
The next term in \eqref{eq:deltabetadecomp} is the real part of
\begin{align}\label{eq:deltabapsecondterm}
\Delta\brac{i\Dap\frac{1}{\Zapbar}} &= i\Delta\brac{\frac{1}{\Zap}}\Util\brac{\pap\nobrac{\frac{1}{\Zapbar}}}_{b} + i\brac{\frac{1}{\Zap}}_{a}\Delta\brac{\pap\frac{1}{\Zapbar}}.
\end{align}
Using \eqref{eq:DeltaZdecompb}, \eqref{eq:DeltaZLinfinity} and \eqref{eq:DeltaZL2}, we have
\begin{multline}\label{eq:sampledeltaZ}
\norm[2]{\Delta\brac{\frac{1}{\Zap}}\Util\brac{\pap\nobrac{\frac{1}{\Zapbar}}}_{b}}\\ \lesssim_{M}\int_{0}^{t}\brac{\norm[2]{\Delta\brac{\bap}}(s) + \norm[2]{\Delta\brac{\Dap\frac{1}{\Zapbar}}}(s)}\diff s + \norm[\infty]{\Delta\brac{\frac{1}{\Zap}}(\cdot,0)} 
\end{multline}
Next, notice that
\begin{align*}
\norm[2]{\brac{\frac{1}{\Zap}}_{a}\Delta\brac{\pap\frac{1}{\Zapbar}}} &= \norm[2]{\brac{\frac{1}{\Zapbar}}_{a}\Delta\brac{\pap\frac{1}{\Zapbar}}}.
\end{align*}
We expand this term
\begin{align*}
\brac{\frac{1}{\Zapbar}}_{a}\Delta\brac{\pap\frac{1}{\Zapbar}}
&= \Delta\brac{\Dapbar\frac{1}{\Zapbar}} - \Delta\brac{\frac{1}{\Zapbar}}\Util\brac{\pap\frac{1}{\Zapbar}}_{b}.
\end{align*}
Hence,
\begin{align}\label{eq:sampledeltaZ2}
\norm[2]{\brac{\frac{1}{\Zap}}_{a}\Delta\brac{\pap\frac{1}{\Zapbar}}} &\leq \norm[2]{\Delta\brac{\Dap\frac{1}{\Zap}}} + \norm[2]{\Delta\brac{\frac{1}{\Zapbar}}\Util\brac{\pap\frac{1}{\Zapbar}}_{b}}.
\end{align}
Above, the first term appears in $ \diff \Ecal/\diff t$ and the second term above can be controlled similarly to \eqref{eq:sampledeltaZ}. Combining the estimates \eqref{eq:sampledeltaZ} and \eqref{eq:sampledeltaZ2}, we obtain a bound on \eqref{eq:deltabapsecondterm}.

In summary, setting
\begin{align*}
G(t) = \int_{0}^{t}\brac{\norm[2]{\Delta\brac{\bap}}(s) + \norm[2]{\Delta\brac{\Dap\frac{1}{\Zapbar}}}(s)}\diff s,
\end{align*}
and using the bounds on \eqref{eq:DeltaBonebound}, \eqref{eq:deltabapsecondterm} from above and Proposition \ref{prop:halphatildebound}, we have that
\begin{align*}
G'(t) &\lesssim_{M} \norm[2]{\Delta\brac{\Dap\frac{1}{\Zap}}}+ \norm[\infty]{\Delta\brac{\frac{1}{\Zap}}(\cdot,0)}  + \norm[\Hhalf]{\Delta\brac{\frac{1}{\Zap}}} + G(t).
\end{align*}
Integrating in time and applying the Gronwall's inequality of Lemma \ref{lem:Gronwall}, we conclude the proof.
\end{proof}

The following proposition will also be quite useful in closing the energy estimate for $\Ecal(t)$. 

\begin{prop}\label{prop:usefultermuniqueness}
We have the decomposition
\begin{align*}
-i\brac{\frac{1}{\Zapbar}}_{a}\overline{\papabs\Delta\brac{\frac{1}{\Zap}}} =  - \overline{\Delta \brac{\Dap\nobrac{\frac{1}{\Zap}}}} + \mathcal{D}
\end{align*}
where
\begin{align*}
\norm[2]{\mathcal{D}} \lesssim_{\M}  \sqrt{\Ecal(t)} + \int_{0}^{t}\sqrt{\Ecal(s)}\diff s.
\end{align*}
Similarly, we also have 
\begin{align*}
\norm[2]{\Util\brac{\frac{1}{\Zapbar}}_{b} \papabs\Delta\brac{\frac{1}{\Zap}}} \lesssim_{\M} \norm[2]{\Delta \brac{\Dap\nobrac{\frac{1}{\Zap}}}} +  \sqrt{\Ecal(t)} + \int_{0}^{t}\sqrt{\Ecal(s)}\diff s.
\end{align*}
\end{prop}
\begin{proof}
We will first prove the first estimate and the second one will follow from the first. Note that $\Hil\pap\frac{1}{\Zap} = \pap\frac{1}{\Zap}$ and hence, we have by Lemma \ref{lem:basicUtil},
\begin{align*}
 & \papabs\Delta\brac{\frac{1}{\Zap}} \nonumber \\
 &= i\Hil\pap\Delta\brac{\frac{1}{\Zap}}  \nonumber\\
&= i\Hil\pap\cbrac{\brac{\frac{1}{\Zap}}_{a} - \Util\brac{\frac{1}{\Zap}}_{b}}  \nonumber\\
&= i\brac{\pap\frac{1}{\Zap}}_{a} - i\Hil\cbrac{\htilap\Util\brac{\pap\nobrac{\frac{1}{\Zap}}}_{b}} \nonumber\\
&= i\brac{\pap\frac{1}{\Zap}}_{a} - i\Hil\cbrac{\Util\brac{\pap\nobrac{\frac{1}{\Zap}}}_{b}}  - i\Hil\cbrac{\brac{\htilap-1}\Util\brac{\pap\nobrac{\frac{1}{\Zap}}}_{b}}.
\end{align*}

Hence using the above decomposition we have,
\begin{align*}
& i\brac{\frac{1}{\Zap}}_{a}{\papabs\Delta\brac{\frac{1}{\Zap}}} \\
&= - \cbrac{\brac{\pap\frac{1}{\Zap}}_{a} - \Hil\brac{\Util\brac{\pap\nobrac{\frac{1}{\Zap}}}_{b}}}\brac{\frac{1}{\Zap}}_{a} \\ 
& \quad +  {\Hil\brac{\brac{\htilap-1}\Util\brac{\pap\nobrac{\frac{1}{\Zap}}}_{b}}}\brac{\frac{1}{\Zap}}_{a}  \\
&\eqdef {D_{1}} +  {D_{2}}.
\end{align*}
For $D_{1} \eqdef D_{11} + D_{12}$, we decompose further:
\begin{align*}
D_{11} = -\brac{\Dap\frac{1}{\Zap}}_{a} + \brac{\frac{1}{\Zap}}_{a}\Hcal\cbrac{\Util\brac{\pap\nobrac{\frac{1}{\Zap}}}_{b}}
\end{align*}
 and
 \begin{align*}
D_{12} = \brac{\frac{1}{\Zap}}_{a}(\Hil-\Hcal)\cbrac{\Util\brac{\pap\nobrac{\frac{1}{\Zap}}}_{b}}.
\end{align*}
First, by Lemma \ref{lem:basicUtil}, we have 
\begin{align*}
D_{11} &= -\brac{\Dap\frac{1}{\Zap}}_{a} + \brac{\frac{1}{\Zap}}_{a}\Util\brac{\pap\nobrac{\frac{1}{\Zap}}}_{b}\\
&=  -\brac{\Dap\frac{1}{\Zap}}_{a} + \Util\brac{\frac{1}{\Zap}}_{b}\Util\brac{\pap\nobrac{\frac{1}{\Zap}}}_{b} + \Delta\brac{\frac{1}{\Zap}}\Util\brac{\pap\nobrac{\frac{1}{\Zap}}}_{b}\\
& = -\Delta\brac{\Dap\nobrac{\frac{1}{\Zap}}} + \Delta\brac{\frac{1}{\Zap}}\Util\brac{\pap\nobrac{\frac{1}{\Zap}}}_{b} \\
&\eqdef -\Delta\brac{\Dap\nobrac{\frac{1}{\Zap}}} + \tilde{D}_{11}.
\end{align*}
Hence we see that $\overline{\mathcal{D}} = \tilde{D}_{11} + D_{12} + D_2$ and we need to control each of these terms.

\textbf{Step 1:} We now control $\tilde{D}_{11}$. The term $\Delta\brac{\frac{1}{\Zap}}$ in $\tilde{D}_{11}$ is expanded using \eqref{eq:DeltaZdecompb}. For simplicity of notation, we will use the convention
\begin{align*}
\mathcal{T}(\ap,t) = \exp\cbrac{\int_{0}^{t}\brac{\bap-\Dap\Zt}_{a}(\h_a(\hinv_a(\ap,t),s),s)\diff s}.
\end{align*}
From the first term in \eqref{eq:DeltaZdecompb}, plugging into $\tilde{D}_{11}$ and using \eqref{eq:DeltaZLinfinity}, we obtain the estimate
\begin{align*}
& \norm[\Ltwo(\Rsp, \diff \ap)]{\Delta\brac{\frac{1}{\Zap}}(\ap,0)\mathcal{T}(\ap,t)\Util\brac{\pap\frac{1}{\Zap}}_{b}(\ap,t)}  \lesssim_{M} \norm[\infty]{\Delta\brac{\frac{1}{\Zap}}(\cdot,0)}  \lesssim_M \sqrt{\Ecal(t)}
\end{align*}
The second term from \eqref{eq:DeltaZdecompb} placed into $\tilde{D}_{11}$ gives using \eqref{eq:DeltaZL2final}
\begin{align*}
 \norm[\Ltwo(\Rsp,\diff\ap)]{\Util\brac{\frac{1}{\Zap}}_{b}(\ap,t)\brac{\mathcal{T}(\ap,t) - 1}\Util\brac{\pap\frac{1}{\Zap}}_{b}(\ap,t) }
& \lesssim_{\M} \sqrt{\Ecal(t)} + \int_{0}^{t}\sqrt{\Ecal(s)}\diff s
\end{align*} 
Hence combining we have
\begin{align*}
\norm[2]{\tilde{D}_{11}} \lesssim_{M} \sqrt{\Ecal(t)} + \int_{0}^{t}\sqrt{\Ecal(s)}\diff s
\end{align*}

\textbf{Step 2:} We now control $D_{12}$. We can decompose $D_{12} = D_{121}+D_{122}+D_{123} + D_{124}$ where
\begin{align*}
D_{121} &= \brac{\Hil-\Hcal}\cbrac{\Delta\brac{\frac{1}{\Zap}}\Util\brac{\pap\nobrac{\frac{1}{\Zap}}}_{b}}
\end{align*}
\begin{align*}
D_{122} &= \sqbrac{\Delta\brac{\frac{1}{\Zap}}, \Hil-\Hcal}\Util\brac{\pap\nobrac{\frac{1}{\Zap}}}_{b}
\end{align*}
\begin{align*}
D_{123} &=  (\Hil-\Hcal)\cbrac{\Util\brac{\frac{1}{\Zap}}_{b}\Util\brac{\pap\nobrac{\frac{1}{\Zap}}}_{b}}
\end{align*}
and
\begin{align*}
D_{124} &=  \sqbrac{\Util\brac{\frac{1}{\Zap}}_{b},\Hil-\Hcal}\Util\brac{\pap\nobrac{\frac{1}{\Zap}}}_{b}.
\end{align*}
Using Proposition \ref{prop:HilHtilcaldiff} and the bound on \eqref{eq:sampledeltaZ} with Proposition \ref{prop:deltabalpha},
\begin{align*}
\norm[2]{D_{121}} \lesssim_{M} \sqrt{\Ecal(t)} 
+ \int_{0}^{t}\sqrt{\Ecal(s)}\diff s + \norm{\Delta\brac{\frac{1}{\Zap}}(\cdot,0)}_{\infty} \lesssim_{M} \sqrt{\Ecal(t)} 
+ \int_{0}^{t}\sqrt{\Ecal(s)}\diff s.
\end{align*}
Next, by Proposition \ref{prop:HilHtilcaldiff},
\begin{align*}
\norm[2]{D_{122}} &\lesssim_{M} \norm[2]{\pap\nobrac{\Delta\brac{\frac{1}{\Zap}}}} \norm[2]{\htilap-1}\\
&\lesssim_{M} \brac{\norm[2]{\brac{\pap\frac{1}{\Zap}}_{a}}+\norm[2]{\pap\Util\brac{\frac{1}{\Zap}}_{b}}} \norm[2]{\htilap-1} \\
&\lesssim_{M} \sqrt{\Ecal(t)}+ \int_{0}^{t}\sqrt{\Ecal(s)}\diff s
\end{align*}
where we have used Proposition \ref{prop:deltabalpha} and Lemma \ref{lem:basicUtil} with Lemma \ref{lem:quantM}.

Using Proposition \ref{prop:HilHtilcaldiff} and Proposition \ref{prop:deltabalpha},
\begin{align*}
\norm[2]{D_{123}} &\leq \norm[2]{(\Hil-\Hcal)\cbrac{\Util\brac{\Dap\nobrac{\frac{1}{\Zap}}}_{b}}}\\
&\lesssim_{M} \norm[\infty]{\Util\brac{\Dap\nobrac{\frac{1}{\Zap}}}_{b}}\norm[2]{\htilap-1}\\
&\lesssim_{\M} \sqrt{\Ecal(t)}+ \int_{0}^{t}\sqrt{\Ecal(s)}\diff s.
\end{align*}
By Proposition \ref{prop:HilHtilcaldiff} and then Lemma \ref{lem:basicUtil}, Lemma \ref{lem:quantM} and Proposition \ref{prop:deltabalpha}, we get
\begin{align*}
\norm[2]{D_{124}}
&\lesssim_{M} \norm[2]{\pap\Util\brac{\nobrac{\frac{1}{\Zap}}}_{b}}\norm[2]{\Util\brac{\pap\nobrac{\frac{1}{\Zap}}}_{b}}\norm[2]{\htilap-1}\\
&\lesssim_{M}\sqrt{\Ecal(t)}+ \int_{0}^{t}\sqrt{\Ecal(s)}\diff s.
\end{align*}

\textbf{Step 3:} We now control $D_{2}$. We decompose $D_{2} = D_{21} + D_{22} + D_{23} + D_{24}$:
\begin{align*}
D_{21} &= \sqbrac{\Util \brac{\frac{1}{\Zap}}_{b}, \Hil}\cbrac{\brac{\htilap-1}\Util\brac{\pap\nobrac{\frac{1}{\Zap}}}_{b}} 
\end{align*}
\begin{align*}
D_{22} &= \Hil\cbrac{\brac{\htilap-1}\Util\brac{\Dap\nobrac{\frac{1}{\Zap}}}_{b}},
\end{align*}
\begin{align*}
    D_{23} &=  \sqbrac{\Delta\brac{\frac{1}{\Zap}},\Hil}\cbrac{\brac{\htilap-1}\Util\brac{\pap\nobrac{\frac{1}{\Zap}}}_{b}}
\end{align*}
and
\begin{align*}
    D_{24} &=  \Hil\cbrac{\brac{\htilap-1}\Util\brac{\pap\nobrac{\frac{1}{\Zap}}}_{b}\Delta\brac{\frac{1}{\Zap}}}  
\end{align*}
First, by the final estimate of Proposition \ref{prop:commutator} and then using Lemma \ref{lem:basicUtil}, Lemma \ref{lem:quantM} and Proposition \ref{prop:deltabalpha},
\begin{align*}
\norm[2]{D_{21}} &\lesssim_{M} \norm[2]{\pap\Util \brac{\frac{1}{\Zap}}_{b}} \norm[2]{\htilap-1}\norm[2]{\Util \brac{\pap\nobrac{\frac{1}{\Zap}}}_{b}}\\
&\lesssim_{M} \sqrt{\Ecal(t)}+ \int_{0}^{t}\sqrt{\Ecal(s)}\diff s.
\end{align*}
Next, for $D_{22}$, by Proposition \ref{prop:deltabalpha},
\begin{align*}
\norm[2]{D_{22}} &\lesssim_{\M} \sqrt{\Ecal(t)}+ \int_{0}^{t}\sqrt{\Ecal(s)}\diff s.
\end{align*}
Next, for $D_{23}$, using the final inequality of Proposition \ref{prop:commutator} and then Lemmas \ref{lem:basicUtil} and \ref{lem:quantM} and Proposition \ref{prop:deltabalpha}, we get
\begin{align*}
\norm[2]{D_{23}}&\lesssim_{M} \norm[2]{\pap\Delta\brac{\frac{1}{\Zapbar}}}\norm[2]{\htilap-1}\norm[2]{\Util\brac{\pap\nobrac{\frac{1}{\Zap}}}_{b}}\\
&\lesssim_{M} \sqrt{\Ecal(t)}+ \int_{0}^{t}\sqrt{\Ecal(s)}\diff s.
\end{align*}
Using Lemma \ref{lem:quantM}, and then \eqref{eq:sampledeltaZ} and \propref{prop:deltabalpha} we get
\begin{align*}
\norm[2]{D_{24}} & \lesssim_{M} \norm[2]{\Delta\brac{\frac{1}{\Zap}}\Util\brac{\pap\frac{1}{\Zap}}_{b}}\\
&\lesssim_{\M}\norm[\infty]{\Delta\brac{\frac{1}{\Zap}}(\cdot,0)} + \sqrt{\Ecal(t)} + \int_{0}^{t}\sqrt{\Ecal(s)}\diff s \\
&\lesssim_{\M}  \sqrt{\Ecal(t)} + \int_{0}^{t}\sqrt{\Ecal(s)}\diff s
\end{align*}
This concludes the proof of the first part of the proposition.

\textbf{Step 4:} We now prove the second part of the proposition. From \eqref{eq:DeltaZdecompb2} we obtain
\begin{align*}
& \Util\brac{\frac{1}{\Zapbar}}_{b}\papabs\Delta\brac{\frac{1}{\Zap}} \\
& = -\Delta\brac{\frac{1}{\Zapbar}}\papabs\Delta\brac{\frac{1}{\Zap}} + \brac{\frac{1}{\Zapbar}}_{a} \papabs\Delta\brac{\frac{1}{\Zap}} \\
& = -\Delta\brac{\frac{1}{\Zapbar}}(\ap,0)\exp\cbrac{\int_{0}^{t}\overline{\Util\brac{\bap-\Dap\Zt}}_{b}(\h_a(\hinv_a(\ap,t),s),s)\diff s}\papabs\Delta\brac{\frac{1}{\Zap}} \\
& \quad + \brac{\frac{1}{\Zapbar}}_{a}(\ap,t)\brac{\exp\cbrac{-\int_{0}^{t}\overline{\Delta\brac{\bap-\Dap\Zt}}(\h_a(\hinv_a(\ap,t),s),s)\diff s} - 1}\papabs\Delta\brac{\frac{1}{\Zap}} \\
& \quad + \brac{\frac{1}{\Zapbar}}_{a} \papabs\Delta\brac{\frac{1}{\Zap}}
\end{align*}
Hence we see that
\begin{align*}
& \norm[2]{\Util\brac{\frac{1}{\Zapbar}}_{b}\papabs\Delta\brac{\frac{1}{\Zap}}} \\
& \lesssim_{M} \norm[\infty]{\Delta\brac{\frac{1}{\Zap}}(\cdot,0)}\norm[2]{\papabs\Delta\brac{\frac{1}{\Zap}}} + \norm[2]{\brac{\frac{1}{\Zapbar}}_{a} \papabs\Delta\brac{\frac{1}{\Zap}}} \\
& \lesssim_{M} \norm[\infty]{\Delta\brac{\frac{1}{\Zap}}(\cdot,0)} + \norm[2]{\brac{\frac{1}{\Zap}}_{a} \papabs\Delta\brac{\frac{1}{\Zap}}} \\
& \lesssim_{M} \norm[2]{\Delta \brac{\Dap\nobrac{\frac{1}{\Zap}}}} +  \sqrt{\Ecal(t)} + \int_{0}^{t}\sqrt{\Ecal(s)}\diff s
\end{align*}
Hence we have completed the proof of the proposition. 
\end{proof}

\subsection{Controlling the energy estimate}

We will now consider the evolution of the energy $\Ecal(t)$. Differentiating in time, we need to bound
\begin{align*}
\frac{\diff \Ecal}{\diff t} = \frac{d}{\diff t} \norm[\Hhalf]{\Delta\brac{\frac{1}{\Zap}}}^2 + \norm{\Delta\brac{\Dap \frac{1}{\Zap}}}_{2}^{2}.
\end{align*}
By \lemref{lem:timederiv}  and \lemref{lem:basicUtil},
\begin{align*}
 &\abs{\frac{d}{\diff t} \norm[\Hhalf]{\Delta\brac{\frac{1}{\Zap}}}^2 - 2\Real \cbrac{\int (\overline{\papabs\Delta\brac{\frac{1}{\Zap}}})\Delta\brac{\Dt\frac{1}{\Zap}}\diff \ap}}\\
 &\hspace{70mm} \lesssim \norm{\Delta \brac{\frac{1}{\Zap}}}_{\Hhalf}^{2}\norm{(\bap)_a}_{\infty} \\
 &\hspace{70mm} \lesssim_M \Ecal(t)
 \end{align*}
Thus, it remains to bound the quantity
\begin{align}\label{eq:uniquenessexpression}
2\Real \cbrac{\int (\overline{\papabs\Delta\brac{\frac{1}{\Zap}}})\Delta\brac{\Dt\frac{1}{\Zap}}\diff \ap}.
\end{align}

Now by \eqref{eq:main},
\begin{align}
 & \Delta\brac{\Dt\frac{1}{\Zap}}  \nonumber \\
 & = \Delta\brac{-i\frac{1}{\Zapbar}\Dap\frac{1}{\Zap} + \frac{\Bone}{\Zap}} \nonumber\\
& = - i\brac{\frac{1}{\Zapbar}}_{a}\Delta\brac{\Dap\nobrac{\frac{1}{\Zap}}} -i \Delta\brac{\frac{1}{\Zapbar}}\Util\brac{\Dap\frac{1}{\Zap}}_{b}  + \Delta \brac{\frac{\Bone}{\Zap}}. \label{eq:DeltaDtdecomposition}
\end{align}

\textbf{Step 1:} Plugging in the first term of \eqref{eq:DeltaDtdecomposition} in \eqref{eq:uniquenessexpression}, we can use Proposition \ref{prop:usefultermuniqueness}.
\begin{align}\label{eq:uniqueterm1}
\begin{split}
2\Real \cbrac{- i\int (\overline{\papabs\Delta\brac{\frac{1}{\Zap}}})\brac{\frac{1}{\Zapbar}}_{a}\Delta\brac{\Dap\nobrac{\frac{1}{\Zap}}}\diff \ap}\\
= -2\norm[2]{\Delta\brac{\Dap\nobrac{\frac{1}{\Zap}}}}^{2} + 2\Real\cbrac{\int (\mathcal{D})\Delta\brac{\Dap\frac{1}{\Zap}}\diff \ap}.
\end{split}
\end{align}
where we observe that
\begin{align*}
\abs{2\Real\cbrac{\int (\mathcal{D})\Delta\brac{\Dap\frac{1}{\Zap}}\diff \ap}}  \lesssim \norm[2]{\mathcal{D}}\norm[2]{\Delta\brac{\Dap\frac{1}{\Zap}}}
\end{align*}

\textbf{Step 2:} The middle term from \eqref{eq:DeltaDtdecomposition} placed into \eqref{eq:uniquenessexpression} and using \eqref{eq:DeltaZdecompb2} gives the terms
\begin{align*}
 \int -i(\overline{\papabs\Delta\brac{\frac{1}{\Zap}}})\Delta\brac{\frac{1}{\Zapbar}} \Util\brac{\Dap\frac{1}{\Zap}}_{b}\diff \ap  = \nobrac{G_{1} + G_{2}}
\end{align*}
where
\begin{align*}
G_{1} &=  -i\int (\overline{\papabs\Delta\brac{\frac{1}{\Zap}}})\Delta\brac{\frac{1}{\Zapbar}}\brac{\ap,0}\\
&\hspace{20mm}\cdot\nobrac{\exp\cbrac{\int_{0}^{t}\overline{\Util\brac{\bap-\Dap\Zt}}_{b}(\h_a(\hinv_a(\ap,t),s),s)\diff s}} \Util\brac{\Dap \frac{1}{\Zap}}_{b}\diff \ap    
\end{align*}
and
\begin{align*}
G_{2} &= i\int \overline{\papabs\Delta\brac{\frac{1}{\Zap}}}\brac{\frac{1}{\Zapbar}}_{a}\\
&\hspace{10mm}\cdot\brac{\exp\cbrac{-\int_{0}^{t}\Delta\brac{\bap-\Dap\Zt}(\h_a(\hinv_a(\ap,t),s),s)\diff s} - 1}\Util\brac{\Dap \frac{1}{\Zap}}_{b}\diff \ap
\end{align*}
Now  by \eqref{eq:DeltaZLinfinity} and Proposition \ref{prop:usefultermuniqueness}
\begin{align}
\abs{G_{1}} & \lesssim_{\M} \norm[2]{\Util\brac{\frac{1}{\Zap}}_{b}\overline{\papabs\Delta\brac{\frac{1}{\Zap}}}}\norm[\infty]{\Delta\brac{\frac{1}{\Zap}}\brac{\cdot,0}}\norm[2]{\Util\brac{\pap \frac{1}{\Zap}}_{b}}\nonumber\\
& \lesssim_{\M} \brac{\norm[2]{\Delta\brac{\Dap\nobrac{\frac{1}{\Zap}}}} +  \sqrt{\Ecal(t)} + \int_{0}^{t}\sqrt{\Ecal(s)}\diff s}\sqrt{\Ecal(t)} \label{eq:uniqueterm5}
\end{align}
Similarly by \eqref{eq:DeltaZL2final} and Proposition \ref{prop:usefultermuniqueness}
\begin{align}
\abs{G_{2}} & \lesssim_{\M} \brac{\norm[2]{\Delta\brac{\Dap\nobrac{\frac{1}{\Zap}}}} + \norm[2]{\mathcal{D}}} \brac{\sqrt{\Ecal(t)} + \int_{0}^{t} \sqrt{\Ecal(s)} \diff s} \label{eq:uniqueterm6}.
\end{align}

\textbf{Step 3:} The final term of \eqref{eq:DeltaDtdecomposition} is
\begin{align*}
\Delta\brac{\frac{\Bone}{\Zap}} = \Delta\brac{\Bone}\brac{\frac{1}{\Zap}}_{a} + \Delta\brac{\frac{1}{\Zap}}\Util\brac{\Bone}_{b}.
\end{align*}
This term in the expression \eqref{eq:uniquenessexpression} gives using \eqref{eq:DeltaZdecompb2}
\begin{align*}
& \int \brac*[\Big]{\overline{\papabs\Delta\brac{\frac{1}{\Zap}}}}\Delta\brac{\frac{\Bone}{\Zap}}\diff \ap \\
& = \int \brac*[\Big]{\overline{\papabs\Delta\brac{\frac{1}{\Zap}}}}\Delta\brac{\Bone}\brac{\frac{1}{\Zap}}_{a}\diff \ap +  \int \brac*[\Big]{\overline{\papabs\Delta\brac{\frac{1}{\Zap}}}}\Delta\brac{\frac{1}{\Zap}}\Util\brac{\Bone}_{b}\diff \ap\\
& = F_{1} + F_{2} + F_{3}
\end{align*}
where
\begin{align*}
F_{1} &= \int \brac*[\Big]{\overline{\papabs\Delta\brac{\frac{1}{\Zap}}}}\Delta\brac{\Bone}\brac{\frac{1}{\Zap}}_{a}\diff \ap
\end{align*}
\begin{align*}
F_{2} &=  \int (\overline{\papabs\Delta\brac{\frac{1}{\Zap}}})\brac{\frac{1}{\Zap}}_{a}\Util\brac{\Bone}_{b}\\
&\hspace{20mm}\cdot\brac{-\exp\cbrac{-\int_{0}^{t}\Delta\brac{\bap-\Dap\Zt}(\h_a(\hinv_a(\ap,t),s),s)\diff s} + 1}\diff \ap
\end{align*}
and
\begin{align*}
F_{3} &= \int (\overline{\papabs\Delta\brac{\frac{1}{\Zap}}})\Delta\brac{\frac{1}{\Zap}}(\ap,0) \Util\brac{\Bone}_{b}\\
&\hspace{20mm}\cdot\nobrac{\exp\cbrac{\int_{0}^{t}\Util\brac{\bap-\Dap\Zt}_{b}(\h_a(\hinv_a(\ap,t),s),s)\diff s}}\diff\ap
\end{align*}
For $F_{1}$, we use  \eqref{eq:DeltaBonebound}, \propref{prop:deltabalpha} and  \propref{prop:usefultermuniqueness} to get
\begin{align}
\abs{F_{1}} &\leq \norm[2]{\brac{\frac{1}{\Zap}}_{a}\papabs\Delta\brac{\frac{1}{\Zap}}}\norm[2]{\Delta\brac{\Bone}} \nonumber\\
&\lesssim_{M}  \brac{\norm[2]{\Delta\brac{\Dap\nobrac{\frac{1}{\Zap}}}} + \norm[2]{\mathcal{D}}}\brac{\sqrt{\Ecal(t)}+\int_{0}^{t}\sqrt{\Ecal(s)}\diff s} \label{eq:uniqueterm2}
\end{align}
For $F_{2}$ we have by \eqref{eq:B1inftyHhalf}, \eqref{eq:DeltaZL2final} and  \propref{prop:usefultermuniqueness}
\begin{align}
\abs{F_{2}} & \lesssim_{M} \norm[2]{\brac{\frac{1}{\Zap}}_{a}\papabs\Delta\brac{\frac{1}{\Zap}}}\nonumber\\
&\hspace{30mm}\cdot\norm[\Ltwo(\Rsp,\diff\ap)]{-\exp\cbrac{-\int_{0}^{t}\Delta\brac{\bap-\Dap\Zt}(\h_a(\hinv_a(\ap,t),s),s)\diff s} + 1} \nonumber\\
&\lesssim_{\M} \brac{\norm[2]{\Delta\brac{\Dap\nobrac{\frac{1}{\Zap}}}} + \norm[2]{\mathcal{D}}}\brac{\sqrt{\Ecal(t)}+\int_{0}^{t}\sqrt{\Ecal(s)}\diff s}\label{eq:uniqueterm4}
\end{align}
For $F_{3}$, we move half derivative to the latter terms to obtain
\begin{align*}
\abs{F_{3}}
&\leq \norm[\Hhalf]{\Delta\brac{\frac{1}{\Zap}}}\\
&\hspace{10mm}\cdot\norm[\Hhalf]{\Delta\brac{\frac{1}{\Zap}}(\cdot,0)\nobrac{\exp\cbrac{\int_{0}^{t}\Util\brac{\bap-\Dap\Zt}_{b}(\h_a(\hinv_a(\cdot,t),s),s)\diff s}} \Util\brac{\Bone}_{b}}
\end{align*}
We observe that 
\begin{align*}
\norm[\Linfty\cap\Hhalf]{\Delta\brac{\frac{1}{\Zap}}(\cdot,0)} \leq \sqrt{\Ecal(t)}
\end{align*}
From \propref{prop:chainrule} and \lemref{lem:quantM} we have
\begin{align*}
\norm[\Linfty\cap\Hhalf]{\exp\cbrac{\int_{0}^{t}\Util\brac{\bap-\Dap\Zt}_{b}(\h_a(\hinv_a(\cdot,t),s),s)\diff s}} \lesssim_{M} 1
\end{align*}
Similarly from \lemref{lem:quantM} we have
\begin{align*}
\norm[\Linfty\cap\Hhalf]{\Util(\Bone)_b} \lesssim_{M} 1
\end{align*}
Hence by repeated application of \propref{prop:Leibniz} we obtain
\begin{align}\label{eq:uniqueterm7}
\abs{F_{3}} \lesssim_{M} \norm[\Hhalf]{\Delta\brac{\frac{1}{\Zap}}}\sqrt{\Ecal(t)} \lesssim_{M} \Ecal(t)
\end{align}

\textbf{Step 4:} Hence, in total, combining the bounds \eqref{eq:uniqueterm1},
\eqref{eq:uniqueterm5},
\eqref{eq:uniqueterm6},
\eqref{eq:uniqueterm2}, \eqref{eq:uniqueterm4},
\eqref{eq:uniqueterm7} and then the bound on $\mathcal{D}$ from Proposition \ref{prop:usefultermuniqueness}, we have
\begin{align*}
\frac{\diff \Ecal}{\diff t} &= -\norm{\Delta\brac{\Dap \frac{1}{\Zap}}}_{2}^{2} + \mathcal{N}
\end{align*}
where
\begin{align}
\abs{\mathcal{N}} &\lesssim_{M}  \brac{\sqrt{\Ecal(t)} + \int_{0}^{t} \sqrt{\Ecal(s)} \diff s}^{2} + \norm[2]{\Delta\brac{\Dap\nobrac{\frac{1}{\Zap}}}} \brac{\sqrt{\Ecal(t)} + \int_{0}^{t} \sqrt{\Ecal(s)} \diff s} \label{eq:Nunique}
\end{align}
Let the implicit constant given by $\lesssim_{M}$ in \eqref{eq:Nunique} be given by $C(M)$. For the second term above, we apply Young's inequality for products:
\begin{align*}
& C(M)\norm[2]{\Delta\brac{\Dap\nobrac{\frac{1}{\Zap}}}} \brac{\sqrt{\Ecal(t)} + \int_{0}^{t} \sqrt{\Ecal(s)} \diff s}\\
& \leq
\frac{1}{2}C(M)^{2}\brac{\sqrt{\Ecal(t)} + \int_{0}^{t} \sqrt{\Ecal(s)} \diff s}^{2} + \half \norm[2]{\Delta\brac{\Dap\nobrac{\frac{1}{\Zap}}}}^{2}.
\end{align*}
Hence by defining $\mathcal{F}(t) = \sup_{s\in [0,t]} \Ecal(s)$ we obtain
\begin{align*}
\frac{\diff \Ecal}{\diff t} &\lesssim_{\M} \brac{\sqrt{\Ecal(t)} + \int_{0}^{t} \sqrt{\Ecal(s)} \diff s}^{2}\lesssim_{M} \mathcal{F}(t)
\end{align*}
and then integrating the above in time and taking the supremum in time, we have
\begin{align*}
\mathcal{F}(t) \lesssim_{M} \Ecal(0) + \int_{0}^{t} \mathcal{F}(s) \diff s.
\end{align*}
Finally, applying Gronwall's inequality,
\begin{align*}
    \mathcal{F}(t) \lesssim_{M} \mathcal{F}(0).
\end{align*}
thereby concluding the proof of \thmref{thm:mainuniqueness}

We now prove an additional estimate that will not be useful until the existence proof in Section \ref{sec:existence}.

\begin{lemma}\label{lem:uniquenessforexistenceBone}
Say $\ell:  \mathbb{R}\rightarrow \mathbb{R}$ is a diffeomorphism with $\ell-\ap$ in $H^1$. Then, we have for any $f\in \dot{H}^{1}$,
\begin{align}\label{eq:ellineq}
\norm*[\big][\dot{H}^{\frac{1}{2}}]{f\circ \ell - f}
&\lesssim \norm[2]{\pap f}\norm*[\big][2]{\ell-\ap}^{\frac{1}{4}}\norm*[\big][2]{\ell_\ap-1}^{\frac{1}{4}}\nonumber \\ & \quad + C\brac*[\big]{\norm*[\big][\infty]{\brac*[\big]{\ell^{-1}}_{\ap}}, \norm*[\big][\infty]{\ell_\ap}}\norm*[\big][2]{\ell_\ap-1}\norm[2]{\pap f}
\end{align}
In particular, assuming the hypothesis of Theorem \ref{thm:mainuniqueness}, then we have that $\htil(\cdot,t): \mathbb{R}\rightarrow \mathbb{R}$ is a diffeomorphism with $\htil(\cdot,t)-\ap$ in $H^1$. Moreover,
\begin{align*}
\sup_{t\in\sqbrac{0,T}}\norm*[\big][2]{\htil-\ap}(t) \leq C(M)
\end{align*}
where $M$ is the same constant as in Theorem \ref{thm:mainuniqueness}. Hence the quantity $\norm*[\big][\dot{H}^{\frac{1}{2}}]{f\circ \ell - f}$ satisfies the inequality \eqref{eq:ellineq} with $\ell = \htil$.
\end{lemma}

\begin{proof}
The inequality \eqref{eq:ellineq} is a restatement of Lemma 6.1 in \cite{Wu19}.
Next, we show that $\htil(\cdot,t)-\ap$ is in $L^2$. At $t=0$, $\htil(\ap,0)-\ap = 0$.
\begin{align*}
\brac{\Dt}_{a}\brac{\htil-\ap} = \brac{\partial_t h_b}\circ h_a^{-1} -\bvar_a = (b_b)\compose h_b\compose h_a^{-1} - \bvar_a.
\end{align*}
Since $\bap\in L^\infty$,
\begin{align*}
    \frac{d}{dt}\norm*[\big][2]{\htil-\ap}^{2} &\lesssim \norm[\infty]{(\bap)_a}\norm*[\big][2]{\htil-\ap}^{2} + \norm*[\big][2]{\htil-\ap}\norm[2]{(b_b)\compose h_b\compose h_a^{-1} - \bvar_a} \\
    & \leq C(M) \brac{\norm*[\big][2]{\htil-\ap}^{2} + \norm*[\big][2]{\htil-\ap}}
\end{align*}
Further, we have $\htil(\cdot,t)-\ap$ in $\dot{H}^1$ by Proposition \ref{prop:deltabalpha}. This concludes the proof.
\end{proof}

\section{An equivalent system and an interesting identity}\label{sec:interestingidentity}

In this section we will reformulate system \eqref{eq:main} into an equivalent system \eqref{eq:Dtg}, which will be useful to prove existence of solutions in Sobolev spaces (which is done in the next section). As a consequence of this equivalent formulation, we will also derive a new identity \eqref{eq:consHhalfg} which we think is quite interesting.

The new system is given in the real valued variable $g(\cdot,t)$ with the equations being:
\begin{align}\label{eq:Dtg}
\begin{split}
c & =  e^{-i\Hil \g}\\
\bstar & = - i\Hil\brac{c^2} \\
(\pt + \bstar\pap) \g & =  - c^2\papabs \g     
\end{split}
\end{align}
To get system \eqref{eq:Dtg} from system \eqref{eq:main} we use the following transformation
\begin{align}\label{eq:systemonetwo}
\g  = \Imag(\log( \Zap)) 
\end{align}
and the following to get system \eqref{eq:main} from system \eqref{eq:Dtg}
\begin{align}\label{eq:systemtwoone}
\Zap = e^{i (\Id + \Hil)\g}
\end{align}
We now show that the two systems are equivalent.
\begin{lem}\label{lem:systemequiv}
Let $s\geq 2$ and $T > 0$. Then $\Z(\cdot,t)$ solves \eqref{eq:main} with $(\Zap -1, \frac{1}{\Zap} -1) \in \Linfty([0,T], H^{s}(\Rsp)\times H^{s}(\Rsp))$ if and only if $g(\cdot,t)$ solves \eqref{eq:Dtg} along with $\g \in \Linfty([0,T], H^{s}(\Rsp))$, where the transformations between them are given by \eqref{eq:systemonetwo} and  \eqref{eq:systemtwoone}
\end{lem}
\begin{proof}
Step 1: We first assume that $\Zap(\cdot,t)$ solves \eqref{eq:main} and prove that $\g(\cdot,t)$ solves \eqref{eq:Dtg}. 
\begin{enumerate}
\item If $(\Zap -1, \frac{1}{\Zap} -1) \in \Linfty([0,T], H^{s}\times H^{s})$ for $s\geq 2$, then for any $t\in[0,T]$ we see that $\norm[\infty]{\Zap}(t) + \norm[\infty]{\frac{1}{\Zap}}(t) + \norm[2]{\abs{\Zap} - 1}(t) \leq M < \infty $
for some $M>0$. Now as $\Zap -1 \in H^{s}(\Rsp)$ we observe that $\Psizp$ extends continuously to $\Pminusbar$ and hence $\log(\Psizp)$ also extends continuously to the boundary. Hence it makes sense to talk about the function $\log(\Zap)$.  
Observe that if $C_1>0$ then
\begin{align*}
c_1 \abs{z} \leq \abs{e^z -1} \leq c_2 \abs{z} \qq \tx{ for all } z\in \Rsp, \abs{z} \leq C_1
\end{align*}
for some $c_1,c_2 >0$ depending only on $C_1$. Now as $\abs{\Zap} = e^{\Real(\log(\Zap))}$ we see that $\Real(\log(\Zap)) \in \Ltwo$. Hence we see that $\Imag(\log(\Zap)) \in \Ltwo$ and hence $\g \in \Ltwo$ and $\Zap = e^{i(\Id + \Hil)\g}$. Now using \eqref{eq:systemonetwo} and the formula $\pap g = \Imag\brac{\frac{1}{\Zap}\pap\Zap}$ we see that $\g \in \Linfty([0,T], H^{s}(\Rsp))$.

\item  As $\Zap = e^{i(\Id + \Hil)\g}$ we can then observe from \eqref{eq:systemonetwo} that $\cvar  = e^{-i\Hil \g} = \frac{1}{\Zapabs}$ and $\bstar =  -i\Hil\brac{\frac{1}{\Zapabs^2}} = \bvar$. Hence we can now identify $\bstar$ with $\bvar$ and the operator $(\pt + \bstar\pap) $ is the same as the operator $\Dt = (\pt + \bvar\pap)$. 

\item Let us define $f = \Real(\log(\Zap))$ and so $\Zap = e^{f + ig}$. Note that both $f, g \to 0$ as $\abs{\ap} \to \infty$ (as $\Zap \to 1$ as $\abs{\ap} \to \infty$). We now see that $\Hil(\f + i\g) = \f + i\g$ and hence
\begin{align*}
    \Hil(\f) = i\g \quad \tx{ and } \quad \Hil(\g) = -if
\end{align*}
Therefore
\begin{align*}
    \pap\frac{1}{\Zapabs} = \pap e^{-f} = -\frac{1}{\Zapabs} \f_\ap = -\frac{1}{\Zapabs}\papabs g.
\end{align*}
Now we have
\begin{align*}
    \Dt (\f + i\g) & = -\Dt \log\brac{\frac{1}{\Zap}} = - \Zap\Dt\frac{1}{\Zap} 
\end{align*}
Using the main equation \eqref{eq:main},
\begin{align}\label{eq:f+igeq}
    \Dt(\f + i\g) = -\Zap\cbrac{-\frac{i}{\Zapabs^2}\pap\frac{1}{\Zap} + \frac{\Bone}{\Zap}}
\end{align}
From \eqref{eq:Bone} we see that $\Bone$ is real valued and hence taking imaginary part to the equation \eqref{eq:f+igeq} we obtain from \eqref{form:RealImagTh}
\begin{align*}
    \Dt\g = \frac{1}{\Zapabs}\Real\cbrac{\frac{\Zap}{\Zapabs}\pap\frac{1}{\Zap}} =  \frac{1}{\Zapabs}\pap\frac{1}{\Zapabs}
\end{align*}
and similarly by taking the real part we get from \eqref{form:RealImagTh}
\begin{align*}
    \Dt \f = -\Bone - \frac{1}{\Zapabs}\Imag\cbrac{\frac{\Zap}{\Zapabs}\pap\frac{1}{\Zap}} = -\Bone + \frac{1}{\Zapabs^2}g_{\ap}
\end{align*}
Hence we get the equation \eqref{eq:Dtg}
\begin{align}\label{eq:Dtgduplicate}
    \Dt \g + \frac{1}{\Zapabs^2}\papabs \g = 0
\end{align}
and also 
\begin{align}\label{eq:Dtf}
    \Dt \f + \frac{1}{\Zapabs^2}\papabs \f = -\Bone
\end{align}
\end{enumerate}
Step 2: We now assume that $g(\cdot,t)$ solves \eqref{eq:Dtg} and prove that $\Zap(\cdot,t)$ solves $\eqref{eq:main}$. We will again use the notation $f = i\Hil g$.  
\begin{enumerate}
\item Observe that if $C_1>0$ then 
\begin{align*}
 \abs{e^z -1} \leq c_2 \abs{z} \qq \tx{ for all } z\in \Csp, \abs{z} \leq C_1
\end{align*}
where $c_2$ depends only on $C_1$. Hence via a similar calculation from step 1 and using \eqref{eq:systemtwoone} we see that if $\g \in \Linfty([0,T], H^{s}(\Rsp))$ then $(\Zap -1, \frac{1}{\Zap} -1) \in \Linfty([0,T], H^{s}(\Rsp)\times H^{s}(\Rsp))$. We also observe that in this case we have $\log(\Psizp) = K_{-y} \conv (i(\Id + \Hil)\g)$ and hence $\log(\Psizp)$ is well defined. Hence we easily obtain 
\begin{align*}
\lim_{c \to \infty} \sup_{\abs{\zp}\geq c}\nobrac{\abs{\Psizp(\zp) - 1}}  = 0
\qquad \tx{ and } \quad \Psizp(\zp) \neq 0 \quad \tx{ for all } \zp \in \Pminus 
\end{align*}
Note that the condition $\lim_{z\to \infty}\pt\Psi(z,t) = 0$ is a choice that we can make and is not a priori implied by the system \eqref{eq:Dtg}. Hence we simply make this choice.  

\item We again have $\cvar  = e^{-i\Hil \g} = \frac{1}{\Zapabs}$ and $\bstar = \bvar = -i\Hil\brac{\frac{1}{\Zapabs^2}}$. Hence $(\pt + \bstar\pap) = (\pt + \bvar\pap) = \Dt$. 

\item As an intermediate computation, we compute the following
\begingroup
\allowdisplaybreaks
\begin{align*}
& -i\cbrac{\sqbrac{\bvar,\Hil}\pap g + \sqbrac{\frac{1}{\Zapabs^2},\Hil}\papabs g} \\
& = -\Real\cbrac{\sqbrac{\bvar,\Hil}\pap (f + ig) + \sqbrac{\frac{1}{\Zapabs^2},\Hil}\papabs (f + ig)} \\
& = -\Real(\Id - \Hil)\cbrac{-i\Hil\brac{\frac{1}{\Zapabs^2}}\pap(f + ig) + \frac{i}{\Zapabs^2}\pap(f + ig)} \\
& = -\Real(\Id - \Hil)\cbrac{i\pap(f + ig)\cbrac{\frac{1}{\Zapabs^2} - \Hil\brac{\frac{1}{\Zapabs^2}}}} \\
& = -\Real(\Id - \Hil)\cbrac{i\pap(f + ig)\cbrac{\frac{2}{\Zapabs^2} - (\Id + \Hil)\brac{\frac{1}{\Zapabs^2}}}} \\
& = -2\Real(\Id - \Hil)\cbrac{i\pap(f + ig)\frac{1}{\Zapabs^2}} \\
& = 2\Real(\Id - \Hil)\cbrac{\frac{i}{\Zapbar}\pap\frac{1}{\Zap}} \\
& = -2\Imag(\Id - \Hil)\cbrac{\frac{1}{\Zapbar}\pap\frac{1}{\Zap}} \\
& = \Bone
\end{align*}
\endgroup

\item We can now derive the main equation of \eqref{eq:main}. We see that
\begingroup
\allowdisplaybreaks
\begin{align*}
& \Dt \frac{1}{\Zap} \\
& = \Dt e^{-i(\Id + \Hil)g} \\
& = -\frac{i}{\Zap}\Dt(\Id + \Hil)g \\
& = -\frac{i}{\Zap}\cbrac{\sqbrac{\bvar,\Hil}\pap g - (\Id + \Hil)\cbrac{\frac{1}{\Zapabs^2}\papabs g}} \\
& = -\frac{i}{\Zap}\cbrac{\sqbrac{\bvar,\Hil}\pap g + \sqbrac{\frac{1}{\Zapabs^2}, \Hil}\papabs g - \frac{1}{\Zapabs^2}\pap(f + ig)} \\
& = -\frac{i}{\Zapabs^2}\pap\frac{1}{\Zap} - \frac{i}{\Zap}\cbrac{\sqbrac{\bvar,\Hil}\pap g + \sqbrac{\frac{1}{\Zapabs^2}, \Hil}\papabs g} \\
& = -\frac{i}{\Zapabs^2}\pap\frac{1}{\Zap} + \frac{\Bone}{\Zap}
\end{align*}
\endgroup
\end{enumerate}
This proves the equivalence between the two systems. 

\end{proof}

We again note the equations we derived for $g$ and $f$ in the above derivation:
\begin{align*}
    \Dt \g + \frac{1}{\Zapabs^2}\papabs \g = 0
\end{align*}
and also 
\begin{align*}
    \Dt \f + \frac{1}{\Zapabs^2}\papabs \f = -\Bone
\end{align*}
As a direct consequence, we see that for regular enough solutions we have from \lemref{lem:papabs} and $\Bone \geq 0$ that
\begin{align*}
    \inf_{\ap \in \Rsp} g (\ap,0) \leq \inf_{\ap \in \Rsp} g (\ap,t) \leq \sup_{\ap \in \Rsp} g (\ap,t) \leq \sup_{\ap \in \Rsp} g (\ap,0)
\end{align*}
and
\begin{align*}
    \sup_{\ap \in \Rsp} f (\ap,t) \leq \sup_{\ap \in \Rsp} f (\ap,0)
\end{align*}
Recall that $g$ is the angle of the interface and hence is related to the slope of the interface by the relation $\tan(g) = \grad h$, if $h$ is the height function. Similarly $f$ is related to the Taylor sign condition by the relation $- \frac{\partial P}{\partial n} = \frac{1}{\Zapabs} = e^{-f}$.  These estimates are well known (see \cite{AlMeSm20} and its references) and this gives another proof of them. We now proceed to derive a new identity.

Now from the structure of the above equations, it is very natural to consider the energy
\begin{align*}
\int \Zapabs^2 F^2 \diff \ap
\end{align*}
where $F$ is some function. We now see from \eqref{form:DtZapabs} and \lemref{lem:timederiv} that
\begin{align*}
\frac{\diff}{\diff t} \int \Zapabs^2 F^2 \diff \ap 
& = \int \Zapabs^2 F^2\cbrac{2\Real(\Dap\Zt) - \bvarap} \diff \ap + 2\int \Zapabs^2 F \Dt F \diff \ap
\end{align*}
Now using \eqref{eq:DarcyLawA1} we see that $\Dap\Zt = i\Dap\frac{1}{\Zapbar} $. Hence from \eqref{eq:bvarap} we get
\begin{align*}
\frac{\diff}{\diff t} \int \Zapabs^2 F^2 \diff \ap 
& = - \int \Zapabs^2 F^2\Bone \diff \ap + 2\int \Zapabs^2 F \Dt F \diff \ap
\end{align*}
Therefore
\begin{align*}
 \frac{\diff}{\diff t} \int \Zapabs^2 F^2 \diff \ap + \int \Zapabs^2 F^2\Bone \diff \ap =  2\int \Zapabs^2 F \Dt F \diff \ap
\end{align*}
Now putting $F = \g$ in the above identity, we obtain the following interesting identity:
\begin{align}\label{eq:consHhalfg}
\frac{\diff}{\diff t} \int \Zapabs^2 \g^2 \diff \ap + \int \Zapabs^2 \g^2\Bone \diff \ap + 2 \norm[\Hhalf]{g}^2 = 0.
\end{align}
Note that by integrating the above identity in time we get
\begin{align*}
\int \Zapabs^2(\cdot,T)g^2(\cdot,T) \diff \ap + \int_0^T \int \Zapabs^2 g^2 \Bone \diff\ap\diff s &+ 2 \int_0^T \norm[\Hhalf]{g}^2(s) \diff s \\
&\hspace{10mm}= \int \Zapabs^2(\cdot,0)g^2(\cdot,0) \diff \ap
\end{align*}

Observe that if the interface has a corner of angle $\nu \pi$, then $\Zap (\ap) \sim (\ap)^{\nu-1}$ and hence the right hand side is finite only if  $\nu > \half$, i.e. the angle is bigger than $\pi/2$. In this case we see that $\int_0^{\infty} \norm[\Hhalf]{g}^2(s) \diff s < \infty$ and hence for a.e. $t \in [0,\infty)$, the function $g(\cdot,t) \in \Hhalf$. As $g$ represents the angle of the interface, this means that for almost every $t \in [0,\infty)$, the interface cannot have a sharp corner. Note that instantaneous smoothing of angles bigger than $\pi/2$ has already been proven in the work \cite{ChJeKi09}, however the above identity gives a different perspective on this phenomenon and could be useful for other purposes as well. 

\begin{rem}
While working on this problem, we were informed (via private communication) that Thomas Alazard independently and almost simultaneously obtained a very similar though not identical identity to \eqref{eq:consHhalfg}. We find it remarkable that that there are different versions of this identity and this gives hope that
there are more undiscovered conservation laws in this problem. 
\end{rem}

\section{Existence and Rigidity}\label{sec:existence}

This section is dedicated to completing the proof of Theorem \ref{thm:mainwellposed} and Theorem \ref{thm:mainangle}. We first show the existence of solutions in Sobolev spaces and then using these, we construct our solutions that can have singular points. Finally, we outline the proof of rigidity of interface corners.

\subsection{Existence in Sobolev space}

We begin by proving existence of solutions in Sobolev spaces using the equivalent formulation \eqref{eq:Dtg}. First, we need to mollify the formulation. To do so, let $\phi$ be a smooth bump function satisfying $\phi(\ap) \geq 0$ for all $\ap\in\Rsp$ and $\int \phi(\ap) \diff\ap = 1$.  For $s>0$ let $\phi_s(\ap) = \frac{1}{s}\phi\brac*[\big]{\frac{\ap}{s}}$. Consider the smoothing operator $\Jdel$ defined via $\Jdel(f) = f * \phi_\delta$ for $\del>0$ and $\Jdel(f) = f$ for $\del =0$. The following lemma is from \cite{Ag21}
\begin{lemma}\label{lem:commutatordel}
Let $f,g \in \Scalsp(\Rsp)$ and let $0 < \del,\del_1,\del_2 \leq 1$. Then we have
\begin{enumerate}
\item $\norm[2]{J_{\del_1}(f) - J_{\del_2}f} \lesssim \max\{\del_1^s,\del_2^s\}\norm[H^s]{f}$ \quad for $0< s\leq 1$
\item $\norm[2]{\sqbrac{\Jdel,f}\pap g} \lesssim \del\norm[H^2]{f}\norm[2]{g_\ap}$ 
\item $\norm[2]{\sqbrac{\Jdel,f}\pap g} \lesssim \norm[\infty]{f_\ap}\norm[2]{g}$ 
\item $\norm[2]{\papabs^\half\sqbrac{\Jdel,f}\pap g} \lesssim \norm[\infty]{f_\ap}\norm[\Hhalf]{g}$
\end{enumerate}
where the constants in the estimates are independent of $\del$.
\end{lemma}

Consider the following system in the variable $\g^{\del,\ep}$

\begin{align}\label{eq:approxeqg}
    \pt {g^{\del,\ep}} -\ep\Delta{g^{\del,\ep}} = \phi_{\del}\conv\brac*[\Bigg]{-\bvar^{\del,\ep}\pap{g^{\del,\ep}} - \frac{1}{\abs{\Zap^{\del,\ep}}^2}\papabs{g^{\del,\ep}} }
\end{align}
where 
\begin{align*}
    \Zap^{\del,\ep} = e^{i(\Id + \Hil)g^{\del,\ep}}, \quad \bvar^{\del,\ep} = -i\Hil\brac{\frac{1}{\abs{Z^{\del,\ep}}^{2}}}
\end{align*}
and $\phi_{\delta}$ is a smoothing mollifier. For convenience, we will denote $f^{\del,\ep} = i\Hil g^{\del,\ep}$. Defining the energy
\begin{align*}
    E_n(t) = \norm*[\big][H^n]{g^{\delta,\ep}}(t)^{2}
\end{align*}
we have the following two estimates on $E_{n}(t)$:
\begin{prop}
Let $n\geq 2$ and let $g^{\delta,\ep}(\cdot,t)$ be a smooth solution to \eqref{eq:approxeqg}. Then for $\ep>0$ and $0\leq \del \leq 1$, we have 
\begin{align}\label{eq:Enbound}
    \frac{\diff}{\diff t} E_n(t) \leq C_{\ep}(E_2(t)) E_n(t)
\end{align}
and for $0 \leq \ep \leq 1$ and $\del = 0$, we have 
\begin{align}\label{eq:Enboundnoepsilon}
    \frac{\diff}{\diff t} E_n(t) \leq C(E_2(t)) E_n(t).
\end{align}
\end{prop}
\begin{proof}
 We will only consider the case $\epsilon = 0$ and $\delta = 0$, and the estimates \eqref{eq:Enbound} and \eqref{eq:Enboundnoepsilon} will follow easily from simple modifications of the argument. For simplicity of notation, we will simply drop the superscripts, e.g. $g^{\del,\ep} = g$, $\Zap^{\del,\ep} = \Zap$, $f^{\del,\ep} = f$. When $\delta > 0$ and $\epsilon > 0$, the linear Laplacian term with coefficient $\epsilon$ controls the evolution independent of $\delta >0$. In the case $\delta = 0$, $\epsilon \geq 0$, we can apply the argument in the case $\epsilon = 0$ and $\delta = 0$ to control the energy independent of $\epsilon$ since $\epsilon > 0$ only gives a smoothing term. In the following, the symbol $\lesssim_{2}$ will be used if the implicit constant depends only on $E_{2}(t)$.
 
 Differentiating the highest order energy term in time,
\begin{align*}
\frac{1}{2}\frac{d}{dt}\norm[\dot{H}^n]{g}(t)^{2} =  \int (\pap^{n}g) \pap^{n}\brac{-\bvar \pap g- \frac{1}{\abs{\Zap}^2}\papabs g} \diff \ap
\end{align*}
First, we look at the term
\begin{align*}
\int (\pap^{n}g)\pap^{n}\brac{\frac{1}{\Zapabs^{2}}\papabs g }\diff \ap 
&=  \sum_{m=0}^{n} {\binom{n}{m}}\int (\pap^{n}g)\brac{\pap^{m}\frac{1}{\Zapabs^{2}}}\pap^{n-m}\papabs g \diff \ap.
\end{align*}
For $m=0$, the term is
\begin{align}
\int \frac{1}{\Zapabs^{2}}(\pap^{n}g) \papabs\pap^{n} g \diff \ap 
& =  \int \frac{1}{\Zapabs}(\pap^{n}g)\papabs\brac{\frac{1}{\Zapabs}\pap^{n} g} \diff \ap \label{eq:mzeroterm} \\ 
& \quad + \int \frac{1}{\Zapabs}(\pap^{n} g)  \sqbrac{\frac{1}{\Zapabs},\papabs}\pap^{n}g \diff \ap \nonumber
\end{align}
We extract the term
\begin{align*}
    \int \frac{1}{\Zapabs}\pap^{n}g\papabs\brac{\frac{1}{\Zapabs}\pap^{n} g} \diff \ap & = \norm[\Hhalf]{\frac{1}{\Zapabs}\pap^{n}g}^{2}
\end{align*}
For the other term, note that
\begin{align*}
\sqbrac{\frac{1}{\Zapabs},\papabs}\pap^n g &= i\brac{\sqbrac{\frac{1}{\Zapabs},\Hil}\pap^{n+1}g - \Hil\cbrac{\brac{\pap\frac{1}{\Zapabs}}\pap^n g}}
\end{align*}
By Proposition \ref{prop:commutator},
\begin{align*}
\norm[2]{\sqbrac{\Hil,\frac{1}{\Zapabs}}\pap^{n+1} g} \lesssim \norm[\infty]{\pap \frac{1}{\Zap}}\norm[2]{\pap^n g}
\end{align*}
and 
\begin{align*}
\norm[2]{\Hil\cbrac{\brac{\pap\frac{1}{\Zap}}\pap^n g}} \lesssim \norm[\infty]{\pap \frac{1}{\Zap}}\norm[2]{\pap^{n} g}.
\end{align*}
For the above norms of $\pap\frac{1}{\Zap}$ and the remaining terms, bounds on $1/\abs{\Zap}^{2}$ are needed.
\begin{align*}
\norm[\infty]{\frac{1}{\Zapabs}} \leq \norm[\infty]{\frac{1}{\abs{\Zap}}- 1} + 1 = \norm[\infty]{e^{-f}- 1} + 1
\end{align*}
Now, 
\begin{align*}
\norm[\infty]{e^{-f}-1} &\leq \sum_{j=1}^{\infty} \frac{1}{j!} \norm[\infty]{f}^{j} \lesssim \sum_{j=1}^{\infty} \frac{1}{j!} \norm[H^{1}]{f}^{j} = \sum_{j=1}^{\infty} \frac{1}{j!} \norm[H^{1}]{g}^{j}.
\end{align*}
Similarly,
\begin{align*}
\norm[\infty]{\pap\frac{1}{\Zap}} = \norm[\infty]{\pap e^{-i(\Id + \Hil)g}} = \norm[\infty]{ e^{-i(\Id + \Hil)g} \pap\brac{g + \Hil g}} \lesssim \norm[\infty]{e^{-f}}\norm[H^1]{\pap g}.
\end{align*}
Hence, we have 
\begin{align*}
    \norm[\infty]{\frac{1}{\Zapabs}} + \norm[\infty]{\pap\frac{1}{\Zap}} \leq C(E_{2}(t))
\end{align*}
and thus,
\begin{align*}
   \abs{\int \frac{1}{\Zapabs}(\pap^{n} g)  \sqbrac{\frac{1}{\Zapabs},\papabs}\pap^{n}g \diff \ap} \lesssim_{2} \norm[2]{\pap^n g}^{2}.
\end{align*}
For $1 \leq m \leq  n$ we see from \lemref{lem:interpolationmult} that
\begin{align*}
& \abs{\int (\pap^{n}g)\brac{\pap^{m}\frac{1}{\Zapabs^{2}}}\pap^{n-m}\papabs g \diff \ap} \\
& \lesssim \norm[H^n]{g}\norm[2]{\brac{\pap^{m}\frac{1}{\Zapabs^{2}}}\pap^{n-m}\papabs g} \\
& \lesssim \norm[H^n]{g}\cbrac{ \sum_{\substack{j_{1}+\ldots+j_{k}=m \\ j_1, \cdots, j_k \geq 1}} \norm[2]{\frac{1}{\Zapabs^{2}}\pap^{j_{1}}f\cdots\pap^{j_{k}}f \pap^{n-m+1}f }} \\
& \lesssim_{2} \norm[H^n]{g}\cbrac{ \sum_{\substack{j_{1}+\ldots+j_{k}=m \\ j_1, \cdots, j_k \geq 1}} \norm[2]{\pap^{j_{1}}f\cdots\pap^{j_{k}}f \pap^{n-m+1}f }} \\
& \lesssim_{2} \norm[H^n]{g}^2
\end{align*}

Thus, in total,
\begin{align*}
-\int (\pap^{n}g)\pap^{n}\brac{\frac{1}{\Zapabs^{2}}\papabs g }\diff \ap &\leq -\norm[\Hhalf]{\frac{1}{\Zapabs}\pap^n g}^{2} + C(E_{2}(t))E_{n}(t)
\end{align*}
Next, we consider the term
\begin{align*}
   \int (\pap^{n}g) \pap^{n}\brac{\bvar \pap g} \diff \ap =  \sum_{k=0}^{n} {\binom{n}{k}}  \int (\pap^{n}g) (\pap^{k}\bvar) (\pap^{n-k+1} g) \diff \ap.
\end{align*}
First, consider $k=0$. Then, integrating by parts, we obtain
\begin{align*}
 \abs{\int (\pap^{n}g) \bvar (\pap^{n+1} g) \diff \ap}  \lesssim \norm[\infty]{\bvarap}\norm[2]{\pap^n g}^2 \lesssim_{2}\norm[H^n]{g}^{2}
\end{align*}
For the remaining terms, we bound
\begin{align*}
\norm[H^m]{\bvar} = \norm[H^m]{\frac{1}{\abs{\Zap}^{2}} - 1} \lesssim_{2} \norm[H^m]{g}
\end{align*}
and apply similar arguments as earlier. For the evolution of the lower order energy term, $\norm[2]{g}$, one can apply the above techniques with $n=0$ and it is easily controlled.
Hence, combining the above terms, we obtain both estimates \eqref{eq:Enbound} and \eqref{eq:Enboundnoepsilon}.
\end{proof}

Next, we have the difference energy:
\begin{align*}
    \EDelta(t) = \norm[H^1]{{g^{\del,\ep}} - {g^{\del',\ep'}}}^2(t)
\end{align*}
Using the shorthand $g = g^{\del,\ep}$ and $g' = g^{\del', \ep'}$and $\mathcal{N}(g)$ to denote the nonlinear right hand side of \eqref{eq:approxeqg} except the mollifier, we have the difference equation
\begin{align*}
    \partial_t  (g-g') = (\ep-\ep')\Delta g + \ep' \Delta(g-g') + (\phi_{\del}-\phi_{\del'}) \ast \brac{\mathcal{N}(g)} + \phi_{\del'} \ast \brac{\mathcal{N}(g)- \mathcal{N}(g')}  
\end{align*}
Let us now consider the time evolution of $\EDelta(t)$. 

\begin{prop}\label{prop:H1Deltaprop1}
Let $\delta,\delta' \geq 0$ and $\ep,\ep' \geq 0$. Say $g^{\del,\ep}, g^{\del',\ep'}$ are solutions of \eqref{eq:approxeqg} such that
\begin{align*}
    \sup_{t\in\sqbrac{0,T}} \norm*[\big][H^2]{g^{\del,\ep}}(t) + \norm*[\big][H^2]{g^{\del',\ep'}}(t) \leq C_{2} < \infty.
\end{align*}
For $t\in \sqbrac{0,T}$ if $\delta, \delta' \geq 0$ and $\ep = \ep' >0$, then we have 
\begin{align}
   \frac{d}{dt}\EDelta(t) \leq  \frac{C(C_{2})}{\ep}\EDelta(t) + C(C_{2})\max\cbrac{\del^{\frac{1}{2}},\del'^{\frac{1}{2}}}
\end{align}
and if $\delta = \delta' = 0$ and $\ep, \ep' \geq 0$, then we have
\begin{align*}
\frac{d}{dt}E_{\Delta}(t) \leq C(C_{2})\EDelta(t) + C(C_{2})\abs{\ep-\ep'}
\end{align*}
for a constant $C(C_2)$ depending on $C_{2}$.
\end{prop}
\begin{proof}
We will use the following notation: If $a \leq C(C_2)b$ for a constant $C(C_2)$ depending on $C_2$, we will write $a \lesssim_{2} b$. First note that it is easy to see that
\begin{align*}
\norm*[\big][H^1]{b^{\delta,\ep} - b^{\delta',\ep'}} + \norm[H^1]{\frac{1}{\abs*[\big]{\Zap^{\delta,\ep}}^2} - \frac{1}{\abs*[\big]{\Zap^{\delta',\ep'}}^2} } \lesssim_{2} \norm[H^1]{g - g'}
\end{align*}

Now consider any $\del,\del' \geq 0$ and $\ep,\ep' > 0$.   We now control the time derivative of the highest order term in $\EDelta(t)$
\begin{align*}
\frac{1}{2}\frac{d}{dt}\norm[\dot{H}^1]{g - g'}^2 
& =  -\int \Delta(g - g')\cbrac{(\ep - \ep')\Delta g + \ep'\Delta(g-g')}\diff \ap \\
& \quad -\int \Delta(g - g')\cbrac{(\phi_{\delta} - \phi_{\delta'})\conv \mathcal{N}(g) + \phi_{\delta'}\conv(\mathcal{N}(g) - \mathcal{N}(g'))}\diff \ap 
\end{align*} 
We see that
\begin{align*}
  -\int \Delta\brac{ g -  g'} \brac{\ep' \Delta (g-g')} &= -\ep' \norm[\dot{H}^{2}]{g-g'}
\end{align*}
Also,
\begin{align*}
\abs{\int \Delta\brac{ g - g'}\brac{\ep-\ep'}\Delta g \diff \ap} \lesssim_{2} \abs{\ep-\ep'}.
\end{align*}
Meanwhile it is easy to see that,
\begin{align*}
&\abs{\int \Delta\brac{ g -  g'} \brac{ \phi_{\del'}\conv \brac{\mathcal{N}(g) - \mathcal{N}(g')}} \diff \ap }\\
& \lesssim_2 \norm[\dot{H}^2]{g-g'}\norm[L^1]{\phi_{\del'}}\norm[H^1]{g - g'}\\
& \leq \frac{C(C_{2})}{\ep'}\EDelta + \frac{\ep'}{4} \norm[\dot{H}^2]{g-g'}
\end{align*}
and hence the higher order term is absorbed by the linear decay term. Next by Lemma \ref{lem:commutatordel},
\begin{align*}
\abs{\int \Delta\brac{g -  g'} \brac{(\phi_{\del}-\phi_{\del'}) \ast \brac{\mathcal{N}(g)}} \diff \ap}
&\lesssim_{2} \norm[H^{\frac{1}{2}}]{\mathcal{N}(g)} \max\cbrac{\del^{\frac{1}{2}},\del'^{\frac{1}{2}}}\\
&\lesssim_{2}\max\cbrac{\del^{\frac{1}{2}},\del'^{\frac{1}{2}}}.
\end{align*}
Carrying out a similar calculation for the time derivative of $\norm[2]{g - g'}^2$ we obtain
\begin{align}\label{eq:deltaestepdel}
    \frac{d}{dt}\EDelta \leq -\frac{\ep'}{2}\norm[\dot{H}^{2}]{g-g'}^{2} + \frac{C(C_{2})}{\ep'}\EDelta + C(C_{2}) \brac{\abs{\ep-\ep'} + \max\cbrac{\del^{\frac{1}{2}},\del'^{\frac{1}{2}}}}.
\end{align}
Hence, fixing $\ep = \ep' > 0$, we conclude the proof of the first estimate.

Now we assume that $\delta = \delta' = 0$. We have for $\ep, \ep' \geq 0$,
\begin{align*}
\half\frac{d}{dt} \norm[\dot{H}^1]{g - g'}^2 = -\int \Delta(g - g')\cbrac{(\ep - \ep')\Delta g + \ep'\Delta(g - g') + (\mathcal{N}(g) - \mathcal{N}(g'))} \diff \ap
\end{align*}
The first and second terms are easily controlled and so we now focus on the third term. Observe that
\begin{align*}
\mathcal{N}(g) - \mathcal{N}(g') 
& = -(\bvar^{\ep} - \bvar^{\ep'})\pap g^\ep - \bvar^{\ep'}\pap(g^{\ep} - g^{\ep'}) -\brac{\frac{1}{\abs*[\big]{\Zap^{\ep}}^2} - \frac{1}{\abs*[\big]{\Zap^{\ep'}}^2}}\papabs g^{\ep} \\
& \quad - \frac{1}{\abs*[\big]{\Zap^{\ep'}}^2}\papabs(g^{\ep} - g^{\ep'})
\end{align*}
The terms corresponding to the first and third term are easily controlled. For the second term we observe that by doing integration by parts 
\begin{align*}
\abs{\int \Delta(g - g')\bvar^{\ep'}\pap(g^{\ep} - g^{\ep'})\diff \ap } \lesssim \norm*[\big][\infty]{\bvar^{\ep'}_{\ap}}\norm[H^1]{g - g'}^2 \lesssim_{2} \EDelta
\end{align*}
For the last term we see that
\begin{align*}
\int \Delta(g - g') \frac{1}{\abs*[\big]{\Zap^{\ep'}}^2}\papabs(g^{\ep} - g^{\ep'}) \diff \ap & = - \int \pap(g - g') \brac{\pap\frac{1}{\abs*[\big]{\Zap^{\ep'}}^2}}\papabs(g^{\ep} - g^{\ep'}) \diff \ap \\
& \quad - \int \pap(g - g') \nobrac{\frac{1}{\abs*[\big]{\Zap^{\ep'}}^2}}\pap\papabs(g^{\ep} - g^{\ep'}) \diff \ap
\end{align*}
The first integral is easily controlled and for the second, we see that this is exactly similar to the term \eqref{eq:mzeroterm} we see in the $E_n(t)$ energy estimates and is estimated similarly
\begin{align*}
& - \int \pap(g - g') \nobrac{\frac{1}{\abs*[\big]{\Zap^{\ep'}}^2}}\pap\papabs(g^{\ep} - g^{\ep'}) \diff \ap\\ 
& \leq  -\norm[\Hhalf]{\frac{1}{\abs*[\big]{\Zap^{\epsilon'}}}\pap\brac*{g^{\ep} - g^{\ep'}}}^{2} + C\norm[\infty]{\pap\frac{1}{\Zap^{\epsilon}}}\norm[2]{\pap\brac{g-g'}}^{2}
\end{align*}
where $C>0$ is a universal constant. Hence this term is also controlled. Now the lower order terms are also easily controlled and hence we obtain
\begin{align*}
\frac{d}{dt} E_{\Delta} \leq  C(C_{2})\abs{\ep-\ep'}  + C(C_{2})E_{\Delta}.
\end{align*}
\end{proof}

Hence, by  \propref{prop:H1Deltaprop1}, we first see that the sequence $g^{\del,\ep}$ is Cauchy in $H^{1}$ for $\delta \geq 0$ (keeping $\ep>0$ fixed). Hence, taking the limit, we can consider $\delta = 0$, $\epsilon \geq 0$. Then, we see that $g^{0,\ep}$ is Cauchy in $H^{1}$ for $\ep \geq 0$ and we take the limit. We get the following well-posedness theorem in Sobolev spaces:

\begin{thm}\label{thm:existencesobolev}
We have the following
\begin{enumerate}
\item Let $s\geq 2$ and let $\g(0) \in H^{s}(\Rsp)$. Then there exists a time $T>0$ so that the initial value problem to \eqref{eq:Dtg} has a unique solution $\g \in \Linfty([0,T], H^{s}(\Rsp))$. Moreover if $T_{max}$ is the maximum time of existence then either $T_{max}=\infty$ or $T_{max} < \infty$ with
\begin{align*}
\sup_{t\in [0,T_{max})}\nobrac{ \norm[H^{2}]{\g(\cdot,t)}} = \infty
\end{align*}
\item If $\g_1(t)$ and $\g_2(t)$ are two solutions of \eqref{eq:Dtg} in $[0,T]$ with
\begin{align*}
\sup_{t\in [0,T)} \cbrac{\norm[H^{2}]{\g_1(\cdot,t)}  + \norm[H^{2}]{\g_2(\cdot,t)}  } = M <\infty
\end{align*}
Then there is constant $C(M,T)>0$ depending only on $M$ and $T$ such that
\begin{align*}
\sup_{t\in [0,T)}\cbrac{ \norm[H^{1}]{\g_1(\cdot,t) - \g_2(\cdot,t)}} \leq  C(M,T)\cbrac{\norm[H^{1}]{\g_1(\cdot,0) - \g_2(\cdot,0)}}
\end{align*}
\end{enumerate}
\end{thm}
By the equivalence of the two formulations given in Lemma \eqref{lem:systemequiv}, we obtain 
\begin{cor}\label{cor:existencesobolev}
We have the following
\begin{enumerate}
\item Let $s\geq 2$ and let the initial data $\Z(\cdot,0)$ satisfy $(\Zap-1,\frac{1}{\Zap} - 1)(0) \in H^{s}(\Rsp)\times H^{s}(\Rsp)$. Then there exists a time $T>0$ so that the initial value problem to \eqref{eq:main} has a unique solution $\Z(\cdot, t)$ satisfying $(\Zap-1,\frac{1}{\Zap} - 1) \in \Linfty([0,T], H^{s}(\Rsp)\times H^{s}(\Rsp))$. Moreover if $T_{max}$ is the maximum time of existence then either $T_{max}=\infty$ or $T_{max} < \infty$ with
\begin{align*}
\sup_{t\in [0,T_{max})}\cbrac{ \norm[H^{2}]{\Zap - 1}(t) + \norm[H^{2}]{\frac{1}{\Zap} - 1 }(t)} = \infty
\end{align*}
\item Let $\Z^1(\cdot,t)$ and $\Z^2(\cdot,t)$ be two solutions of \eqref{eq:main} in $[0,T]$ with 
\begin{align*}
\sup_{t\in [0,T_{max})}\cbrac{ \norm[H^{2}]{\nobrac{\Zap^i - 1}}(t) + \norm[H^{2}]{\nobrac{\frac{1}{\Zap^i} - 1}}(t) } \leq M <\infty
\end{align*}
for both $i=1,2$ for some $M>0$. Then there is constant $C(M,T)>0$ depending only on $M$ and $T$ such that
\begin{align*}
& \sup_{t\in [0,T)}\cbrac{ \norm[H^{1}]{\Zap^1 - \Zap^2}(t) + \norm[H^{1}]{\frac{1}{\Zap^1} - \frac{1}{\Zap^2}}(t)} \\
& \leq  C(M,T)\cbrac{ \norm[H^{1}]{\Zap^1 - \Zap^2}(0) + \norm[H^{1}]{\frac{1}{\Zap^1} - \frac{1}{\Zap^2}}(0)}
\end{align*}
\end{enumerate}
\end{cor}
 
\subsection{Existence of rough solutions}

We can now prove Theorem \ref{thm:mainwellposed} in this section. First, we need to state some useful lemmas and notation that will be used in the proof.

For functions $g, g_n: X \to \Csp$, we write  ``$g_n \Ra g$ on $X$" to mean that $g_n$ converges to $g$ uniformly on compact subsets of $X$. The following two lemmas are from \cite{Wu19} and follow from the Arzela-Ascoli Theorem.
\begin{lemma}\label{lem:Arzela1}
Say $1 < p \leq \infty$ and $\cbrac{f_{n}}$ is a sequence of smooth functions on $\mathbb{R}\times \sqbrac{0,T}$. If there exists a single constant $C > 0$ such that
\begin{align*}
    \sup_{\sqbrac{0,T}}\norm[\infty]{f_n(t)} + \sup_{\sqbrac{0,T}}\norm[p]{\partial_x f_n(t)} + \sup_{\sqbrac{0,T}}\norm[\infty]{\partial_t f_n(t)} \leq C
\end{align*}
for every $n\in\mathbb{N}$, then there exists a continuous and bounded function $f$ on $\mathbb{R}\times \sqbrac{0,T}$ and a subsequence $f_{n_{j}} \Ra f$  on $\mathbb{R}\times \sqbrac{0,T}$.
\end{lemma}

\begin{lemma}\label{lem:Arzela2}
Suppose $f_n \Ra f$ on $\mathbb{R}\times \sqbrac{0,T}$ and there is a single constant $C > 0$ such that $\norm[L^\infty\brac{\mathbb{R}\times \sqbrac{0,T}}]{f_n} \leq C$ for all $n\in\mathbb{N}$. Then, $K_{y'} \ast f_n \Ra K_{y'} \ast f$ on $\Pminusbar\times \sqbrac{0,T}$.
\end{lemma}

Next, we prove a lemma that gives us bounds that will be used to prove uniform bounds on various terms in the existence argument.

\begin{lemma}\label{lem:universalF}
Let $Z(\cdot,t)$ be a smooth solution to \eqref{eq:main} in $[0,T]$. Define 
\begin{align*}
    J & = 1 + T + \norm[2]{\frac{1}{\Zap}(\cdot,0) - 1} + \sup_{t \in [0,T]}\Mcal(t) \\
    L & = 1 + T + \norm[\infty]{\Zap(\cdot,0)} + \norm[2]{\frac{1}{\Zap}(\cdot,0) - 1} + \sup_{t \in [0,T]}\Mcal(t)
\end{align*}
Then there exists a universal function $F:[0,\infty) \to [0,\infty)$ so that we have for all $t \in [0,T]$
\begin{align*}
\norm[H^1]{\frac{1}{\Zap} - 1}(t) \leq F(J) 
\end{align*}
and 
\begin{align*}
    \norm[H^2]{\Zap - 1}(t) + \norm[H^2]{\frac{1}{\Zap} - 1} \leq F(L)
\end{align*}
\end{lemma}
\begin{proof}
For simplicity we will write $A \lesssim_J B$ whenever $A \leq F(J) B$ and similar meaning for $A \lesssim_L B$. We first observe from \eqref{eq:hapzapratio} that for any $t \in [0,T]$
\begin{align*}
\norm[\infty]{\Zap}(t) & \leq \norm[\infty]{\Zap}(0)\exp\cbrac{\int_0^t \brac{\norm[\infty]{\Dap\Zt}(s) + \norm[\infty]{\bvarap}(s)}\diff s} \lesssim_L 1
\end{align*}
Now if we let 
\begin{align*}
f(t) = \norm[2]{\frac{1}{\Zap} - 1}^2(t) + 1
\end{align*}
then we see from \lemref{lem:timederiv}, \eqref{eq:main} and \propref{prop:commutator} that
\begin{align*}
\frac{\diff f(t)}{\diff t} & \lesssim \norm[\infty]{\bvarap}f(t) + f(t)^{\half} \norm[2]{\Dt\frac{1}{\Zap}} \\
& \lesssim \norm[\infty]{\bvarap}f(t) + f(t)^{\half}\norm[\infty]{\frac{1}{\Zap}}^2\norm[2]{\pap\frac{1}{\Zap}} \\
& \lesssim \norm[\infty]{\bvarap}f(t) + f(t)^{\half}\norm[2]{\pap\frac{1}{\Zap}}\brac{1 + \norm[2]{\frac{1}{\Zap} - 1}\norm[2]{\pap\frac{1}{\Zap}}} \\
& \lesssim_J f(t)
\end{align*}
Hence we see that $f(t) \lesssim_J 1$ for all $t \in [0,T]$, thereby proving the first estimate. Now using the bound for $\norm[\infty]{\Zap}(t)$ we see that for all $t \in [0,T]$
\begin{align*}
\norm[2]{\frac{1}{\Zap} - 1}(t) + \norm[2]{\Zap - 1}(t) \lesssim_L 1
\end{align*}
Now from the definition of $\Mcal(t)$ and the bound \eqref{eq:DapZInfinity} one can also easily see that
\begin{align*}
\norm[2]{\pap\frac{1}{\Zap}} + \norm[2]{\frac{1}{\Zap^2}\pap\Zap} + \norm[2]{\frac{1}{\Zap^2}\pap^2\frac{1}{\Zap}} + \norm[2]{\frac{1}{\Zap^4}\pap^2\Zap} \lesssim \Mcal
\end{align*}
Hence using the bound for $\norm[\infty]{\Zap}$ we see that
\begin{align*}
\norm[H^1]{\pap\frac{1}{\Zap}} + \norm[H^1]{\pap\Zap} \lesssim_L 1
\end{align*}
thereby proving the lemma. 
\end{proof}

\begin{proof}[Proof of \thmref{thm:mainwellposed}]

We will now present the remaining proof of Theorem \ref{thm:mainwellposed}. The proof is divided into 5 steps. The choice of data mollification and the proof for uniform convergence on compact subsets of various quantities from Step 1 - Step 4 are similar to \cite{Wu19}. Then, in Step 5, we show the solution satisfies the system in the sense of Definition \ref{def:solution}. In particular, we perform the calculations necessary to show $\frac{1}{\Zap}$ satisfies \eqref{eq:main} in the sense of distributions.

First, we define, for $0< \epsilon \leq 1$, the mollified data
\begin{align*}
\Z^\ep (\ap,0) = \Psi(\ap-\ep i,0),\quad\Psi^\ep(\zp,0) = \Psi(\zp - \ep i,0), \quad  \h^\ep(\al,0) = \al 
\end{align*}
We will also define $\bvar^\ep = h_t^\ep\compose(\h^\ep)^{-1}$. With this we can define the material derivative $\Dtep = \pt + \bvar^\ep\pap$ and $\Dap^\ep = \frac{1}{\Zapep}\pap$. At this point it is also useful to note that at time $t = 0$ for any smooth solution to \eqref{eq:main}, we have  $\Mcal_1(0) = \Mcal(0)$. 

Now  for each $\ep>0$, the initial data is smooth and $\brac{\Zapep(\cdot,0) - 1, \frac{1}{\Zapep}(\cdot,0) -1 } \in H^s(\Rsp)\times H^s(\Rsp)$ for all $s \geq 2$. Hence from \corref{cor:existencesobolev}, we see that there exists a time $T_\ep$ depending on $\ep$, so that we have a smooth solution $\Zep(\cdot,t)$ to \eqref{eq:main} in the time interval $[0,T_\ep]$. Now from \eqref{eq:c0} and \eqref{eq:Mcal1} we observe that for any $0<\ep \leq 1$, we have that $c_0^{\ep} \leq c_0$ and $ \Mcal^{\ep}(0) = \Mcal_1^\ep(0) \leq \Mcal_1(0)$. Hence using the main a priori estimate \thmref{thm:apriori}, the control of the $H^2$ norm by $\Mcal$ in \lemref{lem:universalF} and the $H^2$ blow up criterion of \corref{cor:existencesobolev}, we see that in fact the solution $\Zep(\cdot,t)$ exists on a time interval $[0,T]$ independent of $\ep$ and the time $T$ depends only on $\Mcal_1(0)$. Moreover in this time interval $[0,T]$ we have for all $0<\ep \leq 1$
\begin{align*}
    \Mcal_1^\ep(t) \leq \Mcal^\ep(t)  \leq C(\Mcal_1(0))
\end{align*}
In the following steps, we will first prove uniform bounds on appropriate terms. Then, we use these uniform bounds with Lemma \ref{lem:Arzela1} and Lemma \ref{lem:Arzela2} to obtain uniform convergence arguments for the relevant terms to the system \eqref{eq:main}. Finally, this allows us to prove the existence of a solution in the sense of Definition \ref{def:solution}.

\textbf{Step 1:} First, we prove that both $\bvar^\ep$ and $\bvar_{t}^{\ep}$ are uniformly bounded in $L^{\infty}$ by the quantities $\Mcal_1(0)$ and $c_0$. Throughout this step, we will use the bounds presented in Lemma \ref{lem:universalF} and the convention that $C(c_0+\Mcal_{1}(0))$ is a constant depending only on $c_0+\Mcal_{1}(0)$.

First, note that, by Lemma \ref{lem:universalF},
\begin{align*}
\norm[\infty]{\bvar^{\ep}} &= \norm[\infty]{\Hil\sqbrac{\frac{1}{\abs*[\big]{\Zapep}^{2}}}} = \norm[\infty]{\Hil\sqbrac{\frac{1}{\abs*[\big]{\Zapep}^{2}}-1}}\leq \norm[{H}^{1}]{\frac{1}{\abs*[\big]{\Zapep}^{2}}-1}
\leq C(c_0 + \Mcal_{1}(0)).
\end{align*}
Next,
\begin{align*}
    \norm[\infty]{\bvar^{\ep}_{t}} &\lesssim \norm[\infty]{\Dtep\bvar^{\ep}} + \norm[\infty]{\bvar^\ep \bap^{\ep}}
\end{align*}
Analogous bounds to \eqref{eq:Bap1infty} and the above estimate on $\bvar^\ep$ give $\norm[\infty]{\bvar^\ep \bap^{\ep}} \leq C(c_0+\Mcal_{1}(0))$. Next,
\begin{align}\label{eq:Dtepbepdecomp}
\norm[\infty]{\Dtep\bvar^{\ep}} &\lesssim \norm[2]{\Dtep\bvar^{\ep}} + \norm[2]{\pap\Dtep\bvar^{\ep}}
\end{align}
The first term in \eqref{eq:Dtepbepdecomp} is
\begin{align*}
\norm[2]{\Dtep\bvar^{\ep}} &\lesssim \norm[2]{\sqbrac{\Dtep, \Hil}\frac{1}{\abs*[\big]{\Zapep}^{2}}} + \norm[2]{\frac{1}{\Zapepbar}\Dtep \frac{1}{\Zapep}}
\end{align*}
Computing the commutator and then by Proposition \ref{prop:commutator}, we have
\begin{align*}
\norm[2]{\sqbrac{\Dtep, \Hil}\frac{1}{\abs*[\big]{\Zapep}^{2}}} \lesssim \norm[2]{\sqbrac{\bvar^\ep, \Hil}\brac{\pap\frac{1}{\abs*[\big]{\Zapep}^2}}} \leq C(c_0 + \Mcal_{1}(0))
\end{align*}
Next, by Proposition \ref{prop:commutator}, $\norm[2]{\Bone^\ep} \lesssim \norm[\infty]{\frac{1}{\Zapep}}\norm[2]{\pap\frac{1}{\Zapep}}$
and hence
\begin{align*}
    \norm[2]{\frac{1}{\Zapepbar}\Dtep \frac{1}{\Zapep}} &\leq \norm[2]{\frac{1}{\abs*[\big]{\Zapep}^{3}}\pap \frac{1}{\Zapep}} + \norm[2]{\frac{1}{\abs*[\big]{\Zapep}^{2}}\Bone^\ep}\leq C(c_0 + \Mcal_{1}(0))
\end{align*}
The second term in \eqref{eq:Dtepbepdecomp} is
\begin{align*}
  \norm[2]{\pap\Dtep \bvar^{\ep}} &\lesssim   \norm[2]{\pap\sqbrac{\Dtep, \Hil}\frac{1}{\abs*[\big]{\Zapep}^{2}}} + \norm[2]{\pap\brac{\frac{1}{\Zapepbar}\Dt^\ep \frac{1}{\Zapep}}}
\end{align*}
By Proposition \ref{prop:commutator} and the uniform bound on $\bap^\ep$ analogous to \eqref{eq:Bap1infty}, we have
\begin{align*}
\norm[2]{\pap\sqbrac{\Dt^\ep, \Hil}\frac{1}{\abs*[\big]{\Zapep}^{2}}} &\lesssim \norm[2]{\pap\sqbrac{\bvar^\ep, \Hil}\frac{1}{\Zapepbar}\pap\frac{1}{\Zapep}} \leq C(c_0 + \Mcal_{1}(0)).
\end{align*}
Next,
\begin{align*}
\norm[2]{\pap\brac{\frac{1}{\Zapepbar}\Dtep \frac{1}{\Zapep}}} &\lesssim \norm[2]{\pap\brac{\frac{1}{\abs*[\big]{\Zapep}^{2}\Zapepbar}\pap \frac{1}{\Zapep}}} + \norm[2]{\pap\brac{\frac{1}{\abs*[\big]{\Zapep}^{2}}\Bone^\ep}}
\end{align*}
By analogous bound to \eqref{eq:DapZInfinity},
\begin{align*}
     \norm[2]{\pap\brac{\frac{1}{\abs*[\big]{\Zapep}^{2}\Zapepbar}\pap \frac{1}{\Zapep}}}  &\lesssim  \norm[2]{\frac{1}{\Zapep}\Dap^\ep \brac{\Dap^\ep \frac{1}{\Zapep}}} + \norm[2]{\frac{1}{\Zapep}\brac{\pap\frac{1}{\Zapep}}\brac{{\Dap^\ep}\frac{1}{\Zapep}}}\\
     &\leq C(c_0 + \Mcal_1(0))
\end{align*}
Finally by analogous bound to \eqref{eq:B1inftyHhalf} and \eqref{eq:DapB1} we get
\begin{align*}
\norm[2]{\pap\brac{\frac{1}{\abs*[\big]{\Zapep}^{2}}\Bone^\ep}} &\lesssim \norm[2]{\frac{1}{\Zapep}\Bone^\ep\pap\frac{1}{\Zapep}} + \norm[2]{\frac{1}{\abs*[\big]{\Zapep}^{2}}\pap\Bone^\ep} \leq C(c_0 + \Mcal_1(0))
\end{align*}
Hence,
\begin{align}\label{eq:bvarepbounds}
 \sup_{t\in\sqbrac{0,T}} \cbrac{ \norm[\infty]{\bvar^\ep}(t) + \norm[\infty]{\bvar^\ep_t}(t) + \norm[\infty]{\bap^\ep}(t)}\leq C(c_0+ \Mcal_1(0))
\end{align}

\textbf{Step 2:} In this step, we will prove uniform bounds on $\hep$, $\halep$ and $\hep_t$. First, since $\bvar^\ep$ satisfies the bound \eqref{eq:bvarepbounds} and $\bvar^\ep = \hep_t\circ (\hep)^{-1}$, we obtain that $\norm[\infty]{\hep_t}\leq C(c_0+ \Mcal_1(0))$. Next,
\begin{align*}
\norm[\infty]{\hep(\al,t) - \al} = \norm[\infty]{\int_{0}^{t}\hep_t(\al,s) \diff s} \lesssim \int_{0}^{t}\norm[\infty]{\hep_t}(s)\diff s \leq C(c_0 + \Mcal_1(0)).
\end{align*}
Finally, differentiating we obtain $\frac{d}{dt}\halep = \bap^\ep (\hep,t) \halep$ with $\halep (\al,0)= 1$. Thus, solving the initial value problem in time,
\begin{align*}
 e^{-t\sup_{s\in\sqbrac{0,t}}\norm[\infty]{\bap^\ep}(s)} \leq \halep(\al,t)\leq e^{t\sup_{s\in\sqbrac{0,t}}\norm[\infty]{\bap^\ep}(s)} .
\end{align*}
Using this above two-sided bound and by the bound on $\bap^\ep$ from \eqref{eq:bvarepbounds}, we have
 \begin{align}\label{eq:halepbound}
 0 < c_{1} \leq \frac{\hep(\al,t)-\hep(\beta,t)}{\al-\beta}\leq c_{2} < \infty \q \forall \al,\beta\in\mathbb{R} \tx{ with } \alpha \neq \beta \tx{ and } t\in\sqbrac{0,T}
 \end{align}
with $c_1, c_2$ depending only on $c_0$ and $\Mcal_1(0)$, and moreover
 \begin{align}\label{eq:hepbounds}
   \sup_{t\in\sqbrac{0,T}} \cbrac{  \norm[\infty]{\hep(\al,t) - \al}(t) +   \norm[\infty]{\hep_t}(t) + \norm[\infty]{\halep}(t) }\leq  C(c_0+ \Mcal_1(0)).
 \end{align}
 
\textbf{Step 3:} In this step, we will prove uniform bounds on $\frac{\halep}{\zalep}$ and $\zt$. From \eqref{eq:main} and \lemref{lem:universalF} we find that
\begin{align*}
\norm[\infty]{\Dt^\ep\frac{1}{\Zapep}}   \lesssim \norm[\infty]{\frac{1}{\Zapep}}\norm[\infty]{\frac{1}{\Zapep}\pap\frac{1}{\Zapep}} + \norm[\infty]{\frac{1}{\Zapep}}\norm[\infty]{\Bone^{\ep}} \leq C(c_0 + \Mcal_1(0))
\end{align*}
Also from \lemref{lem:universalF}, we clearly have
\begin{align*}
\norm[\infty]{\frac{1}{\Zapep}} + \norm[2]{\pap\frac{1}{\Zapep}} \leq C(c_0 + \Mcal_1(0))
\end{align*}
Now using the fact that $\frac{\halep}{\zalep} = \frac{1}{\Zapep}\compose \hep$ and \eqref{eq:hepbounds}, we see that
\begin{align*}
\sup_{t \in [0,T]} \cbrac{\norm[\infty]{\frac{\halep}{\zalep}}(t) + \norm[2]{\pal\frac{\halep}{\zalep}}(t) + \norm[\infty]{\pt \frac{\halep}{\zalep}}(t) }\leq C(c_0 + \Mcal_1(0))
\end{align*}
Now from Darcy's law \eqref{eq:DarcyLawA1} and the equivalence of it with the system \eqref{eq:main} we see that
\begin{align*}
\sup_{t \in [0,T]} \cbrac{\norm[\infty]{\Ztep}(t) + \norm[2]{\pap\Ztep}(t) + \norm[\infty]{\Dt\Ztep}(t)} \leq C(c_0 + \Mcal_1(0))
\end{align*}
and hence using $\ztep = \Ztep \compose \hep$ we get
\begin{align*}
\sup_{t \in [0,T]} \cbrac{\norm[\infty]{\ztep}(t) + \norm[2]{\pap\ztep}(t) + \norm[\infty]{\pt\ztep}(t)} \leq C(c_0 + \Mcal_1(0))
\end{align*}
We now have control for all of the necessary terms.

\textbf{Step 4:} From the above controlled quantities and using \lemref{lem:Arzela1}, \lemref{lem:Arzela2} we get convergence of several quantities by taking a sub-sequence $\ep_j \to 0$. By following the same proof as in \cite{Wu19} and writing $\ep_j = \ep $ as abuse of notation we obtain:
\begin{enumerate}[label=(\alph*)]
\item There exists a continuous function $h:\Rsp\times[0,T] \to \Rsp$ so that $h(\cdot,t):\Rsp \to \Rsp$ is a homeomorphism and
\begin{align*}
h^\ep \Rightarrow h \quad \tx{ and } \quad (h^\ep)^{-1} \Rightarrow h^{-1} \qq \tx{ on } \Rsp\times[0,T]
\end{align*}
Moreover there exists constants $c_1, c_2 > 0$ depending only on $c_0$ and $\Mcal_1(0) $ so that
\begin{align*}
0< c_1 \leq \frac{h(\al,t) - h(\be,t)}{\al - \be} \leq c_2 < \infty \qq \tx{ for all }  \al,\be \in \Rsp \tx{ with }\al \neq \be \tx{ and } t \in [0,T]
\end{align*}

\item There exists a continuous function $b:\Rsp\times[0,T] \to \Rsp$ such that
\begin{align*}
    \bvar^\ep \Ra \bvar \qq \tx{ on } \Rsp\times[0,T]
\end{align*}
Now as $\bvar^\ep \compose \hep = \hep_t$, we see that the function $h$ defined above is in fact continuously differentiable in $t$ satisfying 
\begin{align*}
    \hep_t \Ra h_t \qq \tx{ on } \Rsp\times[0,T]
\end{align*}
and we have $b\compose h = h_t $. 

\item There exists a continuous function $z:\Rsp\times[0,T] \to \Csp$ such that $z$ is continuously differentiable with respect to $t$, with $\zt$ being continuous and bounded function on $\Rsp\times[0,T]$ satisfying
\begin{align*}
\zep \Ra \z \qq \ztep \Ra \zt \qq \tx{ on } \Rsp\times[0,T]
\end{align*}

\item If $\Z(\ap,t) = \z(\hinv(\ap,t),t) $ and $\Zt(\ap,t) = \zt(\hinv(\ap,t),t) $, then observe that $\Z$ and $\Zt$ are continuous functions on $\Rsp\times[0,T]$ with $\Zt$ being bounded. We also have
\begin{align*}
\Zep \Ra \Z \qq \Ztep \Ra \Zt \qq \tx{ on } \Rsp\times[0,T]
\end{align*}

\item There exists a continuous function $\Psi:\Pminusbar \times [0,T] \to \Csp$ such that $\Psi(\cdot,t)$ is conformal on $\Pminus$ and $\onePsizp$ extends continuously to $\Pminusbar\times [0,T]$. Its boundary value is given by $\Z(\ap,t) = \Psi(\ap,t)$ and we also have
\begin{align*}
\Psiep \Ra \Psi \qq \onePsizpep \Ra \onePsizp \qq \tx{on } \Pminusbar \times[0,T]\\
\Psitep \Ra \Psit \qq  \Psizpep \Ra \Psizp  \qq \tx{on } \Pminus \times[0,T]
\end{align*}
As $\onePsizp$ extends continuously to $\Pminusbar\times[0,T]$, we call its boundary value as $\frac{1}{\Zap}$ by abuse of notation (as mentioned in \remref{rem:oneoverZap}).   With this we also have
\begin{align*}
    \frac{1}{\Zapep} \Ra \frac{1}{\Zap} \qq \tx{ on } \Rsp\times[0,T]
\end{align*}
The proof also shows that $\frac{\Psit}{\Psizp}$ extends continuously to $\Pminusbar\times[0,T]$ and we have
\begin{align*}
\frac{\Psitep}{\Psizpep} \Ra \frac{\Psit}{\Psizp} \qq \tx{ on } \Pminusbar\times[0,T]
\end{align*}
\end{enumerate}

\textbf{Step 5:} We now show that the function $\Psi$ obtained above actually solves the system according to  \defref{def:solution}. From the fact that $\sup_{t \in [0,T]}\Mcal_1^\ep(t) \leq C(\Mcal_1(0))$ for all $0<\ep \leq 1$, we see that $\sup_{t \in [0,T]}\Mcal_1(t) \leq C(\Mcal_1(0))$. Similarly from the convergence of $\Psiep$ to $\Psi$ and \lemref{lem:universalF} we also see that 
\begin{align*}
\sup_{\substack{y'<0\\t\in\sqbrac{0,T}}} \norm[H^1]{\frac{1}{\Psizp}(\cdot + iy',t)- 1} \leq C(c_0 + \Mcal_1(0))
\end{align*}
and hence 
\begin{align}\label{eq:boundZap}
\sup_{t \in [0,T]}\norm[H^1]{\frac{1}{\Zap}(\cdot,t) - 1} \leq C(c_0 + \Mcal_1(0))
\end{align}

From the construction, it is easy to see that most of the other conditions in the definition are automatically satisfied and the only thing left to show is that $\frac{1}{\Zap}$ satisfies the system \eqref{eq:main} in the sense of distributions. Note from the bound \eqref{eq:boundZap}, the definitions of $\bvar$ and $\Bone$ make sense. Now observe that for $\phi$, a smooth, compactly supported function in $\mathbb{R}\times \sqbrac{0,T}$, we have
\begin{align*}
& \int_{\mathbb{R}\times\sqbrac{0,T}} \phi \brac{\frac{1}{\abs*[\big]{\Zapep}^2}\pap\frac{1}{\Zapep}-\frac{1}{\Zapabs^2}\pap\frac{1}{\Zap}} \\
& = \int_{\mathbb{R}\times\sqbrac{0,T}} \phi \frac{1}{\abs*[\big]{\Zapep}^2}\brac{\pap\frac{1}{\Zapep}-\pap\frac{1}{\Zap}} 
+ \int_{\mathbb{R}\times\sqbrac{0,T}} \phi \brac{\frac{1}{\abs*[\big]{\Zapep}^2}- \frac{1}{\Zapabs^2}}\pap\frac{1}{\Zap}
\end{align*}

The second term above vanishes as $\ep\rightarrow 0$ because we have $\sup_{\sqbrac{0,T}}\norm[2]{\pap\frac{1}{\Zap}} < \infty$ and $\frac{1}{\Zapep}$ converges uniformly to $\frac{1}{\Zap}$ on compact subsets of $\mathbb{R}\times\sqbrac{0,T}$. For the first term, we integrate by parts in space
\begin{align*}
& \int_{\mathbb{R}\times\sqbrac{0,T}} \phi \frac{1}{\abs*[\big]{\Zapep}^2}\brac{\pap\frac{1}{\Zapep}-\pap\frac{1}{\Zap}} \\ 
& = -\int_{\mathbb{R}\times\sqbrac{0,T}} \cbrac{\phi \brac{\overline{\Dap^\ep}\frac{1}{\Zapep} + \Dap^\ep\frac{1}{\Zapepbar}}\brac{\frac{1}{\Zapep}-\frac{1}{\Zap}} + (\pap\phi) \frac{1}{\abs*[\big]{\Zapep}^2}\brac{\frac{1}{\Zapep}-\frac{1}{\Zap}}}
\end{align*}
Since $\Dap^\ep \frac{1}{\Zapep}$ and $\frac{1}{\Zapep}$ are uniformly bounded and $\frac{1}{\Zapep}$ converges uniformly to $\frac{1}{\Zap}$ on compact subsets of $\mathbb{R}\times\sqbrac{0,T}$, this term vanishes as $\ep\rightarrow 0$.

Next, the weak convergence of the term $\bvar^\ep\pap\frac{1}{\Zapep}$ to $\bvar\pap\frac{1}{\Zap}$ is similar to the above terms because we know $\bvar^\ep \Ra \bvar$ on $\mathbb{R}\times \sqbrac{0,T}$.

The final difference to consider is
\begin{align}\label{eq:weakBoneterm}
    \int_{\mathbb{R}\times\sqbrac{0,T}} \phi \brac{\frac{\Bone^\ep}{\Zapep}-\frac{\Bone}{\Zap}} = 
    \int_{\mathbb{R}\times\sqbrac{0,T}} \phi\Bone \brac{\frac{1}{\Zapep}-\frac{1}{\Zap}}  + 
    \int_{\mathbb{R}\times\sqbrac{0,T}} \phi \frac{1}{\Zapep} \brac{\Bone^\ep-\Bone} 
\end{align}
The first term converges as in the arguments above. The second term requires a different argument. Let $0< \ep_1, \ep_2 \leq 1$ and consider the $\ep_1$ and $\ep_2$ solutions as solutions $A$ and $B$ as used in \secref{sec:uniqueness}. Then we see that
\begin{align*}
\frac{1}{\Zap^{\ep_1}} - \frac{1}{\Zap^{\ep_2}} = \Delta\brac{\frac{1}{\Zap}} + \brac{\frac{1}{\Zap^{\ep_2}}\compose \htil -  \frac{1}{\Zap^{\ep_2}}}
\end{align*}
Hence we see that for $t \in [0,T]$ we have from \thmref{thm:mainuniqueness}, \lemref{lem:uniquenessforexistenceBone} and Proposition \ref{prop:deltabalpha}
\begin{align*}
\norm[\Hhalf]{\frac{1}{\Zap^{\ep_1}} - \frac{1}{\Zap^{\ep_2}}}(t) 
& \leq \norm[\Hhalf]{\Delta(\frac{1}{\Zap})}(t) + \norm[\Hhalf]{\frac{1}{\Zap^{\ep_2}}\compose \htil -  \frac{1}{\Zap^{\ep_2}}}(t) \\
& \lesssim_{\Mcal_1(0)}  \norm[\Linfty\cap\Hhalf]{\Delta(\frac{1}{\Zap})}(0) + \norm*[\big][2]{\htilap-1}^{\frac{1}{4}}(t) + \norm*[\big][2]{\htilap-1}(t) \\
& \lesssim_{\Mcal_1(0)} \norm[\Linfty\cap\Hhalf]{\Delta(\frac{1}{\Zap})}(0) + \norm[\Linfty\cap\Hhalf]{\Delta(\frac{1}{\Zap})}^{\frac{1}{4}}(0) \\
& \lesssim_{\Mcal_1(0)} \norm[H^1]{\Delta(\frac{1}{\Zap} - 1)}(0) + \norm[H^1]{\Delta(\frac{1}{\Zap} - 1)}^{\frac{1}{4}}(0) 
\end{align*}

Hence for any $t \in [0,T]$, by Proposition \ref{prop:commutator} and Sobolev embedding,
\begin{align*}
\norm[2]{\Bone^{\ep_1} - \Bone^{\ep_2}}(t) 
& \lesssim \norm[2]{\sqbrac{\frac{1}{\Zap^{\ep_1}} - \frac{1}{\Zap^{\ep_2}}, \Hil}\pap\frac{1}{\Zap^{\ep_1}}} + \norm[2]{\sqbrac{\frac{1}{\Zap^{\ep_2}},\Hil}\pap\brac{\frac{1}{\Zap^{\ep_1}} - \frac{1}{\Zap^{\ep_2}}}} \\
& \lesssim \norm[\Hhalf]{\frac{1}{\Zap^{\ep_1}} - \frac{1}{\Zap^{\ep_2}}}\brac{\norm[2]{\pap\frac{1}{\Zap^{\ep_1}}} + \norm[2]{\pap\frac{1}{\Zap^{\ep_2}}}} \\
& \lesssim_{\Mcal_1(0)} \norm[H^1]{\Delta(\frac{1}{\Zap} - 1)}(0) + \norm[H^1]{\Delta(\frac{1}{\Zap} - 1)}^{\frac{1}{4}}(0) 
\end{align*}
This implies that $\Bone^{\ep}(t)$ forms a Cauchy sequence and hence converges in $L^{2}$. However, we know that they converge in distribution to $\Bone(t)$. Hence we see that 
\begin{align*}
\sup_{t \in [0,T]} \norm[2]{\Bone^{\ep} - \Bone}(t) \to 0 \qq \tx{ as } \ep \to 0
\end{align*}
Hence, the second term in \eqref{eq:weakBoneterm}  vanishes as $\epsilon\rightarrow 0$. Collecting all of the above terms, we have that
\begin{align*}
(\pt + \bvar^\ep\pap)\frac{1}{\Zapep} & =  - i\frac{1}{\abs*[\big]{\Zapep}^2}\pap\frac{1}{\Zapep} + \frac{\Bone^\ep}{\Zapep}
\end{align*}
converges weakly to \eqref{eq:main}.

The uniqueness of the solution in the class $\mathcal{SA}$ follows from \thmref{thm:mainuniqueness} and \lemref{lem:uniquenessforexistenceBone}. Hence the theorem is proved.

\end{proof}

\subsection{Rigidity}

In this section, we give a brief outline of the proof of Theorem \ref{thm:mainangle}, mirroring the method in \cite{Ag20}. The proof is essentially identical because the proof in \cite{Ag20} relies on the energies and not so much on the details of the system being solved.  We omit the computational or technical details and modify the energies where appropriate. First, analogously to \eqref{eq:hapzapratio}, we have
\begin{align}\label{eq:zapep}
 \frac{h_\al^\ep}{z_\al^\ep} (\al,t) =  \frac{h_\al^\ep}{z_\al^\ep} (\al,0) \exp\cbrac{\int_0^t \brac{ \frac{h_{t\al}^\ep}{h_\al^\ep}  -  \frac{z_{t\al}^\ep}{z_\al^\ep} }(\al,s) \diff s}.
\end{align}
As $ \frac{h_{t\al}^\ep}{h_\al^\ep} \compose (h^{\ep})^{-1} = \bvarap^\ep$ and $\frac{z_{t\al}^\ep}{z_\al^\ep} \compose (h^{\ep})^{-1} = \Dap^\ep \Zt^\ep$, we see from \eqref{eq:bvarepbounds} and \eqref{eq:DapZInfinity} that the integrand in \eqref{eq:zapep} is bounded in $L^{\infty}$. Moreover,  $\frac{\halep}{\zalep} = \oneZapep\compose \hep$, $\hep(\al,0) = \al$ and $\oneZapep(\ap,t) = \onePsizpep(\ap,t)$, so letting $\ep \rightarrow 0$, gives
\begin{align*}
c_1\abs{\onePsizp(\al,0)} \leq \abs{\onePsizp(\h(\al,t),t)} \leq c_2\abs{\onePsizp(\al,0)}
\end{align*}
for some constants $c_1, c_2 > 0$ depending only on $\Mcal_1(0)$. Thus, for $t\in\sqbrac{0,T}$, the singular set is propagated,  $S(t) = \cbrac{h(\al,t) \in \Rsp \suchthat \al \in S(0) }$.

Next, since solutions satisfy
\begin{align*}
\sup_{\substack{y'<0\\t\in\sqbrac{0,T}}} \norm[H^1]{\frac{1}{\Psizp}(\cdot + iy',t)- 1} < \infty,
\end{align*} and since the energy $\Mcal_{1}(t) < \infty$, it follows as in the argument in \cite{Ag20} that for each $t\in [0,T]$ 
\begin{align}\label{eq:rigidLinfinitybound}
\sup_{\yp<0}\norm[\Linfty(\Rsp,\diff \xp)]{\frac{1}{\Psizp}\partial_{\zp}\frac{1}{\Psizp}(\xp+i\yp,t)} \leq C(\Mcal_1(t)).
\end{align}
Hence $\frac{1}{\Psizp}\partial_{\zp}\frac{1}{\Psizp}(\cdot,t)$ is a bounded holomorphic function and following the argument directly in \cite{Ag20}, we obtain that $\frac{1}{\Psizp}\partial_{\zp}\frac{1}{\Psizp}(\cdot,t)$ extends continuously to $\Pminusbar$. Moreover, \eqref{eq:rigidLinfinitybound} implies that around any singular point $\ap\in S(t)$, we have for any $\delta > 0$
\begin{align*}
\lim_{\yp \to 0^-}\int_{\al'-\delta}^{\al'+\delta}\abs{\Psizp}^2(s + iy',t)\diff s = \infty 
\end{align*}
Hence following the same argument as in \cite{Ag20} we see that $\frac{1}{\Psizp}\partial_{\zp}\frac{1}{\Psizp}(\ap,t) = 0$ at any singular point $\ap\in S(t)$. The proof for $\frac{1}{\Psizp}\overline{\partial_{\zp}\frac{1}{\Psizp}}(\ap,t)$ is similar.

The identity \eqref{eq:rigididentity} for non-singular points follows from \eqref{eq:zapep} by again using $\frac{\halep}{\zalep} = \oneZapep\compose \hep$, $\hep(\al,0) = \al$, $\oneZapep(\ap,t) = \onePsizpep(\ap,t)$ and $\frac{\Ztapep}{\Zapep}(\ap,t) = \brac{\frac{1}{\Psizpep}\overline{\partial_{z}\frac{1}{\Psizpep}}}(\ap,t)$ and taking limits $\ep\rightarrow 0$ by using the convergence given by the existence proof and the dominated convergence theorem. For any singular point $\al_{0} \in S(0)$, taking the limit along non-singular points $\al_{n}\in NS(0)$, since we know the continuity of $\frac{1}{\Psizp}\overline{\partial_{\zp}\frac{1}{\Psizp}}(\zp,t)$ on $\Pminusbar$, we can apply the dominated convergence theorem to obtain
\begin{align}\label{eq:NSidentity2}
\lim_{\alpha_{n} \rightarrow \alpha_{0}}\frac{\frac{\Zap}{\Zapabs}(h(\al_{n},t),t)}{\frac{\Zap}{\Zapabs}(\al_{n},0)} = \exp\cbrac{i\Imag \brac{\int_0^t \brac{\frac{1}{\Psizp}\overline{\partial_{\zp}\frac{1}{\Psizp}}}(h(\al_{0},s),s)\diff s}}.
\end{align}
Since $\frac{1}{\Psizp}\overline{\partial_{\zp}\frac{1}{\Psizp}}(\ap,t) = 0$ for all $\ap \in S(t)$ and the singular set $S(0)$ propagates along the flow $h$, we have $\brac{\frac{1}{\Psizp}\overline{\partial_{\zp}\frac{1}{\Psizp}}}(h(\al_{0},s),s)=0$  in \eqref{eq:NSidentity2} and we obtain \eqref{eq:rigidlimit}.

Finally, we can show the velocity of the particle at the tip of the corner or cusp is $-i$. We have $\frac{1}{\Zapep}$ and $\Ztep$ satisfying \eqref{eq:DarcyLawA1} by the discussion after \eqref{eq:main}. Now by the existence argument, $\frac{1}{\Zapep}$ and $\Ztep$ converge uniformly on compact subsets of $\mathbb{R}\times\sqbrac{0,T}$. Hence, $\frac{1}{\Zap}$ and $\Zt$ satisfy \eqref{eq:DarcyLawA1}. Since at singular points, $\frac{1}{\Zap} = 0$, we have that $\Zt = -i$ which completes the proof.

\section{Appendix}
In this appendix, we state several useful estimates that are used in this paper. Several of them are exact statements or slight modifications from \cite{Ag20a} and some we must give new proofs here.

\begin{lem}\label{lem:papabs}
For $f \in \Scalsp$ we have
\begin{align*}
    (\papabs f )(\ap) = \frac{1}{\pi}\int \frac{f(\ap) - f(\bp)}{(\ap - \bp)^2}\diff \bp
\end{align*}
\end{lem}
\begin{proof}
As $\papabs = i\Hil \pap$ and as $\Hil(1) = 0$, we have
\begin{align*}
    (\papabs f)(\ap) & = i\pap\brac{\frac{1}{i\pi} \int \frac{1}{\ap - \bp}f(\bp) \diff \bp} \\
    & = -\frac{1}{\pi}\pap \int \frac{f(\ap) - f(\bp)}{\ap - \bp} \diff \bp \\
    & = \frac{1}{\pi}\int \frac{f(\ap) - f(\bp)}{(\ap - \bp)^2}\diff \bp
\end{align*}
\end{proof}

For functions $f_1,f_2,f_3:\Rsp \to \Csp$ we define the function $\sqbrac{f_1,f_2 ; f_3}:\Rsp \to \Csp$ as 
\begin{align}\label{eq:foneftwofthree}
\sqbrac{f_1, f_2;  f_3}(\ap) = \frac{1}{i\pi} \int \brac{\frac{f_1(\ap) - f_1(\bp)}{\ap - \bp}}\brac{\frac{f_2(\ap) - f_2(\bp)}{\ap-\bp}} f_3(\bp) \diff \bp
\end{align}

\begin{prop}\label{prop:tripleidentity}
Let $f,g,h \in \mathcal{S}(\Rsp)$. Then we have the following identities
\begin{enumerate}
\item $h\pap[f,\Hil]\pap g = \sqbrac{h\pap f,\Hil}\pap g + \sqbrac{f, \Hil}\pap\brac{h\pap g} - \sqbrac{h, f ; \pap g} $ 

\item $\Dt [f,\Hil]\pap g = \sqbrac{\Dt f, \Hil}\pap g + \sqbrac{f, \Hil}\pap(\Dt g) - \sqbrac{\bvar, f; \pap g} $  
\end{enumerate}
\end{prop}
\begin{proof}
Both of these identities are proven in the Appendix section of \cite{Ag20a}.
\end{proof}

\begin{prop}\label{prop:Coifman}
Let $H \in C^1(\Rsp),A_i \in C^1(\Rsp) $ for $i=1,\cdots m$ and let $\delta>0$ be such that 
\begin{align*}
\delta \leq \abs{\frac{H(x)-H(y)}{x-y} } \leq \frac{1}{\delta} \quad \tx{ for all } x\neq y
\end{align*}
Let $0\leq k \leq m+1$ and define
\begin{align*}
T(A,f)(x) & = p.v. \int \frac{\Pi_{i=1}^{m}(A_i(x) - A_i(y))}{(x-y)^{m+1-k}(H(x)-H(y))^k }f(y)\diff y
\end{align*}
then we have the estimates
\begin{enumerate}
\item $\norm[2]{T(A,f)} \leq C(\norm[\infty]{H'},\delta) \norm[\infty]{A_1'}\cdots\norm[\infty]{A_m'}\norm[2]{f}$ 
\item $\norm[2]{T(A,f)} \leq C(\norm[\infty]{H'},\delta) \norm[2]{A_1'}\norm[\infty]{A_2'}\cdots\norm[\infty]{A_m'}\norm[\infty]{f}$
\end{enumerate}

\end{prop}
\begin{proof}
Both estimates follow as special case of Proposition 9.2 in \cite{Ag20a}.
\end{proof}

\begin{prop}\label{prop:Hardy}
Let $f \in \Scalsp(\Rsp)$. Then we have
\begin{enumerate}

\item $\norm[\infty]{f} \lesssim \norm[H^s]{f}$ if $s>\frac{1}{2}$ and for $s=\half$ we have $\norm[BMO]{f} \lesssim \norm[\Hhalf]{f}$

\item $
\begin{aligned}[t]
\int \abs{\frac{f(\ap) - f(\bp)}{\ap - \bp}}^2 \diff \bp \lesssim \norm[2]{f'}^2 
\end{aligned}
$

\item $
\begin{aligned}[t]
\norm[\Ltwo(\Rsp, \diff \ap)]{\sup_{\bp}\abs{\frac{f(\ap) - f(\bp)}{\ap - \bp}}} \lesssim \norm[2]{f'}
\end{aligned}
$

\item $
\begin{aligned}[t]
\norm[\Hhalf]{f}^2 = \frac{1}{2\pi}\int\!\! \!\int \abs{\frac{f(\ap) - f(\bp)}{\ap - \bp}}^2 \diff \bp \diff\ap
\end{aligned}
$

\item $
\begin{aligned}[t]
\norm[\Ltwo(\Rsp^2, \diff\ap\diff\bp)]{\pbp\brac{\frac{f(\ap) - f(\bp)}{\ap - \bp}} } \lesssim \norm[\Hhalf]{f'}
\end{aligned}
$
\end{enumerate}
\end{prop}

\begin{proof}
This proposition is directly found in the Appendix section of \cite{Ag20a}.
\end{proof}

\begin{prop}\label{prop:commutator}
Let $f,g \in \mathcal{S}(\Rsp)$ with $s,a\in \Rsp$ and $m,n \in \Zsp$. Then we have the following estimates
\begin{enumerate}
\item $\norm*[\big][2]{\papabs^s\sqbrac{f,\Hil}(\papabs^{a} g )} \lesssim  \norm*[\big][BMO]{\papabs^{s+a}f}\norm[2]{g}$ \quad  for $s,a \geq 0$

\item $\norm*[\big][2]{\papabs^s\sqbrac{f,\Hil}(\papabs^{a} g )} \lesssim  \norm*[\big][2]{\papabs^{s+a}f}\norm[BMO]{g}$ \quad  for $s\geq 0$ and $a>0$

\item $\norm*[\big][2]{\sqbrac*[\big]{f,\papabs^\half}g } \lesssim \norm*[\big][BMO]{\papabs^\half f}\norm[2]{g}$

\item $\norm*[\big][2]{\sqbrac*[\big]{f,\papabs^\half}(\papabs^\half g) } \lesssim \norm*[\big][BMO]{\papabs f}\norm[2]{g}$

\item $\norm[\Linfty\cap\Hhalf]{\pap^m\sqbrac{f,\Hil}\pap^n g} \lesssim \norm*[\big][2]{\pap^{(m+n+1)}f}\norm[2]{g}$ \quad  for $m,n \geq 0$

\item $\norm[2]{\partial_{\ap}^m\sqbrac{f,\Hil}\partial_{\ap}^n g} \lesssim \norm*[\infty]{\partial_\ap^{(m+n)} f}\norm[2]{g}$ \quad  for $m,n \geq 0$

\item $\norm[2]{\partial_{\ap}^m\sqbrac{f,\Hil}\partial_{\ap}^n g} \lesssim \norm*[2]{\partial_\ap^{(m+n)} f}\norm[\infty]{g}$ \quad for $m\geq 0$ and $n\geq 1$

\item $\norm[2]{\sqbrac{f,\Hil}g} \lesssim \norm[2]{f'}\norm[1]{g}$
\end{enumerate}
\end{prop}
\begin{proof}
The above estimates are shown in the Appendix section of \cite{Ag20a}.
\end{proof}

\begin{prop}\label{prop:Leibniz}
Let $f,g,h \in \mathcal{S}(\Rsp)$ with $s,a\in \Rsp$ and $m,n \in \Zsp$. Then we have the following estimates
\begin{enumerate}
\item $\norm[2]{\papabs^s (fg)} \lesssim  \norm[2]{\papabs^s f}\norm[\infty]{g} + \norm[\infty]{f}\norm[2]{\papabs^s g}$ \quad for $s > 0$
\item $\norm[\Hhalf]{fg} \lesssim \norm[\Hhalf]{f}\norm[\infty]{g} + \norm[\infty]{f}\norm[\Hhalf]{g}$
\item $\norm[\Hhalf]{fg} \lesssim \norm[2]{f'}\norm[2]{g} + \norm[\infty]{f}\norm[\Hhalf]{g}$
\end{enumerate}
\end{prop}
\begin{proof}
The above estimates are shown in the Appendix section of \cite{Ag20a}.
\end{proof}

\begin{lemma}\label{lem:interpolationmult}
Let $k,n \in \Nsp$ and $f_1,f_2, \cdots, f_k \in \Scalsp(\Rsp)$. Let $r_1,r_2\cdots, r_k \in \Zsp$ with $r_1 + \cdots + r_k = n$ and $ r_i\geq 0$ for all $1\leq i\leq k$ and. Let $r = \max\cbrac{r_1,r_2,\cdots, r_k} \geq 1$. Then 
\begin{enumerate}
\item $\norm*[\big][2]{f_1^{(r_1)}\cdots f_k^{(r_k)} } \leq C(K)\cbrac{\norm[H^s]{f_1'} + \cdots + \norm[H^s]{f_k'} }$ \quad for $s = \max\cbrac{r-1,n-2}$
\item $\norm*[\big][\Hhalf]{f_1^{(r_1)}\cdots f_k^{(r_k)} } \leq C(K)\cbrac{\norm[H^s]{f_1'} + \cdots + \norm[H^s]{f_k'} }$  \quad for $s = \max\cbrac{r-\half,n-2}$
\end{enumerate}
with $K = (\norm[\infty]{f_1} + \norm[H^1]{f_1'}) + \cdots +  (\norm[\infty]{f_k} + \norm[H^1]{f_k'})$ and $C(K)$ is a constant depending only on $K$.
\end{lemma}
\begin{proof}
These estimates are shown in the Appendix section of \cite{Ag20a}.
\end{proof}

\begin{lem}[Gronwall's Inequality]\label{lem:Gronwall}
Say $v(t) \geq 0$, $w(t)$ non-decreasing and $u(t)$ satisfies
\begin{align*}
u(t) \leq w(t) + \int_{0}^{t}u(s)v(s) \diff s.
\end{align*}
Then, $u(t)$ satisfies the bound
\begin{align*}
u(t) \leq w(t)\exp\brac{\int_{0}^{t}v(s)\diff s}.
\end{align*}
\end{lem}
\begin{proof}
The above inequality follows as an easy corollary to the standard integral version of Gronwall's inequality, e.g. as found in \cite{Dr03}.
\end{proof}

\begin{prop}\label{prop:chainrule}
Let $g\in\mathcal{S}(\Rsp)$. Then
\begin{align*}
    \norm[\Hhalf]{e^g - 1} \lesssim_{\norm[\infty]{g}} \norm[\Hhalf]{g}
\end{align*}
\end{prop}

\begin{proof}
We first observe that from repeated application of \propref{prop:Leibniz}, there exists a universal constant $C$ such that for any $n \in \Nsp$ we have
\begin{align*}
    \norm[\Hhalf]{g^n} \lesssim nC^{n-1}\norm[\infty]{g}^{n-1}\norm[\Hhalf]{g} 
\end{align*}
Now since $g\in L^{\infty}$,  we can use the Taylor expansion of $e^{x}-1$ to obtain
\begin{align*}
\norm[\Hhalf]{e^{g}-1} &\leq \sum_{n=1}^{\infty}\frac{1}{n!}\norm[\Hhalf]{g^{n}} \\
& \lesssim \sum_{n=1}^{\infty}\frac{1}{(n-1)!}C^{n-1}\norm[\infty]{g}^{n-1}\norm[\Hhalf]{g} \\
& \lesssim e^{C\norm[\infty]{g}}\norm[\Hhalf]{g}
\end{align*}
thereby proving the lemma. 
\end{proof}
 
\begin{prop}\label{prop:triple}
Let $f,g,h \in \mathcal{S}(\Rsp)$ . Then we have the following estimates
\begin{enumerate}

\item $\norm[2]{\sqbrac{f,g;h}} \lesssim \norm[2]{f'}\norm[2]{g'}\norm[2]{h}$

\item $\norm[2]{\pap\sqbrac{f,\sqbrac{g,\Hil}}h} \lesssim \norm[2]{f'}\norm[2]{g'}\norm[2]{h}$

\item $\norm[2]{\sqbrac{f,g; h'}} \lesssim \norm[\infty]{f'}\norm[\infty]{g'}\norm[2]{h}$

\item $\norm[\Hhalf]{\sqbrac{f,g; h'}} \lesssim \norm[\infty]{f'}\norm[\infty]{g'}\norm[\Hhalf]{h}$

\item $\norm[\Linfty\cap\Hhalf]{\sqbrac{f,g;h}} \lesssim \norm[\infty]{f'}\norm[2]{g'}\norm[2]{h}$

\end{enumerate}
\end{prop}
\begin{proof}
These estimates are shown in the Appendix section of \cite{Ag20a}.
\end{proof}

\begin{prop} \label{prop:LinftyHhalf}
Let $f \in \mathcal{S}(\Rsp)$ and let $w$ be a smooth non-zero weight with $w,\frac{1}{w} \in \Linfty(\Rsp) $ and $w' \in \Ltwo(\Rsp)$. Then 
\begin{enumerate}
\item $\norm[\infty]{f}^2 \lesssim \norm[2]{\frac{f}{w}}\norm[2]{w(f')}$
\item $\norm[\Linfty\cap\Hhalf]{f}^2 \lesssim \norm[2]{\frac{f}{w}}\norm[2]{(wf)'} +  \norm[2]{\frac{f}{w}}^2\norm[2]{w'}^2$
\end{enumerate}
\end{prop}
\begin{proof}
These inequalities are shown in the Appendix section of \cite{Ag20a}.
\end{proof}

\begin{prop}\label{prop:Hhalfweight}
Let $f,g \in \mathcal{S}(\Rsp)$ and let $w,h \in \Linfty(\Rsp)$ be smooth functions with $w',h' \in \Ltwo(\Rsp)$. Then 
\begin{align*}
\norm[\Hhalf]{fwh} \lesssim \norm[\Hhalf]{fw}\norm[\infty]{h} + \norm[2]{f}\norm[2]{(wh)'} + \norm[2]{f}\norm[2]{w'}\norm[\infty]{h}
\end{align*}
If in addition we assume that $w$ is real valued then
\begin{align*}
\norm[2]{fgw} \lesssim \norm[\Hhalf]{fw}\norm[2]{g} + \norm[\Hhalf]{gw}\norm[2]{f} + \norm[2]{f}\norm[2]{g}\norm[2]{w'} 
\end{align*}
\end{prop}
\begin{proof}
These estimates are shown in the Appendix section of \cite{Ag20a}.
\end{proof}
Recall from \eqref{eq:Hcal} that
\begin{align*}
(\Hcal f)(\ap) &= \frac{1}{i\pi} p.v. \int \frac{\htilbp(\bp)}{\htil(\ap) - \htil(\bp)}f(\bp) \diff \bp \\
 (\Hcaltil f)(\ap) &=\frac{1}{i\pi} p.v. \int \frac{1}{\htil(\ap) - \htil(\bp)}f(\bp) \diff \bp
\end{align*}
We now prove some estimates related to these operators.

\begin{prop}\label{prop:HilHtilcaldiff}
Let $f,f_1,f_2,f_3,g \in \Scalsp(\Rsp)$ and let $\htil:\Rsp \to \Rsp$ be a homeomorphism such that there exists a constant $L\geq 1$ for which
\begin{align*}
\frac{1}{L} \leq \abs{\frac{\htil(x) - \htil(y)}{x-y}} \leq L \quad \tx{ for all } x\neq y
\end{align*}
Let $\Hil$ be the Hilbert transform \eqref{eq:HilbertP} and let $\Hcal, \Htil$ be as defined as in \eqref{eq:Hcal}. We will suppress the dependence of $L$ i.e. we write $a\lesssim b$ instead of $a\leq C(L)b$. With this notation we have the following estimates:
\begin{enumerate}
\item $\norm[2]{\Hcal(f)} \lesssim \norm[2]{f}$ and $\norm*[2]{\Htil(f)} \lesssim \norm[2]{f}$. 
\item  $\norm[\Hhalf]{\Hcal(f)} \lesssim  \norm[\Hhalf]{f}$
\item $\norm[2]{(\Hil - \Hcal) f} \lesssim \norm*[\big][2]{\htilap -1}\norm[\infty]{f}$ 
\item $\norm*[2]{\sqbrac*{f,\Hil - \Htil}g} \lesssim \norm[2]{f'}\norm*[\big][2]{\htilap - 1}\norm[2]{g}$
\item  $\norm*[2]{\sqbrac*{f,\Hil - \Hcal}g} \lesssim \norm[2]{f'}\norm*[\big][2]{\htilap - 1}\norm[2]{g}$
\end{enumerate}
\end{prop} 
\begin{proof}
To simplify the calculations we define
\[ 
\begin{array}{*2{>{\displaystyle}c}}
\lpar \qquad \dis F(a,b) = \frac{f(a)-f(b)}{a-b}  & \quad  F_{h}(a,b) = \frac{f(a)-f(b)}{\htil(a)-\htil(b)} \\
H(a,b) = \frac{(\htil(a)-a)-(\htil(b)-b)}{a-b} & \quad  H_h(a,b) = \frac{(\htil(a)-a)-(\htil(b)-b)}{\htil(a)-\htil(b)} 
\end{array}
\]

\begin{enumerate}[leftmargin =*, align=left]
    \item This has already been shown as part of \lemref{lem:quantM}. 
    \item This has already been shown as part of \lemref{lem:quantM}. 
    \item We write $(\Hil - \Hcal) = (\Hil - \Htil) + (\Htil - \Hcal)$. Observe that
\begin{align*}
(\Htil - \Hcal) f = \Htil ((1-\htilap)f)
\end{align*}
Hence as $\Htil $ is bounded on $\Ltwo$ we have $\norm*[2]{(\Htil-\Hcal)f} \lesssim \norm*[2]{\htilap-1}\norm[\infty]{f} $. Now we have
\begin{align*}
((\Hil - \Htil)f)(\ap) = \frac{1}{i\pi} \int \brac{\frac{1}{\ap-\bp} - \frac{1}{\htil(\ap)-\htil(\bp)} }f(\bp) \diff\bp = \frac{1}{i\pi} \int \frac{H_h(\ap,\bp)}{\ap-\bp} f(\bp) \diff \bp
\end{align*} 
Now using  \propref{prop:Coifman} we see that $\norm*[2]{(\Hil-\Htil)f} \lesssim \norm*[2]{\htil'-1}\norm[\infty]{f} $. Hence the required estimate follows.
\item Note that
\begin{align*}
(\sqbrac*{f,\Hil - \Htil}g)(\ap) = \frac{1}{i\pi} \int F(\ap,\bp)H_h(\ap,\bp) g(\bp) \diff \bp
\end{align*}
and hence by Cauchy Schwarz inequality and Hardy's inequality \propref{prop:Hardy} we have
\begin{align*}
\abs*{\sqbrac*{f,\Hil - \Htil}g}(\ap) \lesssim \norm*[2]{\htilap -1}\brac{\int \abs{F(\ap,\bp)}^2\abs{g(\bp)}^2 \diff \bp}^\half
\end{align*}
The estimate now follows from another application of Hardy's inequality  \propref{prop:Hardy}.
\item We have
\begin{align}
\sqbrac{f, \Htil - \Hcal}g &= f\brac{\Htil - \Hcal}g - \brac{\Htil - \Hcal}\brac{fg} \nonumber\\
&= f\Htil\brac{\brac{1-\htilap}g} - \Htil\brac{\brac{1-\htilap}\brac{fg}} \nonumber\\
&= \sqbrac{f, \Htil}\brac{\brac{1-\htilap}g} \nonumber\\
&=-\sqbrac{f, \Hil}\brac{\brac{1-\htilap}g} -\sqbrac{f, \Hil - \Htil}\brac{\brac{1-\htilap}g}. \label{eq:lastline610}
\end{align}
For the first term in \eqref{eq:lastline610}, we use the final inequality in Proposition \ref{prop:commutator} and for the second term, we use the previous inequality of Proposition \ref{prop:HilHtilcaldiff}. This concludes the proof of this estimate.
\end{enumerate}

\end{proof}

\bibliographystyle{amsplain}
\bibliography{References.bib}

\end{document}